\documentclass[12pt,oneside]{amsart}
\usepackage{amscd,amssymb}

\usepackage{graphicx, amsmath, amssymb, amscd, amsthm, euscript, psfrag, amsfonts,bm}

\DeclareMathAlphabet{\mathpzc}{OT1}{pzc}{m}{it}

\newcommand{\dd}{\mathrm{d}}

\newtheorem{theorem}{Theorem}[section]

\newtheorem{fact}[theorem]{Fact}

\newtheorem{dfn}[theorem]{Definition}
\newcommand{\BMS}{\operatorname{BMS}}
\newcommand{\Lip}{\operatorname{Lip}}

\newenvironment{definition}[1][Definition]{\begin{trivlist}
\item[\hskip \labelsep {\bfseries #1}]}{\end{trivlist}}

\newenvironment{remark}[1][Remark]{\begin{trivlist}
\item[\hskip \labelsep {\bfseries #1}]}{\end{trivlist}}

\usepackage{amssymb}

\DeclareMathAlphabet{\mathpzc}{OT1}{pzc}{m}{it}

\newtheorem{thm}[theorem]{Theorem}\newtheorem{Thm}[theorem]{Theorem}

\newtheorem{rmk}[theorem]{Remark}
\newtheorem{Cor}[theorem]{Corollary}

\newtheorem{prop}[theorem]{Proposition}
\newtheorem{Nota}[theorem]{Notation}

\newtheorem{cor}[theorem]{Corollary}
\newtheorem{lem}[theorem]{Lemma}\newtheorem{Lem}[theorem]{Lemma}

\numberwithin{equation}{section}

\numberwithin{equation}{section}

\newcommand{\be}{begin{equation}}

\newcommand{\bH}{\mathbb H}
\newcommand{\LG}{\Lambda(\Gamma)}
\newcommand{\tS}{\tilde S}
\newcommand{\Euc}{\op{Euc}}
\newcommand{\bfv}{\mathbf{v}}

\newcommand{\q}{\mathbb{Q}}

\newcommand{\e}{{\epsilon}}
\newcommand{\z}{\mathbb{Z}}
\renewcommand{\q}{\mathbb{Q}}
\newcommand{\N}{\mathbb{N}}
\renewcommand{\c}{\mathbb{C}}
\newcommand{\br}{\mathbb{R}}
\newcommand{\bbr}{\mathbb{R}}
\newcommand{\R}{\mathbb{R}}
\newcommand{\bz}{\mathbb{Z}}
\newcommand{\gfrak}{\mathfrak{g}}
\newcommand{\tH}{\tilde{H}}

\newcommand{\gt}{{\Cal G}^t}
\newcommand{\gr}{{\Cal G}^r}
\newcommand{{\grinv}}{{\Cal G}^{-r}}
\newcommand{\qp}{\mathbb{Q}_p}
\newcommand{\bqp}{\mathbf G (\mathbb{Q}_p)}
\newcommand{\bzp}{\mathbf G (\mathbb{Z}_p)}
\newcommand{\zp}{\mathbb{Z}_p}
\newcommand{\A}{\mathbb{A}}\newcommand{\p}{\mathbf{p}}
\newcommand{\ba}{\backslash}\newcommand{\bs}{\backslash}

\newcommand{\G}{\Gamma}

\newcommand{\g}{\gamma}\renewcommand{\t}{\tilde}

\newcommand{\Zcl}{\operatorname{Zcl}}\newcommand{\Haar}{\operatorname{Haar}}

\newcommand{\Cal}{\mathcal}
\renewcommand{\P}{\mathcal P}
\newcommand{\bG}{\mathbf G}
\newcommand{\la}{\langle}
\newcommand{\ra}{\rangle}
\newcommand{\pr}{\rm{pr}}
\newcommand{\bM}{\mathbf M}
\newcommand{\SL}{\operatorname{SL}}
\newcommand{\GL}{\operatorname{GL}}
\newcommand{\bp}{\begin{pmatrix}}
\newcommand{\ep}{\end{pmatrix}}
\newcommand{\Int}{\operatorname{Int}}

\renewcommand{\bp}{{\rm bp}}

\newcommand{\inv}{^{-1}}
\newcommand{\cC}{{\mathcal{C}}}\newcommand{\cH}{{\mathcal{H}}}

\newcommand{\Spin}{\operatorname{Spin}}
\newcommand{\PGL}{\operatorname{PGL}}
\renewcommand{\O}{\operatorname{O}}
\newcommand{\SO}{\operatorname{SO}}
\newcommand{\SU}{\operatorname{SU}}
\newcommand{\Sp}{\operatorname{Sp}}
\renewcommand{\H}{\mathcal{H}}
\newcommand{\I}{\operatorname{I}}
\newcommand{\T}{\operatorname{T}}
\newcommand{\Th}{\operatorname{T}^1(\bH^n)}
\newcommand{\bh}{\partial(\mathbb{H}^n)}
\newcommand{\vol}{\operatorname{vol}}
\newcommand{\s}{\sigma}
\newcommand{\lm}{\lambda_m}
\newcommand{\bc}{\mathbb C}
\newcommand{\bL}{\mathbf L}

\newcommand{\sk}{\operatorname{sk}}
\newcommand{\bT}{\mathbf T}
\newcommand{\bhi}{G_{w_f}}
\renewcommand{\L}{\Cal L}
\newcommand{\Vol}{\op{Vol}}
\newcommand{\PSL}{\op{PSL}}

\newcommand{\ok}{\mathcal O_k}

\newcommand{\tp}{\hat\psi}
\newcommand{\om}{\Cal O_{\A}^m}
\newcommand{\omf}{\Cal O_{\A_f}^m}
\newcommand{\cB}{{\mathcal B}}
\newcommand{\cD}{{\mathcal D}}
\newcommand{\cS}{{\mathcal S}}
\newcommand{\tO}{{\tilde\Omega}}
\newcommand{\cF}{{\mathcal F}}
\newcommand{\ncal}{\mathcal{N}}
\newcommand{\mcal}{\mathcal{M}}
\newcommand{\vcal}{\mathcal{V}}
\newcommand{\ecal}{\mathcal{E}}
\newcommand{\bcal}{\mathcal{B}}

\newcommand{\apz}{\mathpzc{a}}
\newcommand{\bpz}{\mathpzc{b}}
\newcommand{\cpz}{\mathpzc{c}}
\newcommand{\dpz}{\mathpzc{d}}
\newcommand{\hpz}{\mathpzc{h}}
\newcommand{\npz}{\mathpzc{n}}
\newcommand{\mpz}{\mathpzc{m}}
\newcommand{\ppz}{\mathpzc{p}}
\newcommand{\qpz}{\mathpzc{q}}

\newcommand{\PS}{\rm{PS}}
\newcommand{\muPS}{\mu^{\PS}}

\newcommand{\norm}[1]{\lVert #1 \rVert}
\newcommand{\abs}[1]{\lvert #1 \rvert}

\newcommand{\op}{\operatorname}\newcommand{\supp}{\operatorname{supp}}
\newcommand{\ii}{\item}
\renewcommand{\lll}{L^2(\bG(\q)\backslash \bG(\A))}
\renewcommand{\deg}{\text{DEP}}
\newcommand{\vs}{\vskip 5pt}
\newcommand{\bga}{\bG(\A)}
\newcommand{\hz}{\Cal Z_\iota(s)}
\newcommand{\He}{\operatorname{H}}
\newcommand{\tmg}{{\tilde m_\gamma}}
\newcommand{\tkg}{{\tilde k_\gamma}}
\newcommand{\tilag}{{\tilde a_\gamma}}\newcommand{\tbg}{{\tilde b_\gamma}}
\newcommand{\codim}{\operatorname{codim}}

\newcommand{\BR}{\operatorname{BR}}
\newcommand{\Leb}{\operatorname{Leb}}
\newcommand{\Hn}{{\mathbb H}^n}
\newcommand{\rank}{\operatorname{rank}}
\newcommand{\corank}{\operatorname{corank}}
\newcommand{\parcorank}{\operatorname{pb-corank}}
\newcommand{\diam}{\operatorname{diam}}
\newcommand{\Isom}{\operatorname{Isom}}
\newcommand{\bN}{\mathbf N}
\newcommand{\tE}{\tilde E}
\newcommand{\cl}[1]{\overline{#1}}
\renewcommand{\muPS}{\mu^{\PS}}
\renewcommand{\setminus}{-}
\newcommand{\Lie}{{\rm Lie}}

\newcommand{\Span}{\operatorname{span}}
\renewcommand{\be}{\begin{equation}}
\newcommand{\ee}{\end{equation}}

\newcommand{\Ad}{\operatorname{Ad}}
\newcommand{\Mob}{\operatorname{Mob}(\hat \c)}
\newcommand{\fG}{\mathfrak B}
\newcommand{\gri}{{\mathcal G}^{r_i}}
\newcommand{\gro}{{\mathcal G}^{r_0}}

\newcommand{\bg}{{\bm{\g}}}
\newcommand{\bk}{{\bm{k}}}
\newcommand{\bss}{{\bm{s}}}
\newcommand{\bw}{{\bm{w}}}
\newcommand{\Z}{\z}
\newcommand{\B}{B}
\newcommand{\cE}{{\mathcal E}}
\newcommand{\Res}{\op{Res}}
\newcommand{\psig}{\psi_\gamma}
\newcommand{\univ}{\mathbb{E}}
\newcommand{\tuniv}{\mathcal{T}\univ}
\newcommand{\funiv}{\mathcal{F}\univ}
\newcommand{\dist}{{\rm d}}
\newcommand{\simep}{\sim_{O(\e)}}
\newcommand{\Ccal}{\mathcal{C}}
\newcommand{\ccal}{\mathcal{C}}
\newcommand{\pcone}{{\Ccal_{T,\beta}^{++}}}
\newcommand{\ncone}{{\Ccal_{T,\beta}^{--}}}
\newcommand{\conv}{{\rm Conv}}
\newcommand{\tconv}{\mathcal{T}\conv}
\newcommand{\tpi}{\tilde \pi}
\newcommand{\nlength}{\bar{\mathcal{L}}_\gamma}
\newcommand{\length}{{\mathcal{L}}_\g}
\newcommand{\tgamma}{\tilde{\g}}
\newcommand{\muone}{{\mu_{T,\eta}}^{(1)}}
\newcommand{\mutwo}{{\mu_{T,\eta}}^{(2)}}
\newcommand{\tmu}{\tilde\mu}
\newcommand{\conn}{{\rm Conn}_{\g}}
\newcommand{\taxis}{\tilde{R}_{\g}}
\newcommand{\Ghyp}{{\G_{{\rm hyp}}}}
\newcommand{\cusp}{{\rm cusp}}
\newcommand{\hcal}{\mathcal{H}}\newcommand{\et}{\mathcal{E}(T)}
\newcommand{\M}{\mathcal M}
\newcommand{\coc}{\mathsf{c}}
\newcommand{\cyl}{\mathsf{C}}
\newcommand{\mPS}[1]{\mu_{#1}^{\operatorname{PS}}}
\newcommand{\Pp}{\hat \sigma}

\title[Uniform exponential mixing]{Uniform exponential mixing and Resonance free regions
  for convex cocompact congruence subgroups of $\SL_2(\z)$}

\author{Hee Oh  and Dale Winter}

\thanks{Oh was supported in part by NSF Grant \#1361673.}

\begin{document}

\maketitle

\let\thefootnote\relax\footnote{2010 MSC. Primary:  37D35; 22E40; 37A25; 37D40; 11F72. Secondary: 37F30; 11N45.}
\small{\it Dedicated to Peter Sarnak on the occasion of his sixty-first birthday}

\begin{abstract}   Let $\G<\SL_2(\z)$ be
a non-elementary finitely generated subgroup and let
$\G(q)$ be its congruence subgroup of level $q$ for each $q\in \mathbb N$. 
We obtain an asymptotic formula for the matrix coefficients of $L^2(\Gamma (q) \ba \SL_2(\br))$
with a {\it uniform} exponential error term for all square-free
$q$ with no small prime divisors. As an application we establish a uniform resonance-free half plane for the resolvent
of the Laplacian on $\G(q)\ba \bH^2$ over $q$ as above.   
 Our approach is to extend Dolgopyat's dynamical proof of exponential mixing of the geodesic flow uniformly over congruence covers, by establishing
  uniform spectral bounds for congruence transfer operators associated to the geodesic flow.  One of the key ingredients is the expander theory
    due to Bourgain-Gamburd-Sarnak.
 \end{abstract}


\section{Introduction}
\subsection{Uniform exponential mixing}
Let $G=\SL_2(\br)$ and $\G$ be a  non-elementary finitely generated subgroup of $\SL_2(\z)$. We will assume that
$\G$ contains the negative identity $-e$ but no other torsion elements. In other words, $\G$ is the pre-image of a torsion-free subgroup
 of $\PSL_2(\z)$ under the canonical projection $\SL_2(\z)\to \PSL_2(\z)$.
 For each $q\ge 1$, consider the congruence subgroup of $\G$ of level $q$:
$$\G(q):=\{\gamma\in \G: \gamma \equiv e\text{ mod $q$}\} .$$



For $t\in \br$, let $$a_t=\begin{pmatrix} e^{t/2} & 0\\ 0 & e^{-t/2}\end{pmatrix}.$$
As is well known, the right translation action of $a_t$ on $\G\ba G$ corresponds to the geodesic flow 
when we identify $\G\ba G$ with the unit tangent bundle of a hyperbolic surface $\G \ba \bH^2$.  We fix a Haar measure $dg$ on $G$. 
By abuse of notation, we denote by $dg$ the induced $G$-invariant measure on $\G(q)\ba G$.
For real-valued functions $\psi_1, \psi_2 \in L^2(\G(q)\ba G)$,
we consider the matrix coefficient
$$\la a_t \psi_1, \psi_2\ra_{\G(q)\ba G} :=\int_{\G(q)\ba G} \psi_1 (ga_t) \psi_2 (g) dg.$$
 
 The main aim of this paper is to prove 
an asymptotic formula (as $t\to \infty$) with exponential error term for
 the  matrix coefficients
$\la a_t \psi_1, \psi_2\ra_{\G(q)\ba G} $  where the error term is {\it uniform} for all square free $q$ without small prime divisors.

\medskip

Denote by $\Lambda(\G)$ the limit set of $\G$, that is, the set of all accumulation points of
  $\G$-orbits in the boundary $\partial(\bH^2)$ and by
  $0<\delta = \delta_\G  \le 1 $ the Hausdorff dimension of $\Lambda(\G)$.
  
   The notation $C_c^k$ denotes the space
of $C^k$-functions with compact supports: 
\begin{thm} \label{main2} Let $\G<\SL_2(\z)$ be a convex cocompact subgroup, i.e., $\G$ has no parabolic elements.
 Then there exist $\eta>0, C\ge 3$ and $ q_0 > 1$ such that for any square free $q$ with $(q, q_0) = 1$ and any $\psi_1,  \psi_2\in C_c^1(\G(q)\ba G)$, we have

\begin{multline} \label{em4} e^{(1-\delta)t} \la a_t \psi_1, \psi_2\ra_{\G(q)\ba G}=\\
 \frac{1 }{m_q^{\BMS}(\G(q)\ba G)} m_q^{\BR}(\psi_1) m_q^{\BR_*}(\psi_2) +O(||\psi_1||_{C^1} ||\psi_2||_{C^1}\cdot q^C
 \cdot e^{-\eta t})\end{multline}
as $t\to +\infty$; here $m_q^{\BMS}, m_q^{\BR}$, and  $m_q^{\BR_*}$ denote respectively
 the Bowen-Margulis-Sullivan measure, the unstable Burger-Roblin measure, and  the stable Burger-Roblin measure on $\G(q)\ba G$ which are chosen compatibly with the choice of $dg$ (see Section 6 for precise definitions). 

The implied constant can be chosen uniformly for all $C^1$-functions $\psi_1, \psi_2$ whose supports project to a fixed compact subset of $\G \ba G$. 
\end{thm}

 If  $\G<\SL_2(\z)$ is finitely generated with $\delta > \frac{1}{2}$, then
  a version of Theorem \ref{main2}  is known  by  \cite{BGS} and \cite{BKS} with a different interpretation of the main term (also see 
  \cite{LO}, \cite{V}, \cite{MO}).
  Therefore the main contribution of Theorem \ref{main2} lies in the groups $\G$ with $\delta \leq \frac{1}{2}$; such groups are known to be convex cocompact.

\begin{rmk}\begin{enumerate} \rm
\item Selberg's celebrated $\frac{3}{16}$ theorem corresponds exactly to this result in the case $\G = \SL_2(\mathbb{Z})$ with the explicit constants $C = 3$ and $\eta(\G) = \frac{1}{4} - \epsilon$. One can therefore regard Theorem \ref{main2} as yet another generalization of Selberg's theorem to subgroups of infinite covolume. 
\item The optimal $q^C$ would be $q^3$, which is the growth rate of $[\G:\G(q)]$ in the above error term expressed in $C^1$-norms. 
 
  \item One would expect the results described in this paper to hold without the assumption that $q$ is
 square free; the missing piece is the $\ell^2$ flattening lemma (Lemma \ref{bgsflatteninglemma}), which is available in the literature
 only in the case of square free $q$. 
\item Theorem \ref{main2} has an immediate application to counting, equidistirbution and affine sieve; for instance, Theorems 1.7, 1.12, 1.14, 1.16 and 1.17 in \cite{MO} are now valid for
$\G<\SL_2(\z)$ with $\delta \le \frac 12$, with the $L^2$-sobolev norms of functions replaced by $C^1$-norms; the proofs are verbatim repetition
since Theorem \ref{main2} was  the only missing piece in the approach of that paper.
\end{enumerate}
\end{rmk}

 The main term in \eqref{em4} can be related to a Laplace eigenfunction on $\G(q)\ba \bH^2$. Denote by $\Delta$ the negative of the Laplacian on $\bH^2$
 and $\{\nu_x: x\in \bH^2\}$ the Patterson density for $\G$.
Then $\phi_o(x):=|\nu_x|$ is an eigenfunction of $\Delta$ in $C^\infty(\Gamma(q) \ba \bH^2)$
 with eigenvalue $\delta(1-\delta)$ \cite{Pa}, and $\phi_o\in L^2(\G(q) \ba \bH^2)$ if and only if $\delta_\G >1/2$.
If we identity $\bH^2=\SL_2(\br)/\SO(2)$, and
$\psi\in C_c(\G(q)\ba G)$ is $\SO(2)$-invariant,
then $$m_q^{\BR}(\psi)=\int_{\G(q)\ba G} \psi(x) \phi_o(x) dx =m_q^{\BR_*}(\psi).$$

\subsection{Uniform resonance free region}
When $\delta>\frac{1}2$,  Bourgain, Gamburd  and Sarnak  \cite{BGS} established a  uniform spectral gap for the smallest two Laplace eigenvalues
on $L^2(\G(q)\ba \bH^2)$ for all 
  square-free $q\in \N$ with no small prime divisors; for some $\e>0$,  there are no eigenvalues between
   $\delta(1-\delta)$, which is known to be the smallest one, and $\delta (1-\delta) +\e$. 

When $\delta \le \frac{1}{2}$, the $L^2$-spectrum of $\Delta$ 
is known to be purely continuous \cite{LP}, and
 the relevant spectral quantities are the resonances. 
The resolvent of the Laplacian $$R_{\G(q)}(s):=(\Delta - s(1-s))^{-1}
:C_c^\infty(\G(q)\ba \bH^2)\to C^\infty(\G(q)\ba \bH^2)$$ is holomorphic in the half plane $\Re(s)>\frac 12$ and has meromorphic continuation to the complex plane $\c$
with poles of finite rank  \cite{CO} (see also \cite{MM}). These poles are called  {\it resonances}.
Patterson showed that $s=\delta$ is a resonance of rank $1$ and that no other resonances occur in the half-plane $\Re s\ge \delta$
\cite{Pa1}. Naud proved that for some $\e(q)>0$, the half-plane $\Re s>\delta -\e(q)$ is a resonance-fee region except at $s=\delta$ \cite{Naud}.
 Bourgain, Gamburd and Sarnak  showed that for some $\e>0$,
$\{\Re s> \delta -\e \cdot  \min\{ 1, 1/(\log (1+|\Im s|)\}\}$ is a resonance free region except for $s=\delta$, for all square-free $q$ with no small prime divisors
\cite{BGS}.
  We will deduce a  uniform resonance-free half plane from Theorem \ref{main2} (see Section \ref{resolvent_c}):
 
 \begin{Thm}\label{resolvent} Suppose that $\delta \le \frac{1}{2}$. There exist $\e>0$ and $q_0>1$ such that for all square free $q\in \N$ with $(q, q_0)=1$, 
 $$\{\Re s>\delta -\e\} $$
 is a resonance free region  for the resolvent $R_{\G(q)}$ except for a simple pole at $s=\delta$.
 \end{Thm}

Let  $\mathcal P_q$ denote the set of all  primitive closed geodesics in 
$\T^1(\G(q)\ba \bH^2)$ and  let $\ell(C)$ denote  the length of $C\in \mathcal P_q$. 
The Selberg zeta function given by
 $$Z_q(s):=\prod_{k=0}^{\infty}\prod_{C\in \mathcal P_q} (1- e^{-(s+k) \ell (C)})$$
  is known to be an entire function when $\G(q)$ is convex cocompact by \cite{GZ}.

Since the resonances of the resolvent of
 the Laplacian give non-trivial zeros of 
 $Z_q(s)$ by \cite{PP}, 
 Theorem \ref{resolvent} follows from the following:
 \begin{Thm}\label{SZ}
There exist $\e>0$ and $q_0>1$ such that for all square free $q\in \N$ with $(q, q_0)=1$,
the Selberg zeta function $Z_q(s)$ is non-vanishing on the set
 $\{\Re(s)>\delta -\e\}$ except for a simple zero at $s=\delta$. 
\end{Thm} 

 \subsection{On the proof of Main theorems}
 Theorem \ref{main2} is deduced from the following uniform exponential mixing of the Bowen-Margulis-Sullivan measure $m^{\BMS}_q$:
\begin{thm}\label{BMS2} \label{t5.1} There exist $\eta>0, C\ge 3,$ and $ q_0 > 0$ such that, for all square free $q\in \N$ coprime to $q_0$,
and for any $\psi_1,  \psi_2\in C_c^1(\G(q)\ba G)$, we have
\begin{multline} \label{bms2}  \int_{\G(q)\ba G} \psi_1 (ga_t) \psi_2 (g)\; dm^{\BMS}_{q}(g) =
\\ \frac{1}{m_q^{\BMS}(\G(q)\ba G)} m^{\BMS}_{q}(\psi_1) \cdot m^{\BMS}_{ q}(\psi_2) +O(||\psi_1 ||_{C^1} ||\psi_2||_{C^1}\cdot q^C
 \cdot e^{-\eta t})\end{multline}
as $t\to +\infty$, with the implied constant depending only on $\Gamma$.
\end{thm}
Theorem \ref{BMS2} also holds when $\G$ has a parabolic element by \cite{MO}.
For a fixed $q$, Theorem \ref{BMS2} was obtained by Stoyanov \cite{St}.

We begin by discussing the proof of Theorem \ref{BMS2}.
The first step  is to use Markov sections constructed by Ratner \cite{Ra} and Bowen \cite{Bow} to build a symbolic model for the $a_t$-action on the space $\Gamma \ba G$. The Markov section gives a subshift $(\Sigma, \sigma)$ of finite type in an alphabet $\lbrace i_1,  \ldots ,  i_k\rbrace$, together with the associated space $(\Sigma^+, \sigma)$ of one sided sequences. 
Denote by $\tau : \Sigma \rightarrow \mathbb{R}$ the first return time for the flow $a_t$.  The corresponding suspension $\Sigma^\tau$ has a natural flow $\mathcal{G}_t$, a finite measure $\mu$ and an embedding
\[  \zeta : (\Sigma^\tau, \mu, \mathcal{G}_t) \rightarrow  (\Gamma \backslash G, m^{\BMS}, a_t)\]
which is an isomorphism of measure theoretic dynamical systems: this is our symbolic model. This framework will be the topic of Section 2. 

From the one sided shift we construct, for each $a, b\in \br$,  the transfer operator
$\mathcal{L}_{ab}   : C(\Sigma^+) \rightarrow C(\Sigma^+)$ by
\[ (\mathcal{L}_{ab}h)(x)  = \sum_{\sigma (y) = x} e^{-(\delta + a -ib)\tau(y)} h(y) .\]
Pollicott's observation, later used and refined by many other authors (\cite{Do}, \cite{St}, \cite{Av}), was that
the Laplace transform of the correlation function for the system $(\Sigma^\tau, \mu, \mathcal{G}_t) $ can be expressed in terms of
transfer operators using the Ruelle-Perron-Frobenius theorem, and that
the exponential mixing of $(\Sigma^\tau, \mu, \mathcal{G}_t) $ and hence that of $(\Gamma \backslash G, m^{\BMS}, a_t)$ follows 
if we prove a uniform spectral bound on $\mathcal{L}_{ab}$ for H\"older observables to be valid on $|a| \leq a_0$ for some $a_0>0$.

We write $\SL_2(q)$ for the finite group $\SL_2(\z/q\z)$. Following this approach, 
we define congruence transfer operators
$\mathcal{M}_{ab, q} $ on the space $C(\Sigma^+, \c^{\SL_2(q)})$
of vector-valued functions for each $q$ satisfying $\SL_2(q)=\G(q)\ba \G$ (which is the case whenever $q$ does not have small prime divisors): for $x\in \Sigma^+$ and $\gamma \in \SL_2(q)$,
\[ (\mathcal{M}_{ab, q}H)(x, \gamma)  = \sum_{\sigma (y) = x} e^{-(\delta + a -ib)\tau(y)} H(y, \gamma \coc^{-1}(y))) \]
where $\coc: \Sigma^+ \rightarrow \Gamma$ is a cocyle which records the way the $a_t$-flow moves elements from one fundamental domain to another.
The natural extension of Pollicott's idea tells us that {\it uniform} exponential mixing of $ (\Gamma(q) \backslash G, m_q^{\BMS}, a_t)$ will follow if we can establish certain spectral bounds for $\mathcal{M}_{ab, q}$ 
 uniformly for all $|a| \leq a_0$, $b\in \br$ and all $q$ large. This reduction will be carried out in Section 5.

The proof of spectral bounds for transfer operators traditionally falls into two parts. In Section 3 we shall consider the case where $|b|$ is large. The key ideas here are due to Dolgopyat, who gave an ingenious, albeit highly involved, proof of the relevant bounds for $\mathcal{L}_{ab}$ under additional assumptions. 
We will follow a treatment due to Stoyanov, who carries out the bounds on $\mathcal{L}_{ab}$ for axiom A flows. 
The bounds follow from an iterative scheme involving Dolgopyat operators, whose construction relies on the highly oscillatory nature of the functions $e^{ib\tau}$ when $|b|$ is large. This oscillation is also sufficient to establish bounds on the congruence transfer operators
$\mathcal{M}_{ab, q}$; see Theorem \ref{Dargtheorem}. Because the oscillation relies only on local non-integrability properties, the bounds we obtain are uniform in $q$. It is crucial for this argument that the cocycle $\coc$ is locally constant on an appropriate length scale, so that it doesn't interfere with the oscillatory argument. 

We are left, in Section 4, with the proof of the bounds on $\mathcal{M}_{ab, q}$ for $|a| \leq a_0$ and $|b|$ small. The bounds for $\mathcal{L}_{ab}$ in this region follow immediately from the complex Ruelle-Perron-Frobenius theorem and a compactness argument. Since we require bounds on $\mathcal{M}_{ab, q}$ uniformly in $q$, however, this compactness argument is not available to us; instead we follow the approach and use the expansion machinery of Bourgain-Gamburd-Sarnak
\cite{BGS}.

The expansion approach relies on the idea that $\G(q) \backslash G \sim \SL_2(q)  \times \G \ba G$. Very roughly, the hyperbolic nature of geodesic flow allow us to separate variables and to consider functions that are "independent" of the $\G \ba G$ component. We are left considering functions on $\SL_2(q)$; for such functions, the right action of the cocycle $\coc$, together with the expansion machinery and the $\ell^2$-flattening lemma produce the required decay. 
One essential estimate in this argument is proved by means of Sullivan's shadow lemma and the description of the relevant measures
in terms of the Patterson-Sullivan density. The memoryless nature of the Markov model for our flow is crucial here, as it allows us to relate these estimates to certain convolutions.

Theorem \ref{main2} is deduced from Theorem \ref{BMS2} by comparing the transverse intersections for the expansion of
a horocyclic piece, based on the quasi-product structures of the Haar and the BMS measures. 
 
  \medskip
In joint work with Magee \cite{MOW}, we extend a main theorem of \cite{Naud} uniformly over $q$ as well, which has an application to sieve for orbits of a semigroup as used in the work of Bourgain and Kontorovich on Zaremba's conjecture \cite{BK}.
We expect that our methods in this paper generalize 
to convex cocompact thin subgroups $\G$ of $\SO(n,1)$ and moreover to a general rank one group, which we hope to address in a subsequent paper.

\subsection*{Remark:} After submission of this paper new arguments have been developed that allow Theorem \ref{main2}  (and hence Theorems \ref{resolvent} and \ref{SZ}) to be proved without the assumption that $q$ be square free. The key point is to replace the $\ell^2$-flattening lemma with the expansion results of Bourgain and Varju \cite{BV} in the proofs of Propositions \ref{arb1} and \ref{arb2}. The new arguments are described in a recent preprint \cite{BKM} for the setting of \cite{MOW}, and should require only minor modification to apply in our setting.



\subsection*{Acknowledgements:} We are grateful to the
 referee for helpful remarks on the paper, especially for 
 providing  an alternative succinct argument in the deduction of Theorem \ref{resolvent} from Theorem \ref{main2}.

\section{Congruence transfer operators}
In the whole paper, let $G = \SL_2(\R)$ and let $\Gamma <G$  be a non-elementary, convex cocompact subgroup containing the negative  identity. 
We assume that $-e$ is the only torsion element of $\G$. 
If $p:\SL_2(\br)\to \PSL_2(\br)$ is the canonical projection, then $p(\G)$ is a convex cocompact torsion-free subgroup of $\PSL_2(\br)$ and
we have $\G\ba \SL_2(\br)=p(\G)\ba \PSL_2(\br)$. Since our results concern the quotient space $\G\ba G$,
we will henceforth abuse notation so that sometimes $G=\PSL_2(\br)$
and our $\G$ is considered as a torsion-free subgroup of $\PSL_2(\br)$.

We recall that the limit set $\Lambda(\G)$ is a minimal non-empty closed $\G$-invariant subset of the boundary $\partial \bH^2$, and its
Hausdorff dimension $\delta = \delta_\Gamma$  is equal to the critical exponent of $\G$ (see \cite{Pa}).

We denote by $\{\mu_x=\mu_x^{\PS}: x\in \bH^2\}$ the Patterson-Sullivan density for $\G$; that is, each $\mu_x$ is a finite measure
on $\Lambda(\G)$ satisfying
\begin{enumerate}
\item $\gamma_* \mu_x=\mu_{\gamma x}$ for all $\gamma\in \G$;
\item $\frac{d\mu_x}{d\mu_y}(\xi)=e^{\delta \beta_{\xi}(y, x)}$ for all $x,y\in \bH^2$ and $\xi\in \partial(\bH^2)$.
\end{enumerate}
Here $\beta_{\xi} (y,x)$ denotes the Busemann function:
$\beta_{\xi}(y,x)=\lim_{t\to\infty} d(\xi_t, y) -d(\xi_t, x)$ where
$\xi_t$ is a geodesic ray tending to $\xi$ as $t\to \infty$.
Since $\G$ is convex cocompact, $\mu_x$ is simply the $\delta$-dimensional Hausdorff measure on $\Lambda(\G)$ with respect to a spherical metric
viewed from $x$ (up to a scaling). See \cite{Pa} and \cite{Sullivan1979} for references.

Fixing $o\in \bH^2$, the map
$u\mapsto (u^+, u^-, s=\beta_{u^-}(o, u)) $
is a homeomorphism between $\T^1(\bH^2)$ and
$ (\partial(\bH^2)\times \partial(\bH^2) - \{(\xi,\xi):\xi\in \partial(\bH^2)\})  \times \br $.
Using this homeomorphism, and the identification of $\PSL_2(\br)$ with $\T^1(\bH^2)$,
the Bowen-Margulis-Sullivan measure
$\tilde m^{\BMS}=\tilde m^{\BMS}_\G$ on $\PSL_2(\br)$  is defined as follows:
\begin{align*}
d \tilde m^{\BMS}(u)&= e^{\delta \beta_{u^+}(o, u)}\;
 e^{\delta \beta_{u^-}(o, u) }\; d\mPS{o}(u^+) d\mPS{o}(u^-) ds . \end{align*}
 This definition is independent of the choice of $o\in \bH^2$, but does depend on $\G$.

We denote by $m^{\BMS}$ the measure 
on $\G\ba G$ induced by $\tilde m^{\BMS}$; it is called the Bowen-Margulis-Sullivan measure on $\G\ba G$, or the BMS measure for short.

Let $A= \lbrace a_t=\text{diag}(e^{t/2}, e^{-t/2}):t\in \br \rbrace$. 
The right translation action of $A$ on $\G \backslash G$ corresponds to the geodesic flow on $\mathrm{T}^1(\G \backslash \mathbb{H}^2)$. It is easy to check that
the BMS measure is $A$-invariant.  We choose the left $G$- and right $\operatorname{SO}_2(\mathbb{R})$-invariant metric $d$ on $G$ such that $d(e, a_t) = t$. 

Let $N^+$ and $N^-$ be the expanding and contracting horocyclic subgroups for $a_t$:
  \be \label{definenpm} N^+=\{n^+_s:=\begin{pmatrix} 1 & 0 \\ s & 1\end{pmatrix}: s\in \br\}\quad \text{and}\quad N^-=
  \{n^-_s:=\begin{pmatrix} 1 & s \\ 0 & 1\end{pmatrix}: s\in \br\} .\ee
  For $\epsilon > 0$, we will denote by $N^\pm_\epsilon$ the intersection of the $\epsilon$ ball around the identity, $B_\epsilon(e)$, with $N^\pm$. 
 
 We fix a  base point $o \in \mathbb{H}^2$ in the convex hull of the limit set $\Lambda(\Gamma)$, and write $\Omega$ for the support of the BMS measure. The geodesic flow $a_t:\Omega \to \Omega$ is known to be mixing for the BMS measure $m^{\BMS}$ by Rudolph \cite{Ru} (see also \cite{Ba}). Since $\Gamma$ is convex cocompact, $\Omega$ is compact and there is a uniform positive lower
bound for the injectivity radii for points on $\Gamma\backslash G$, which we will simply call the injectivity radius of $\Gamma$.

\subsection{Markov sections} We refer to \cite{handbook} for basic facts about Markov sections. Let $\alpha>0$ be a small number.
 Consider a finite set $z_1, \ldots , z_k$ in $\Omega$ and choose small compact neighborhoods $U_i$ and $S_i$ of $z_i$ in $z_iN_\alpha^+ \cap \Omega$ and  $ z_i N_\alpha^-\cap \Omega$ respectively of diameter at most $\alpha/2$. We write $\mbox{int}^u(U_i)$ for the interior of $U_i$ in the set $z_iN^+_\alpha \cap \Omega$ and define $\mbox{int}^s(S_i)$ similarly.  We will assume that $U_i$ (respectively $S_i$) are proper, that is to say, that $U_i = \overline{\mbox{int}^u(U_i)}$ (respectively $S_i = \overline{\mbox{int}^s(S_i)}$). 
 For $x \in U_i$ and $y \in S_i$, we write $[x, y]$ for the unique local intersection of $xN^-$ and $yN^+A$. We write the rectangles as 
\[ R_i = [U_i, S_i] := \{ [x, y]: x\in U_i, y \in S_i \}\]
and denote their interiors by
\[ \mbox{int}(R_i)= [\mbox{int}^u(U_i), \mbox{int}^s(S_i)] .\]
Note that $U_i = [U_i, z_i] \subset R_i$. The family $\mathcal{R} = \lbrace R_1 , \ldots R_k \rbrace$ is called a complete family of size $\alpha >0$ if
\begin{enumerate}
\item $ \Omega = \cup_1^k R_i a_{[0, \alpha]}$
\item the diameter of each $R_i$ is at most $\alpha$, and 
\item for any $i \neq j$, at least one of the sets $R_i \cap R_ja_{[0, \alpha]}$ or $R_j \cap R_ia_{[0, \alpha]}$ is empty.
\end{enumerate}

Set $ R = \coprod_i R_i$.
Let $\tau: R \rightarrow \mathbb{R}$ denote the first return time and $\mathcal{P} : R \mapsto R$ the first return map:
 $$\tau(x) := \inf \lbrace t> 0 : xa_t \in R\rbrace \quad \text {and} \quad \mathcal P(x):= xa_{\tau(x)}.$$

\begin{dfn}[Markov section]  \rm A complete family
$\mathcal{R} := \lbrace  R_1 \ldots R_k\}$ of size $\alpha$ is called  a Markov section for the flow $a_t$
if the following the Markov property is satisfied:
\[ \mathcal{P}([\operatorname{int}^uU_i, x])  \supset [\operatorname{Int}^uU_j, \mathcal P (x)]
   \mbox{ and }  \mathcal{P}([x, \operatorname{Int}^sS_i] )) \subset [\mathcal P (x),\operatorname{Int}^sS_j]  \]
whenever $x\in \op{int} (R_i)\cap {\mathcal P}^{-1}(\op{int}(R_j))$. 
\end{dfn}  We consider the $k\times k$ matrix
\[ \operatorname{Tr}_{lm} = \left\lbrace  \begin{array}{ll} 1 \mbox{ if }\mbox{int} (R_l) \cap \mathcal{P}^{-1} \mbox{int} (R_m)  \neq \emptyset \\  0 \mbox{ otherwise, }  \end{array} \right.\]
which we will refer to as the transition matrix.
The transition matrix $\mathrm{Tr}$ is called topologically mixing if
there exists a positive integer $N$ such that all the entries of $\mathrm{Tr}^N$ are positive.
Ratner \cite{Ra} and Bowen \cite{Bo} established the existence of Markov sections of arbitrarily small size; using an argument of Bowen and Ruelle \cite{BRu} we may further assume that the associtaed tranition matrix is topologically mixing. We now fix such an $\mathcal{R} = \lbrace R_1=[U_1, S_1], \ldots , R_k= [U_k, S_k]\rbrace$ of size $\alpha$,
where $\alpha>0$ satisfies
\[ \alpha <\tfrac{1}{1000} \cdot \mbox{Injectivity radius of $\Gamma\ba G$} \]
and for all $|s| < 4\alpha$,
\be d(e, n^+_s) \leq |s| \leq 2d(e, n^+_s) \label{boundaalpha} . \ee 
Note that $k \geq 2$ as a consequence of the non-elementary property of $\G$. 

Write
 $$U := \coprod_i U_i \quad \mbox{ and } \quad  \mbox{int}(R) = \coprod_i \mbox{int}(R_i).$$

 The projection map along stable leaves
\[ \pi_S: R \rightarrow U, \mbox{ taking } [x, y] \mapsto x\]
will be very important for us at several stages of the argument. We will write $\Pp$ for the map
$$ \Pp:= \pi_S \circ \mathcal{P}:  U \rightarrow  U. $$

 \begin{dfn}\rm We define the cores of $R$ and $U$ by  
 $$\hat R = \lbrace x \in R: \mathcal{P}^{m} x \in \mbox{int} (R ) \mbox{ for all } m \in \mathbb{Z} \rbrace \mbox{, and}$$
$$ \hat U = \lbrace u \in U : \Pp^{m} u \in \mbox{int}^u(U) \mbox{ for all } m \in  \mathbb{Z}_{\geq 0} \rbrace. $$
\end{dfn}

Note that $\hat {R}$ is $\mathcal{P}$-invariant, and that $\hat U$ is $\Pp$-invariant. The cores are residual sets (that is, their complements are countable unions of nowhere dense closed sets).


\subsection{Symbolic dynamics} \label{ss2.2} We choose $\Sigma$ to be the space of bi-infinite sequences $x \in \lbrace 1, \ldots, k \rbrace^\mathbb{Z}$ such that $\operatorname{Tr}_{x_l x_{l+1}} = 1$ for all $l$. Such sequences will be said to be admissible. We denote by $\Sigma^+$ the space of one sided admissible sequences 
$$\Sigma^+= \lbrace (x_i)_{ i\ge 0 } : \operatorname{Tr}_{x_i, x_{i + 1}} = 1 \mbox{ for all } i \geq 0 \rbrace.$$
We will write $\sigma: \Sigma \rightarrow \Sigma$ for the shift map $(\sigma x)_i = x_{i+1}$.
By abuse of notation we will also allow the shift map to act on $\Sigma^+$. 

\begin{dfn}\rm\rm
For $\theta \in (0, 1)$, we can give a metric $d_\theta$ on $\Sigma$ (resp. on $\Sigma^+$) by choosing
\[ d_\theta(x, x') = \theta^{\inf \lbrace |j|: x_j \neq x'_j \rbrace}.\]
\end{dfn}

For a finite admissible sequence $i = (i_0, \ldots , i_m)$, we obtain a cylinder  of length $m$:
\be\label{defcy}\mathsf{C}[i] := \lbrace u\in\hat U_{i_0}: \Pp^j(u) \in \mbox{int}(U_{i_j}) \mbox{ for all } 0\leq j \leq m \rbrace. \ee
Note that cylinders of length $0$ are precisely $U_i$'s and that cylinders are open subsets of $\hat U$. By a closed cylinder, we mean the closure of some (open) cylinder. We also take this opportunity to introduce embeddings of the symbolic space into the analytic space.

\begin{dfn} [The map $\zeta:\Sigma \rightarrow \hat R$]   \rm
For $ x \in \hat R$, we obtain a sequence $\omega = \omega(x) \in \Sigma$ by requiring  $\mathcal{P}^k x\in R_{\omega_k}$  for all $k \in \mathbb{Z}$. The set 
$ \hat \Sigma := \lbrace \omega(x) : x \in \hat R \rbrace $
is a residual set in $\Sigma$. Using the fact that any distinct pair of geodesics in $\mathbb{H}^2$ diverge from one another (either in positive time or negative time), one can show that the map $x \mapsto \omega(x)$ is injective. We now define a continuous function $\zeta:\Sigma \rightarrow \hat R$ by choosing $\zeta(\omega(x)) = x$ on $\hat \Sigma$ and extending continuously to all of $\Sigma$. 
\end{dfn}

The restriction $\zeta : \hat \Sigma \rightarrow \hat R$ is known to be bijective and to intertwine $\sigma$ and $ \mathcal{P}$. 
 \begin{dfn} [The map $\zeta^+ : \hat \Sigma ^+ \rightarrow \hat U$] \rm
 For $u \in \hat U$, we obtain a sequence $\omega'(u)\in \Sigma^+$ by requiring $\mathcal{P}^k x\in R_{\omega'_k}$  for all $k \in \mathbb{Z}_{\geq 0}$. We obtain an embedding $\zeta^+: \Sigma^+ \rightarrow U$ by sending $\omega'(u) \mapsto u'$ where possible and extending continuously.
We write $\hat \Sigma^+:= (\zeta^+)^{-1}(\hat U)$. The restriction $\zeta^+ : \hat \Sigma ^+ \rightarrow \hat U$ is known to be
 bijective and to intertwine $\sigma$ and $ \Pp $. 
\end{dfn}

For $\theta $ sufficiently close to $1$, the embeddings $\zeta, \zeta^+$ are Lipschitz. We fix such a $\theta$ once and for all. The space $C_\theta(\Sigma)$
(resp. $C_\theta(\Sigma^+)$)
of $d_\theta$-Lipschitz functions on $\Sigma$ (resp. on $\Sigma^+$) is a Banach space with the usual Lipschitz norm
\[ ||f||_{d_\theta} =  \sup |f|+ \sup_{x \neq y} \frac{|f(x) - f(y)|}{d_\theta(x, y)}.\]

Writing $\tilde \tau := \tau \circ \zeta \in C_\theta(\Sigma)$, we form the suspension
\[ \Sigma^\tau :=\Sigma \times \mathbb{R} / (x, t + \tilde \tau x)\sim (\sigma x, t) .\]
We write $\hat \Sigma^\tau$ for the set $(\hat \Sigma \times \mathbb{R} / \sim) \subset \Sigma^\tau$. The suspension embeds into the group quotient via the map
\[ \zeta^\tau : \hat \Sigma^\tau \rightarrow \Gamma \backslash G, \quad (x, s) \rightarrow \zeta(x) a_s\]
and has an obvious flow $\mathcal{G}_t: (x, s) \mapsto(x, t+s)$. The restriction $\zeta^\tau : \hat \Sigma^\tau \rightarrow \G\backslash G$ intertwines $\mathcal{G}_t$ and $a_t$.

\subsection{Pressure and Gibbs measures.}
\begin{dfn}\rm\rm
For a real valued function $f \in C_\theta (\Sigma)$, called the potential function, we define the pressure to be the supremum
\[ Pr_\sigma(f) := \sup_\mu \left( \int_{\Sigma} f d\mu + \mbox{entropy}_\mu(\sigma) \right) \]
over all $\sigma$-invariant Borel probability measures $\mu$ on $\Sigma$; here entropy$_\mu(\sigma)$ denotes the measure theoretic entropy of $\sigma$ with respect to $\mu$. 
\end{dfn}
For a given real valued function $f \in C_\theta (\Sigma)$, there is a unique $\sigma$-invariant probability measure on $\Sigma$ that achieves the supremum above, called
the equilibrium state for $f$.  We will denote it $\nu_f$.  It satisfies $\nu_f(\hat \Sigma) = 1$.

To any $\sigma$-invariant measure $\mu$ on $\Sigma$, we can associate a $\mathcal{G}_t$-invariant measure $\mu^\tau$ on $\Sigma^\tau$; simply take the local product of $\mu$ and the Lebesgue measure on $\mathbb{R}$. Our interest in these equilibrium states is justified in light of the following fact.
\begin{Nota} \rm We will write $\nu$ for the $-\delta ( \tau\circ \zeta) $-equilibrium state on $\Sigma$. We remark that the pressure $Pr_\sigma(-\delta (\tau\circ\zeta))$ is known to be zero.  \end{Nota}
\begin{thm} \label{BMSisGIbbs} Up to a normalization,  the measure $m^{\BMS}$ on
 $\Gamma \backslash G$ coincides with the pushforward $\zeta^\tau_* \nu^{ \tau}$.  
\end{thm}
\begin{proof} Sullivan \cite{Sullivan1984} proved that $m^{\BMS}$ is the unique measure of maximal entropy for the $a_t$ action on $\G \backslash G$. On the other hand $\zeta^\tau_*\nu^{\tau}$ is also a measure of maximal entropy on $(\G \backslash G, a_t)$ by \cite{handbook}. The result follows. 
\end{proof}
In particular, this theorem implies that $(\Sigma^\tau, \mathcal{G}_t, \nu^\tau)$ and $(\G\backslash G, a_t, m^{\BMS})$ are measurably isomorphic as dynamical systems via $\zeta^\tau$. One simple consequence is the following corollary.
\begin{Cor} \label{c2.4}  The measures $(\pi \circ \textup{vis}^{-1}) _* \mu^{\PS}_o$ and $(\pi_S\circ \zeta)_* \nu$ are mutually absolutely continuous on each $U_i$ with bounded Radon-Nikodym derivative. 
Here $\textup{vis}$ denotes the visual map from a lift $\tilde{U}_i$  to $\partial(\mathbb{H}^2)$, and $\pi$ is the projection $G \rightarrow \G \backslash G$.
\end{Cor}
By abuse of notation, we use the notation $\nu$ for the measure $(\pi_S\circ \zeta)_* \nu$ on $U$.

\subsection{Transfer operators.} The identification of $\Sigma^\tau$ and $\G \backslash G$ above allows the use of symbolic dynamics in the study of the BMS measure. In particular we will use the theory of transfer operators. 
\begin{dfn}\rm
For $f \in C_\theta(\Sigma^+)$, we obtain a transfer operator $\mathcal{L}_f: C(\Sigma^+) \rightarrow C(\Sigma^+)$
by taking
\[ \mathcal{L}_f(h)(u) := \sum_{ \sigma(u') = u} e^{f(u')} h(u'). \]
\end{dfn}
A straightforward calculation shows that $\mathcal{L}_f$ preserves $C_\theta(\Sigma)$. 
The following is a consequence of the Ruelle-Perron-Frobenius theorem together with the well-known theory of
 Gibbs measures (see \cite{PP}, \cite{St2}):
\begin{thm} \label{RPF} For each real valued function $f\in C_\theta(\Sigma^+)$, there exist
a positive function $\hat h \in C_\theta (\Sigma^+)$, a probability measure $\hat \nu$ on $\Sigma^+$, and $\epsilon>0, c>0$ such that
\begin{itemize}
\item $ \mathcal{L}_f (\hat h) = e^{Pr_\sigma(f)} \hat h; $
\item the dual operator satisfies $\mathcal{L}^*_f \hat \nu = e^{Pr_\sigma(f)} \hat \nu; $
\item for all $n\in \N$,
 $$| e^{-nPr_\sigma(f)} \mathcal{L}_f^n( \psi)(x) - \hat \nu (\psi) \hat h (x)| \leq c(1-\epsilon)^n ||\psi ||_{\Lip (d_\theta)} ;$$
 with $\hat h$ normalized so that $\hat \nu (\hat h)=1$;

  \item the measure $\hat h\hat \nu$ is $\sigma$-invariant and is the projection of the $f$-equilibrium state  to $\Sigma^+$. 
\end{itemize}
The constants $c, \epsilon$ and the Lipschitz norm of $\hat h$ can be bounded in terms of the Lipschitz norm of $f$; see \cite{St2}
\end{thm}

\begin{remark} Using the identification of $\Sigma^+$ and $\hat U$ by $\zeta^+$, we can regard the transfer operators defined above as operators on $C(\hat U)$. We can also regard the metric $d_\theta$ as a metric on $\hat U$. We will do both of these freely without further comment. 
\end{remark}

We also define the normalized transfer operators. For $a \in \mathbb{R}$ with $|a|$ sufficiently small, 
consider the transfer operator $ \mathcal{L}_{-(\delta + a)\tau}$ on the space $C_{d_{\theta}} (U)$.
 Let $\lambda_a:=e^{Pr_\sigma (-(\delta +a)\tau)}$ be the largest eigenvalue, $\hat \nu_a$ the probability measure such that
 $\mathcal{L}^*_{-(\delta +a)\tau} \hat \nu_a = \lambda_a \hat \nu_a $
and let $h_a$ be the associated positive eigenfunction, normalized so that
$\int h_a \hat d\nu_a=1$.
 It is known that $\lambda_0 = 1$, and that $\lambda_a$ and $h_a$ are
 Lipschitz in $a$ for $|a|$ small. It is also known that
 for $|a|$ small, each $h_a$ is Lipschitz in the $d$-metric \cite{St2}.

\begin{Nota}\rm  For functions $f: \hat U \rightarrow \mathbb{R}$ and $h:  \Sigma^+\rightarrow \mathbb{R}$,
 we will 
 write $$f_n(u) := \sum_{i = 0}^{n-1} f(\Pp^i u)\quad\text{ and} \quad h_n (\omega) = \sum_{i=0}^{n-1} h(\sigma^i\omega) .$$   \end{Nota}

It follows from the fourth part of Theorem \ref{RPF} that there exist $c_1, c_2>0$ such that for all  $x\in \Sigma^+$ and for all $n\in \N$,
\be\label{ggibbs} c_1 e^{-(\delta +a)\tau_n(x)} \lambda_a^{-n}\le \hat \nu_a( {\mathsf C}[x_0, \cdots, x_n] )\le c_2 e^{-(\delta +a)\tau_n(x)} \lambda_a^{-n};\ee
moreover $c_1, c_2$ can be taken uniformly uniformly for $|a|<a_0$ for a fixed $a_0>0$. In particular, $\hat \nu_a$ is a Gibbs measure for the potential function $-(\delta+a)\tau$.

\medskip

 We consider
\begin{equation} f^{(a)} := - (\delta + a)\tau + \log h_0 - \log h_0\circ\sigma - \log \lambda_a,\label{normalized} \end{equation}
and let $ \hat{\mathcal{L}}_{ab} := \mathcal{L}_{f^{(a)} + ib\tau}$, i.e.,
\[ \hat{\mathcal{L}}_{ab}(h)(u) := \frac{1}{\lambda_a h_0(u)} \sum_{ \sigma(u') = u} e^{(-\delta +a-ib) \tau (u')} ({h_0\cdot h})(u') \]
be the associated transfer operator. Note that $ \hat{\mathcal{L}}_{ab}$ preserves the spaces $C_{d}(\hat U)$. 
We remark that the pressure $Pr_\sigma(f^{(a)})$ is zero; so the leading
eigenvalue of  $ \hat{\mathcal{L}}_{a0}$ is $1$, with an eigenfunction $h_a/h_0$.
 Since $f^{(0)}$ is cohomologous to $-\delta \tau$, the corresponding equilibrium states coincide.

\subsection{Congruence transfer operators and the cocycle $\coc$}
Let $\mathcal{D}$ be the intersection of the Dirichlet domain for $(\Gamma, o)$ in $\mathbb{H}^2$
and the convex hull of $\Lambda(\G)$. For each $R_j\subset \G\ba G$,
we choose a lift $\tilde{R}_j = [\tilde{U}_j, \tilde{S}_j]$ to $G$ so that the projection of $\tilde R_j$ to $\mathbb{H}^2$ intersects $\overline{\mathcal{D}}$ non-trivially. We write $\tilde R := \cup \tilde R_i$.

\begin{dfn}[Definition of the cocycle $\mathsf{c}: R \to \Gamma$]  \rm
For $x\in R$ with (unique) lift $\tilde x \in \tilde R$, we define the cocycle $\coc$ by requiring that
 \be \tilde xa_{\tau(x)} \in \mathsf{c}(x) \tilde R .  \label{defcoc} \ee  
 For $n \in \N$ and $x \in U \subset R$, we write $$\coc_n(x):= \coc(x) \coc(\Pp(x)) \ldots \coc(\Pp^{n - 1}x).$$\end{dfn}

\begin{lem} \label{cocycleexists} \begin{enumerate} \item If $x, x' \in R_j \cap \mathcal{P}^{-1}R_l$, then $\coc(x ) = \coc(x')$.
\item If $x, x'$ are both contained in  some cylinder of length $n\ge 1$, then $\coc_n(x) = \coc_n(x')$.
\end{enumerate}
\end{lem}
\begin{proof} Let $x_1, x_2 \in R_j \cap \mathcal{P}^{-1}R_l$.
 If $\tilde x_1, \tilde x_2\ \in \tilde{R}_j$ with $x_j = \G\ba \G \tilde x_j$, then for  $\tilde y_i  := \coc(x_i)^{-1} \tilde x_ia_{\tau(x_i)} \in \tilde{R}_j$, we have
\begin{eqnarray*} d(\coc(x_1) \tilde y_1, \coc(x_2)\tilde y_1) &\leq& d(\tilde x_1, \coc(x_2)\tilde y_1) +\alpha \\
&\leq& d(\tilde x_2, \coc(x_2)\tilde y_1) + 2\alpha \\
&\leq& d(\coc(x_2) \tilde y_2, \coc(x_2)\tilde y_1) + 3\alpha \\ 
&\leq& 4\alpha,  \end{eqnarray*}
which is less than the injectivity radius of $\Gamma$. Thus $\coc(x_1) = \coc(x_2)$ as desired.
The second statement is now straightforward from the definition of $\coc_n$. 
\end{proof}

Let $\Gamma (q) $ be a normal subgroup of $\Gamma$ of finite index and denote by $F_q$ the finite group $\Gamma(q)\backslash \Gamma$. 
We would like a compatible family of Markov sections for the dynamical systems $(\Gamma (q) \backslash G, a_t, m^{\operatorname{BMS}})$.
 The lifts $\tilde R_l$ give a natural choice; for $l \in \lbrace 1, \ldots k\rbrace$ and $\gamma \in F_q$, we take
\[ R^q _{l, \gamma} =\G(q) \gamma \tilde{R}_l \subset \Gamma(q) \backslash G . \]
The collection
\[\mathcal{R}^q := \lbrace R^q_{l, \gamma} : l\in \lbrace 1 \ldots k \rbrace \mbox{ and } \gamma \in F_q \rbrace\]
is a Markov section of size $\alpha$ for $(\Gamma (q) \backslash G, a_t)$ as expected. The first return time $\tau_q$ and first return map $\mathcal{P}_q$ associated to $\mathcal{R}^q$ are given rather simply in terms of the cocycle $c$ and the corresponding data for $\mathcal{R}$. 

\begin{remark} Let $\pi_q : \G(q) \backslash G \rightarrow \G \backslash G$ be the natural covering map. If $\tilde x \in R^q_{l, \gamma}$ and $\pi_q(\tilde x) \in R_l \cap \mathcal{P}^{-1} R_m $, then
\[ \tau_q(\tilde x) = \tau(\pi_q(\tilde x))\]
and $\mathcal{P}_q(x)$ is the lift of $\mathcal{P}(\pi(x))$ to $R^q_{m, \gamma \coc(\tilde x)}$.
\end{remark}
Embedded inside each partition element $R^q_{l, \gamma}$, we have a piece of an unstable leaf. Let $\tilde U_l$ be the lift of $U_l$ contained in $ \tilde R_l$. Then
the subsets
\[ U^q_{l, \gamma} := \Gamma(q) \gamma  \tilde U_l  \subset \Gamma (q) \backslash G
\text{ and }  \hat U^q_{l, \gamma}  :=  U^q_{l, \gamma}   \cap \pi_q^{-1} (\hat U) \]
are contained in $ R^q_{l, \gamma}$. We then write $\hat{U}^q := \coprod \hat U^q_{l, \gamma}$  for the union
and $ \Pp_q: \hat{U}^q \rightarrow \hat{U}^q$ for the natural extension of $\Pp$. 
Just as the partition $\mathcal{R}$ gives rise to a symbolic model of the geodesic flow on $\Gamma \backslash G$, so $\mathcal{R}^q$ provides a model for $\Gamma(q) \backslash G$. In particular we can identify $\hat U^q$ with $\hat U\times F_q$ in a natural way; simply send $(u, \gamma)$ to the image $\gamma \tilde u$ where $\tilde u$ is the lift of $u$ to $\tilde R$. Note then that $\Pp_q$ acts as the map
\[ \Pp_q(u, \gamma) = (\Pp u, \gamma \coc(u)). \]

For $f_q\in C(\hat U^q)$, we may consider the following transfer operators $ \mathcal{L}_{  f_q, q}: C(\hat U^q) \rightarrow  C(\hat U^q)$ given by
\begin{eqnarray*}(\mathcal{L}_{ f_q, q}h)(u, \gamma) &:=& \sum_{\sigma_q(u', \gamma')= (u, \gamma)}e^{f_q(u', \gamma')}h(u', \gamma')\\
&=& \sum_{\sigma(u') =  u}e^{f_q(u', \gamma \coc(u')^{-1})}h(u', \gamma(\coc(u'))^{-1}). \end{eqnarray*}
It will very often be helpful to think of a function $h \in C(\hat {U}^q)$ as a vector valued function $\hat U \rightarrow \mathbb{C}^{F_q}$.  In the case where $f_q(\omega, \gamma) = f(\omega)$ doesn't depend on the group element, we can then recover the congruence transfer operator 
$\mathcal{M}_{ f, q} : C(\hat U, \mathbb{C}^{F_q}) \rightarrow  C(\hat U, \mathbb{C}^{F_q}) $ given by
\[ (\mathcal{M}_{ f, q} H)(u)  = \sum_{\Pp(u') = u}e^{f(u')}H(u')\coc(u') ;\]
where $\coc(u')$ acts on $H(u') \in \mathbb{C}^{F_q}$ by the right regular action. We will often write ${(\mathcal{M}_{ f, q} H)(u, \g)}$ to mean the $\g$ component of $(\mathcal{M}_{ f, q} H)(u)$ . Most of this paper is going to be devoted to a study of these congruence transfer operators. The key example for us will be the normalized congruence transfer operator
$ \hat{\mathcal{M}}_{ab, q}:= \mathcal{M}_{f^{(a)} + ib\tau, q}: C(\hat U, \mathbb{C}^{F_q})\rightarrow C(\hat U, \mathbb{C}^{F_q})$:
\[ (\hat{\mathcal{M}}_{ ab, q} H)(u)  =\frac{1}{\lambda_a h_0( u)} \sum_{\Pp(u') = u}e^{-(\delta +a-ib)\tau(u)} {(h_0 H)} (u')\coc(u') .\]

We then have that for any $n\in \N$,
$$ \hat{\mathcal{M}}_{ab, q}^n H(u, \gamma):=\sum_{\hat \sigma^n(u')=u} e^{(f_n^{(a)} +ib\tau_n)(u')} H(u', \g \coc_n^{-1}(u')).$$
The key point will be to establish spectral properties of these congruence transfer operators. To do this we must first establish norms and banach spaces appropriate to the task. We will write $|\cdot|$ for the usual Hermitian
 norm on $\mathbb{C}^{F_q}$. For Lipschitz functions $H: \hat U  \rightarrow \mathbb{C}^{F_q}$, we define the norms
\begin{equation} ||H||_{1, b} :=   \sup_{u \in \hat U} |H(u)|+   \frac{1}{\max(1, |b|)} \sup_{u \neq u'} \frac{ |H(u) - H(u')|}{d(u, u') } \mbox{ and} \label{deflipnorm}\end{equation}
\begin{equation} ||H||_{2} :=\left( \int |H(u)|^2 d\nu(u) \right)^{1/2}. \label{definel2norm}\end{equation}
We will sometimes also write $||\cdot||_{\Lip(d)} := ||\cdot||_{1, 1}$ for the Lipschitz norm and denote by $C_{\Lip (d)}(\hat U,\mathbb{C}^{F_q} )$ the space of Lipschitz functions for the  norm $||\cdot||_{\Lip(d)}$. 

Consider the space of functions
 \begin{equation} \mathcal{W}(\hat U, \mathbb{C}^{F_q} ) = \lbrace H \in C_{\Lip(d)}(\hat U, \mathbb{C}^{F_q}) : 
  \sum_{\gamma \in F_q }H(u, \gamma) = 0 \mbox{ for all } u \in\hat U \rbrace.  \label{defineW}\end{equation}
 We will write $L_0^2(F_q)$ for the space of complex valued functions on $F_q$ that are orthogonal to constants. We can then think of $ \mathcal{W}(\hat U, \mathbb{C}^{F_q})$ as the space of Lipschitz functions from $\hat U$ to $L_0^2(F_q) $.

We're now in a position to state the main technical result of our argument.  Suppose that $\G $ is a (non-elementary) convex cocompact subgroup of $\SL_2(\z)$.
We recall the congruence subgroups $\G(q)$ of $\G$. 
Since $\G$ is Zariski dense in $\SL_2$, it follows from the strong approximation theorem that there exists $q_0\ge 1$ such that for all $q\in \N$ with $(q, q_0)=1$,  we have 
\be \label{strongapproximation} \G(q)\backslash \G =G(\z/ q\z)=\op{SL}_2(q).\ee

\begin{thm} \label{spectralbound} 
There exist $\epsilon>0, a_0>0, C>0, q'_0>1$ such that for all $|a| < a_0$, $b \in \mathbb{R}$, and for all  square free $q\in \N$ with  $(q, q_0q_0')=1$,
 we have
\[  ||\hat{\mathcal{M}}^{m}_{ab, q} H ||_2 \leq C(1 - \epsilon)^mq^C ||H||_{1, b}\]
 for all $m\in \mathbb{N}$ and all $H \in  \mathcal{W}(\hat U, \mathbb{C}^{F_q} ) $. 
\end{thm}

The next two sections will be focused on the proof of this Theorem. In Section 3 we prefer to work with the analytic space $\hat U$ and the associated function spaces $C(\hat U)$, while in Section 4 the symbolic space $\hat \Sigma ^+$ is preferred. For the most part we can unify these viewpoints through the identification $\zeta : \hat \Sigma^+ \rightarrow \hat U$; in particular for those parts of the argument where we consider the transfer operators acting on the $L^2 (\nu)$ spaces there is no problem, as the measure theory does not see the precise geometry of the spaces  $\hat \Sigma$ and $ \hat U$. The one potential difficulty is where we want to use the $d$-Lipschitz properties of $h_a$ and $f^{(a)}$, which a priori do not follow from the usual statement of the RPF theorem \ref{RPF}. This is clarified by \cite{PS}, which ensures we can proceed as required.

\section{Dolgopyat operators and vector valued functions} \label{bunbounded}
 In this section we aim to prove that Theorem \ref{spectralbound} holds whenever $|b|$ is sufficiently large:

\begin{thm} \label{Dargtheorem} There exist $\epsilon>0, a_0>0,b_0>0, C>0$ such that for all $|a| < a_0$,  $|b| > b_0$,
and for  any normal subgroup $\Gamma (q) $ of $\Gamma$ of finite index,
 we have
\[ ||  \hat{\mathcal{M}}^{m}_{ab, q} H ||_2< C(1 - \epsilon)^m ||H||_{1, b}\]
 for all $m\in \mathbb{N}$ and all $H \in  C_{\Lip (d)}(\hat U, \mathbb{C}^{F_q})$ for $F_q=\G(q) \backslash \G$.
\end{thm}
The strategy here is due to Dolgopyat \cite{Do}, and uses the construction of so-called Dolgopyat operators. This construction was generalized to axiom A flows by Stoyanov \cite{St}, and we will follow his argument. The remaining task is to relate these operators to our vector valued functions. The main reasons we succeed are that
 (1) the cocycle $\coc: \hat U\to \Gamma$ is locally constant (Lemma \ref{cocycleexists}) and (2) its action on $L^2(F_q)$ is unitary.
Both properties are elementary but they are the critical reasons why our approach works.

Following Stoyanov, we begin by defining a new metric on $\hat U$: for $u, u'\in \hat U$, set 
\be \label{defineD} D(u, u') = \inf \lbrace \mbox{diam}(\cyl): \mbox{ $\cyl$ is a cylinder containing $u$ and $u'$} \rbrace \ee
where $\op{diam}(\cyl)$ means the diameter of $\cyl$ in the metric $d$.
Note that for all $u, u' \in \hat U$,
\[ d(u, u') \leq D(u, u') .\]
\begin{dfn}\rm For $E>0$, we write $K_E(\hat U) $ for the set of all positive functions $h \in C(\hat U)$ satisfying 
\[ |h(u) - h(u')| \leq Eh(u') D(u, u') \]
for all $u, u'\in\hat U$ both contained in $ \hat U_i$ for some $i$. 
\end{dfn}
 Theorem \ref{Dargtheorem} follows from the following technical result as in the works of Dolgopyat and Stoyanov.

\begin{thm} \label{dolgopyatoperatorsexist}
There exist positive constants $N \in \mathbb{N}, E > 1, \epsilon, a_0, b_0 $ such that for all $a, b$ with $|a| < a_0$, $|b| > b_0$ there exist a finite set  $\mathcal{J}(b)$ and a family of operators 
$$ \mathcal{N}_{J, a}: C(\hat U) \rightarrow C(\hat U) \mbox{ for } J \in \mathcal{J}(b)$$
 with the properties that:
\begin{enumerate}
\item the operators $\mathcal{N}_{J, a}$ preserve $K_{E|b|}(\hat U)$;
\item we have $\int_{\hat U} |\mathcal{N}_{J, a}h|^2 d\nu \leq (1 - \epsilon) \int_{\hat U} |h|^2 d\nu $ for all $h \in K_{E|b|}(\hat U)$;
\item if $h\in K_{E|b|}(\hat U)$ and $H \in C(\hat U, \mathbb{C}^{F_q})$ satisfy 
$$  |H(u)| \leq h(u)  \text{ and } |H(u) - H(u')| \leq E|b| h(u)D(u, u') $$
for all $u, u' \in \hat U$, then there exists $J \in \mathcal{J}(b)$ such that 
\begin{itemize} \item $ |\hat{\mathcal{M}}^N_{ab, q} H| \leq\mathcal{N}_{J, a} h;$ 
\item for all $u, u' \in \hat U$,
\[  |\hat{\mathcal{M}}^N_{ab, q} H(u) - \hat{\mathcal{M}}^N_{ab, q} H(u')| \leq E|b| (\mathcal{N}_{J, a}h)(u)D(u, u') . \]\end{itemize}
\end{enumerate}
\end{thm}
The operators $\mathcal{N}_{J, a}$ are called Dolgopyat operators. Before moving on we indicate how to deduce Theorem \ref{Dargtheorem} from Theorem \ref{dolgopyatoperatorsexist}.
\begin{proof}[Proof that Theorem \ref{dolgopyatoperatorsexist} implies Theorem \ref{Dargtheorem}] Choose $N \in \mathbb{N}$, $\epsilon, |a|< a_0,|b| > b_0, E,$ and $H$ as in Theorem \ref{dolgopyatoperatorsexist} and set $h_0$ to be the constant function $||H||_{1, b}$. Theorem \ref{dolgopyatoperatorsexist} allows us to inductively construct sequences $J_l \in \mathcal{J}(b)$, and $h_l\in K_{E|b|}(\hat U)$ such that
\begin{enumerate}
\item $h_{l+1} =\mathcal{N}_{J_l, a} h_l $,
\item $|\hat{\mathcal{M}}^{lN}_{ab, q} H(u)| \leq h_l(u) $ pointwise, and
\item $||\hat{\mathcal{M}}^{lN}_{ab, q}H||_2  \leq||h_l||_2 \leq (1 - \epsilon)^l ||H||_{1, b}.$
\end{enumerate}
Now choose $\e' > 0$ such that $(1 - \epsilon')^N = (1 - \epsilon)$. There is a uniform upper bound, say $R_0 > 1$, on the $L^2(\nu)$ operator norm of $\hat{\mathcal{M}}_{ab, q}$, valid for all $b$ and all $|a| < a_0$. For any $m = lN + r$, with $r < N$, we have 
\begin{eqnarray*} ||\hat{\mathcal{M}}^m_{ab, q} H||_2  &=& \left( \int_{\hat U} |\hat{\mathcal{M}}^{r}_{ab, q} \hat{\mathcal{M}}^{lN}_{ab, q} H(u)|^2d\nu \right)^{1/2}\\
&\leq& R_0^r \left( \int_{\hat U}| \hat{\mathcal{M}}^{lN}_{ab, q}H|^2 d\nu\right)^{1/2}\\
&\leq& R_0^r (1 - \epsilon)^l ||H||_{1, b} \\
&\leq&R_0^r(1 - \epsilon')^{lN} ||H||_{1, b}\\
&\leq& {R_0^N}(1 - \epsilon')^{m-N} ||H||_{1, b}.\end{eqnarray*} This proves the claim.
\end{proof}

\subsection{Notation and constants} \label{ss3.1}
We fix notations and constants that will be needed later on. From hyperbolicity properties of the map $\Pp$, we obtain constants $c_0 \in (0, 1), \kappa_1 >  \kappa>1 $, such that for all $n \in \mathbb{N}$,
\be \label{definegamma}  c_0\kappa^nd(u, u')  \leq d(\Pp^n u, \Pp^n u') \leq c_0^{-1} { \kappa_1^n } d(u, u')\ee
for all $u, u' \in \hat U_i$ both contained in some cylinder of length $n$. Note that this implies a similar estimate for $D$:
\be \label{definegammaD}  c_0\kappa^nD(u, u')  \leq D(\Pp^n u, \Pp^n u') \leq c_0^{-1} { \kappa_1^n } D(u, u')\ee
for all $u, u' \in \hat U_i$ both contained in some cylinder of length $n$. Fix $0<a_0' <0.1$. The functions $\tau$ and $h_0$, and hence $ f^{(a)}$, are not $d$-Lipschitz globally, but they are \textit{essentially} $d$-Lipschitz in the following sense; there exists $0<T_0<\infty$ such that 
 \be \label{defineT} T_0\ge  \max_{|a| \leq a_0'} \left\lbrace ||f^{(a)}||_{\infty} \right\rbrace   + ||\tau||_{\infty},\ee
 and
 \be T_0 \geq \frac{|f^{(a)} (u) - f^{(a)}(u')| + |\tau(u) - \tau(u')|}{d(u, u')}  \ee
 for all $|a| < a_0'$ and all $u, u'$ both contained in the same cylinder of length $1$. The following lemma follows from the Markov property. 
 \begin{lem} Suppose that $ \cyl[i_0 ,\ldots ,  i_N]$ is a non-empty cylinder. The map $\hat \sigma^n :\cyl[i_0,  \ldots ,  i_N] \rightarrow \cyl[i_n, \ldots ,i_N]$ is a bi-Lipschitz homeomorphsim. Moreover any section $v$ of $\hat \sigma ^n$ whose image contains $ \cyl[i_0 ,\ldots , i_N]$ restricts to a bi-Lipschitz homeomorphsm $\cyl[i_n, \ldots ,i_N] \rightarrow \cyl[i_0, \ldots ,i_N]$.   \end{lem}
The proof is omitted for brevity.
  We choose a small $r_0 > 0$ and $z_i \in \hat U_i$ such that $2r_0 < \min_i(\diam(U_i))$ and $z_iN^+_{r_0} \cap \Omega \subset \mbox{int}^u(U_i)$ for each $i$ (here again $\Omega$ denotes the support of the BMS measure). 
   We fix $C_1 >0$ and $\rho_1 > 0$ to satisfy the following lemma:
\begin{lem} \cite[Lemma 3.2]{St}  \label{lengthsizelemma}There exist $C_1 > 0$ and $ \rho_1 > 0$ such that, for any cylinder $\mathsf{C}[i]$ of length $m$, we have 
\[ {c_0r_0}{\kappa_1^{-m}} \leq \diam(\mathsf{C}[i]) \leq C_1\rho_1^m.\]
\end{lem}
We also fix  $p_0 \in \mathbb{N}$ and $ \rho \in (0, 1)$ to satisfy the following proposition:
\begin{prop} \cite[Proposition 3.3]{St} \label{smallercylindersaresmaller}
There exist  $p_0 \in \mathbb{N}$ and $ \rho \in (0, 1)$ such that, for any $n$, any cylinder $\mathsf{C}[i] $ of length $n$ and any sub-cylinders $\mathsf{C}[i'] , \mathsf{C}[i'']$ of length $(n + 1)$ and
$ (n + p_0)$ respectively, we have 
$$  \diam(\mathsf{C}[i'']) \leq \rho \diam(\mathsf{C}[i]) \leq \diam(\mathsf{C}[i']).$$
\end{prop}
Choose also $p_1 > 1$ such that \be \label{definep_1} {1}/{4} \leq {1}/{2} - 2\rho^{p_1 - 1}.\ee 

\begin{fact} \label{propertiesofGibbsmeasures} It follows from a property of an equilibrium state  and the fact that $\mbox{Pr}_\sigma(-\delta\tau) = 0$ that there is a constant
 $0<c_1 <1$ such that for any $m\in \N$,
\[ c_1 {e^{-\delta \tau_m(y)} } \leq {\nu(\mathsf{C}[i])} \leq {c_1}^{-1} {e^{-\delta \tau_m(y)} } \]
for any cylinder $\mathsf{C}[i]$ of length $m$ and any $y \in \mathsf{C}[i]$. 
\end{fact}

Now we need to recall some consequences of non-joint-integrability of the $N^+, N^-$ foliations. 


\begin{lem}[Main Lemma of \cite{St}] \label{SML} There exist  $ n_1\in \mathbb{N}, \delta_0 \in (0, 1)$, a non-empty subset $U_0 \subset U_1$
 which is a finite union of cylinders of length $n_1\ge 1$, and $z_0 \in U_0$ such that, setting $\mathcal{U} = \sigma^{n_1}(U_0)$,
 $\mathcal{U}$ is dense in $U$ 
and that for any $N > n_1$, \begin{enumerate}
\item there exist Lipschitz sections $v_1, v_2: U \rightarrow U$ such that $\sigma^N(v_i(x)) = x$ for all $x\in \mathcal{U}$, and $v_i(\mathcal{U})$ is a finite union of 
open cylinders of length $N$;
\item $\overline{v_1(U)} \cap \overline{v_{2}(U)} = \emptyset$;
\item  for all $s\in\R$ such that $z_0 n^+_s \in U_0$, all $ 0 < |t| < \delta_0$ with $z_0n^+_{s + t} \in U_0 \cap \Omega$, we have 
\begin{multline*}
\frac{1}{t} | ( \tau_N \circ v_2 \circ \Pp^{n_1} - \tau_N \circ v_1 \circ \Pp^{n_1})(z_0n^+_{t + s} ) -\\
 ( \tau_N \circ v_2 \circ \Pp^{n_1} - \tau_N \circ v_1 \circ \Pp^{n_1}) (z_0n^+_{ s} ) |
   \geq \frac{\delta_0}{2} \end{multline*}
    (see \eqref{definenpm} for other notation). 
\end{enumerate}
\end{lem}

The next step is to establish certain a priori bounds on the transfer operators. Fix notation as in the previous subsection and choose 
\be A_0 >  {2}{c_0}^{-1} e^{\frac{T_0}{c_0 (\kappa - 1)}} \max \left\lbrace 1, \frac{T_0}{\kappa - 1} \right\rbrace. \label{defineA0}\ee
\begin{lem}\label{l3.7} For all $a\in \mathbb{R}$ with $|a| < a_0'$ as in \eqref{defineT} and all $|b| > 1$, the following hold:
\begin{itemize}
\item if $h \in K_B(\hat U)$ for some $B>0$, then 
\[ \left|  \frac{\hat{\mathcal{L}}^m_{a0}h(u) - \hat{\mathcal{L}}^m_{a0}h(u')   }{  \hat{\mathcal{L}}^m_{a0}h(u') }  \right| \leq A_0 \left[ \frac{B}{\kappa^m} + \frac{T_0}{\kappa - 1} \right] D(u, u')\]
for all $m \geq 0$ and for all $u, u' \in \hat U_i$ for some $i$;
\item if the functions $0<h\in C(\hat U), H \in C(\hat U, \mathbb{C}^{F_q})$ and the constant $B > 0$ are such that
\[|H(v) - H(v') | \leq Bh(v')D(v, v')\]
whenever $v, v'\in \hat U_i$ for some $i$,
 then for any $m \in \N$ and  any  $|b| > 1$,
\[ |\hat{\mathcal{M}}^m_{ab, q}H(u) - \hat{\mathcal{M}}^m_{ab, q}H(u')|  \leq A_0 \left[ \frac{B}{\kappa^m}\hat{\mathcal{L}}^m_{a0}h(u') + |b|(\hat{\mathcal{L}}_{a0}^m|H|(u')) \right]D(u, u')  \]
whenever $u, u' \in \hat U_i$ for some $i$.
\end{itemize}
\end{lem}

\begin{proof} The first part is essentially proved in \cite{St}. We concentrate on the second claim. Let $u, u' \in \hat U_i$ for some $i$ and let $m > 0$ be an integer. Given $v \in \hat U$ with $\Pp^m v = u$, let $\mathsf{C}[i_0, \ldots ,i_m]$ be the cylinder of length $m$ containing $v$. Note that $i_m = i$ and that $\Pp^m \mathsf{C}[i_0, \ldots ,i_m] = \hat U_i$ by the Markov property. Moreover we know that $\Pp^m: \mathsf{C}[i_0, \ldots , i_m] \rightarrow \hat U_i$ is a homeomorphism, so there exists  $v' = v'(v)$ with $\Pp^m v' = u'$. We therefore have $$d(\Pp ^ j v',\Pp ^j v) \leq \frac{1}{c_0 \kappa ^{m-j}}d(u, u')$$ and so
\begin{eqnarray*} |f_m^{(a)}(v) - f_m^{(a)}(v') | &\leq&  \sum_{j=0}^{m-1} | f^{(a)}(\Pp^j v) - f^{(a)}(\Pp^jv') | \\
 &\leq&  \sum_{j=0}^{m-1}  ||f^{(a)} ||_{\Lip(d)}\frac{ D(u, u')}{c_0\kappa^{m - j}}  \\
&\leq& \frac{T_0}{c_0(\kappa - 1)}D(u, u'). \end{eqnarray*}
A similar estimate holds for $|\tau_m(v'(v)) - \tau_m(v)|$ by a similar calculation. In particular 
\be \label{usefuleqin3} e^{f_m^{(a)}(v)} \leq c_0A_0 e^{f_m^{(a)}(v'(v))},\ee
and 
\begin{align} \notag |e&^{(f_m^{(a)} + ib\tau_m)(v) - (f_m^{(a)} + ib\tau_m)(v'(v))} - 1|\\
\notag &\leq e^{|f_m^{(a)}(v) - f_m^{(a)}(v')  |} |(f_m^{(a)} + ib\tau_m)(v) - (f_m^{(a)} + ib\tau_m)(v'(v))|  \\
  \label{usefuleqin3b} &\leq  |b|A_0D(u, u') .   \end{align}

\begin{remark}This type of estimate will be used repeatedly for the rest of the paper, often with little comment. \end{remark}
Recall that $\coc_m(v'(v)) = \coc_m(v)$ by Lemma \ref{cocycleexists}. Using the fact that the diameter of $\hat U_i$ is bounded above by $1$, we now compute
\begin{align*} &|\hat{\mathcal{M}}_{ab, q}^mH(u) - \hat{\mathcal{M}}_{ab, q}^mH(u')| \\&\leq  \sum_{\Pp^m v = u}\left|e^{(f^{(a)}_m-ib\tau_m)(v)} H(v) - e^{(f^{(a)}_m - ib\tau_m)(v'(v))} H(v'(v))\right| \\
&\leq\sum_{\Pp^m v = u}e^{f^{(a)}_m(v)} |H(v) -H(v'(v))| \\ &+ \sum_{\Pp^m v = u}\left| e^{(f^{(a)}_m - ib\tau_m)(v)} - e^{(f^{(a)}_m - ib\tau_m)(v'(v))}\right|\cdot  | H(v'(v))|\\
&\leq\sum_{\Pp^m v = u}e^{f^{(a)}_m(v)}Bh(v'(v))D(v, v'(v)) \\ &+ \sum_{\Pp^m v = u}e^{f^{(a)}_m(v'(v))} \left|e^{(f_m^{(a)} + ib\tau_m)(v) - (f_m^{(a)} + ib\tau_m)(v'(v))} - 1\right| \cdot |H(v'(v)) |\\
&\leq c_0A_0BD(v, v'(v))  \sum_{\Pp^m v = u}e^{f^{(a)}_m(v'(v))}h(v'(v))\\ &+  |b|A_0D(u, u')\sum_{\Pp^m v = u}e^{f^{(a)}_m(v'(v))}   |H(v'(v)) |\end{align*}
by \eqref{usefuleqin3} and \eqref{usefuleqin3b}. By definitions and \eqref{definegammaD} this then yields
\begin{eqnarray*} &&|\hat{\mathcal{M}}_{ab, q}^mH(u) - \hat{\mathcal{M}}_{ab, q}^mH(u')|   \\
&\leq& \frac{A_0BD(u, u') }{\kappa^m} \hat{\mathcal{L}}_{a0}h(u')+|b|A_0D(u, u') \hat{\mathcal{L}}_{a0}|H|(u') \\
&\leq& A_0 \left( \frac{B}{\kappa^m}\hat{\mathcal{L}}_{a0}h(u') +|b|\hat{\mathcal{L}}_{a0}|H|(u')\right) D(u, u')\end{eqnarray*} 
as expected.
\end{proof}

\subsection{Construction of Dolgopyat operators.} We now recall the construction of Dolgopyat operators. Their definitions rely on a number of constants, which we now fix. The meanings of these constants will become clear throughout the rest of the section. Choose
\begin{equation} \label{chooseE} E > \max \left\lbrace  \frac{2A_0T_0}{\kappa - 1}, 4A_0, 1\right  \rbrace; \end{equation}
 \begin{equation}  \label{ChooseN}   N> n_1 \mbox{ such that }\kappa^N > \max \left\lbrace \frac{E}{4c_0} , 6A_0,  \frac{512\kappa_1^{n_1}E}{c_0^2\delta_0 \rho} ,   \frac{200\kappa_1^{n_1}A_0}{c_0^2}  \right\rbrace;\end{equation} 
  \begin{equation} \label{chooseepsilon1} \epsilon _1 <  \min \left\lbrace \frac{c_0^2(\kappa - 1)}{16T_0\kappa_1^{n_1}},  \frac{c_0r_0}{\kappa_1^{n_1}}, \frac{\delta_0}{2}\right\rbrace;  \end{equation}
  \begin{equation} \label{choosemu}  \mu < \min \left( \frac{1}{4}, \frac{c_0^2\rho^{p_0, p_1 + 2}\epsilon_1}{4\kappa_1^N}, \frac{c_2^2\epsilon_1^2}{256}\right);\end{equation}
  where $A_0$ is given in \eqref{defineA0}, and other constants are as in subsection \ref{ss3.1}.  
 Moreover set
 \[ b_0 = 1.\]



For the rest of this subsection, we fix $|b|>b_0$. Let $$\lbrace C_m:=C_m(b)\}$$ be the family of maximal closed cylinders contained in $\overline{U_0}$ (see Lemma \ref{SML}) with $\operatorname{diam}(C_m) \leq \epsilon_1/|b|$. As a consequence of \eqref{chooseepsilon1} and Lemma \ref{lengthsizelemma} we have:
\begin{lem} \label{lengthcm} Each of the cylinders $C_m$ has length at least $n_1 + 1$. \end{lem}
\begin{cor}\label{previous} Let $v_1, v_2$ be the sections for $\Pp^N$ constructed by Lemma \ref{SML}. If $u, u' \in C_m \cap \hat U$, then $\mathsf{c}_N(v_i(\Pp^{n_1}u)) = \mathsf{c}_N(v_i(\Pp^{n_1}u'))$ for $i = 1, 2$. \end{cor}
\begin{proof} Choose $u, u' \in C_m \cap \hat U$. They are both contained in some cylinder of length $n_1 + 1$. Thus $\Pp^{n_1} u, \Pp^{n_1}u'$ are both contained in some cylinder of length $1$. But then $v_i(\Pp^{n_1} u), v_i(\Pp^{n_1} u')$ are both contained in the same cylinder of length $N$ by Lemma \ref{SML}. The result then follows by Lemma \ref{cocycleexists}. \end{proof}
\begin{Nota} We set  $\coc^{(m)}_i= \coc_N(v_i(\Pp^{n_1}u))\in \G$ for any $u \in C_m$; this is well defined by Corollary \ref{previous}. \end{Nota}

Let $\lbrace D_j:=D_j(b): j = 1, \ldots , p \rbrace$ be the collection of sub cylinders of the $C_m$ of $\mbox{length}(C_m) + p_0p_1$. We will say that $D_j, D_{j'}$ are adjacent if they are both contained in the same $C_m$.  We set  
\[  \Xi(b) := \lbrace 1, 2 \rbrace \times \lbrace 1 ,\ldots , p(b) \rbrace,\]
$$\hat D_j := D_j \cap \hat U,\quad  Z_j :=\overline{\sigma^{n_1}(\hat D_j)}, \quad \hat Z_j := Z_j \cap \hat U$$
and
$$ X_{i, j} := \overline{ v_i (\hat Z_j )},\quad \hat X_{i, j} := X_{i, j} \cap \hat U $$
for each  $i \in \lbrace 1, 2 \rbrace$ and
$j \in \lbrace 1 ,\ldots , p\rbrace$. For $J \subset \Xi(b) $, we define
$ \beta_J: C(\hat U) \rightarrow \mathbb{R} $ by $$ \beta_J = 1 - \mu \sum_{(i, j) \in J}w_{i, j} $$
 where $w_{i j}$ is the indicator function of $X_{i, j}$. We recall a number of consequences of the constructions above:
\begin{enumerate}
\item Each cylinder $C_m$ is contained in some $U_n$ and has diameter at least $\rho \epsilon_1 / |b|$; apply Lemma \ref{lengthcm} and Proposition \ref{smallercylindersaresmaller}. 
\item $\rho^{p_0p_1 + 1} \frac{\epsilon_1}{|b|} \leq \operatorname{diam}(D_j) \leq \rho^{p_1} \frac{\epsilon_1}{|b|} $; this follows from the definition of $D_j$ and Proposition \ref{smallercylindersaresmaller}.
\item The sections $v_i$ are $d$-Lipschitz on each $\hat {U}_i$, with Lipschitz constant no larger than $\frac{1}{c_0\kappa^N}$; this follows from \eqref{definegamma}.
 \item The sets $\hat{X}_{i, j}$ are pairwise disjoint cylinders with diameters 
 \be \label{boundXdiam} \frac{c_0^2 \epsilon_1\rho^{p_0p_1 + 1} \kappa^{n_1}  }{ \kappa_1^{N}|b|} \leq \diam(\hat X_{i,j}) \leq \frac{\epsilon_1\rho^{p_1} \kappa_1^{n_1}  }{ c_0^2 \kappa^{N}|b|};\ee apply the previous two comments and \eqref{definegamma}.
 \item The function $\beta_J$ is $D$-Lipschitz on $\hat U$ with Lipschitz constant
 \begin{equation} \frac{\mu \kappa_1^{N}|b|}{c_0^2 \epsilon_1\rho^{p_0p_1 + 1}\kappa^{n_1}}  \label{boundforbeta}; \end{equation}
 this follows from the previous comment and the definition \eqref{defineD} of the metric $D$. 
 \item If $u, u' \in \Pp^{n_1}(C_m) \cap \hat U$ for some $m$,  then 
 \be \label{wj} D(v_i (u), v_i(u')) \leq \frac{\epsilon_1\kappa_1^{n_1}}{c_0^2|b|\kappa^{N}} \text{  for all $i \in \lbrace 1, 2\rbrace$}
 \ee see the definition of $C_m$ and \eqref{definegamma}.
\item If $u', u'' \in  \Pp^{n_1} (C_m) \cap \hat U$, then
\begin{align}   |b|\cdot\left|\left(\tau_N(v_2(u')) - \tau_N(v_1(u'))\right)-  \left(\tau_N(v_2(u'')) - \tau_N(v_1(u'')) \right) \right|  \leq \tfrac{1}{8}; \label{oneeighth} \end{align}
this follows from the definition of $C_m$, the choice \eqref{chooseepsilon1} of $\epsilon_1$, and \eqref{definegamma}. 
\end{enumerate}
Our next lemma, a simple special case of \cite[Lemma 5.9]{St} encapsulates the essential output of non-integrability for our argument. It is deduced from Lemma \ref{SML}.
\begin{lem} \label{largeangle}  For any $C_m$, there exist $D_{j'}, D_{j''} \subset C_m$ such that 
\begin{equation} |b|\cdot\left|\left(\tau_N(v_2(u')) - \tau_N(v_1(u'))\right)-  \left(\tau_N(v_2(u'')) - \tau_N(v_1(u'')) \right) \right| \\ \geq \tfrac{\epsilon_1 \delta_0\rho}{16}  \end{equation}
for all $u' \in \hat Z_{j'}$ and $u'' \in \hat Z_{j''}$. 
\end{lem}
\begin{proof} Fix $m$ and choose $v_0', v_0'', j', j''$ such that $v_0' \in D_{j'} \subset C_m $, and $v_0'' \in D_{j''} \subset C_m$ with $d(v_0', v_0'') > \frac{1}{2} \diam(C_m) $. For any $v' \in D_{j'}$ and $v'' \in D_{j''}$, we have 
\begin{equation*} d(v', v'') \geq d(v_0', v_0'')  - \diam(D_{j'}) - \diam(D_{j''}) \geq \frac{\epsilon_1\rho}{|b|} \left( \tfrac{1}{2} - 2\rho^{p_1 - 1} \right) \geq \tfrac{\epsilon_1\rho}{4}\end{equation*}
by \eqref{definep_1}. Now we recall $z_0$ as in Lemma \ref{SML} and choose $s_1, s_2 \in (-\alpha, \alpha)$ such that $v' = z_0n^+_{s_1}$ and $v'' = z_0n^+_{s_2}$. Thus $|s_1 - s_2| \geq \frac{\epsilon_1\rho}{8}$ by \eqref{boundaalpha}. On the other hand
\[ \tfrac{\delta_0}{2} \geq \diam(C_m) \geq d( v', v'') \geq \tfrac{1}{2} |s_1 - s_2| \] 
by \eqref{chooseepsilon1}  and \eqref{boundaalpha}; the result follows by Lemma \ref{SML} part 3. 
\end{proof}



We are finally in a position to give the definition of our Dolgopyat operators. For $|b| > b_0, |a| < a'_0$ and for each $J\subset \Xi(b) $, we define an operator $\mathcal{N}_{J, a}:C(\hat U) \to  C(\hat U)$ by
$$\mathcal{N}_{J, a}(h):=\hat{\mathcal{L}}_{a0}^N(\beta_Jh).$$

\subsection{Vector valued transfer operators and Dolgopyat operators.} We will now check that appropriate operators $\mathcal{N}_{J, a}$ satisfy the conditions of Theorem \ref{dolgopyatoperatorsexist}. First choose the subsets $J\in \Xi$ that will be of interest.
\begin{dfn} \label{definedense} \rm A subset $J \subset \Xi(b)$ will be called dense if for every $C_m$, there exists $(i, j) \in J$ with $D_j \subset C_m$. We write $\mathcal{J}(b)$ for the collection of all dense subsets of $\Xi(b)$.
\end{dfn}
The following proves parts 1 and 2 of Theorem \ref{dolgopyatoperatorsexist}.

\begin{lem}\label{sl5.6} There exist $a_0 \in (0, a_0')$ and $ \epsilon > 0$ such that for any $|a| < a_0$ and $|b| > b_0$, the family of operators $\lbrace \mathcal{N}_{J, a}: J\subset \mathcal{J}(b) \rbrace$ satisfies:
\begin{enumerate}
\item $\mathcal{N}_{J, a} h\in  K_{E|b|}(\hat U) $ whenever $h \in  K_{E|b|}(\hat U);$ 
\item $\int_U |\mathcal{N}_{J, a} h|^2 d\nu \leq (1 - \epsilon) \int_U |h|^2 d\nu $ for all $h \in K_{E|b|}(\hat U)$;
\item if $H\in C_D(\hat U, \mathbb{C}^{F_q})$ and $h \in K_{E|b|}(\hat U)$ are such that $|H| \leq h$ and 
\[ | H(v) - H(v')| \leq E|b|h(v')D(v, v'), \]
then
\[ |\hat{\mathcal{M}}^{N}_{ab, q} H(v) - \hat{\mathcal{M}}^{N}_{ab, q}  H(v')| \leq E|b|(\mathcal{N}_{J, a} h)(v')D(v, v')\]
where $N$ is given as in \eqref{ChooseN}. 
\end{enumerate}
\end{lem}
\begin{proof} The second part is Lemma 5.8 of \cite{St}; although that paper uses a differently normalized transfer operator, the error is at most a factor  $\sup_{|a| \leq  a_0}\frac{\sup h_a}{\inf h_a}$, which can be absorbed into the decay term for $a_0$ sufficiently small. The other parts are contained in the same paper for complex valued functions; we include the argument for completeness.  Suppose that $h \in K_{E|b|}(\hat U)$, and that $u, u' \in \hat U$. We compute
 \begin{eqnarray*} \left| h\beta_J(u) - h\beta_J(u')\right| &\leq& |h(u) - h(u') | + h(u') |\beta_J(u) - \beta_J(u') |.\end{eqnarray*}
Thus, recalling \eqref{boundforbeta};
\begin{eqnarray*}  |h(u) - h(u') | + h(u') |\beta_J(u) - \beta_J(u')| &\leq& |b| D(u, u') h(u')( E +  \frac{\mu \kappa^N}{c_0\rho^{p_0p_1 + 1}}) \\
  &\leq& |b| D(u, u') h(u')( E +  \tfrac{\rho }{4})
   \end{eqnarray*} 
 by \eqref{choosemu}.  It follows that 
   \[  h\beta_J\in K_{(E + \frac{\rho}{4})|b|/(1 - \mu)}(\hat U).\]
    We may now apply Lemma \ref{l3.7} above to give
 \begin{eqnarray*}\left|  \mathcal{N}_{J, a}h(u) - \mathcal{N}_{J, a}h(u')\right| &=& \left| \mathcal{L}^N_{a0}(h\beta_{J})(u) - \mathcal{L}_{a0}^N (h\beta_{J}) (u')  \right| \\
 &\leq& A_0\left( \frac{(E + \rho/4)|b|}{\kappa^N(1 - \mu)} + \frac{T_0}{\kappa - 1} \right)D(u, u')\mathcal{L}_{a0}^N (h\beta_{J}) (u')  \\
  &\leq& A_0 \left( 2\frac{E|b|}{\kappa^N} + \frac{T_0}{\kappa - 1} \right)D(u, u')\mathcal{N}_{J, a}h(u')  \\
  &\leq& E|b| D(u, u')\mathcal{N}_{J, a}h(u') 
  \end{eqnarray*}
as required. The final part also follows as a direct calculation using Lemma \ref{l3.7}

\end{proof}

Our final task for this section is to prove the following key proposition, which completes the proof of Theorem \ref{dolgopyatoperatorsexist} by addressing part 3. 

\begin{prop}\label{l3.9}  There exists $a_0>0$ with the following property. For any $h, H$ as in Theorem \ref{dolgopyatoperatorsexist}, any $|a| < a_0$, and any $|b| > b_0$ there exists $J \in \mathcal{J}(b)$ such that for all $v\in\hat U$,
\[ |\hat{\mathcal{M}}^N_{ab, q} H(v)| \leq \mathcal{N}_{J, a} h(v) \]
where $N$ is given as in \eqref{ChooseN}. 
\end{prop}
We proceed via a series of lemmas. 
\begin{lem}\label{alternatives} For $|b|>b_0$, functions $h, H$ as in Theorem \ref{dolgopyatoperatorsexist}, we have, for any  $(i, j) \in \Xi(b)$, 
\begin{enumerate}\item  
\[  \frac{1}{2} \leq \frac{ h(v_i(u')) }{h(v_i(u))} \leq 2 \quad \mbox{  for all } u, u' \in \hat Z_j;\]
\item
either $ |H(v_i(u))| \leq \tfrac{3}{4} h(v_i(u)) $ for all $u \in \hat Z_j$ or
$ |H(v_i(u))| \geq \tfrac{1}{4} h(v_i(u)) $ for all  $u \in \hat Z_j$.  
\end{enumerate}
\end{lem}
\begin{proof}
For $h \in K_{E|b|}(\hat U)$ and $u, u' \in \hat Z_j$ with $D_j \subset C_m$, we simply calculate
\begin{eqnarray*} h(v_i(u')) &\leq& h(v_i(u)) + E|b|D(v_i(u'), v_i(u))h(v_i(u))\\
 &\leq& h(v_i(u)) \left( 1 + E|b| \diam(\hat X_{i, j} ) \right)\\
 &\leq& 2h(v_i(u) ) \end{eqnarray*}
by \eqref{boundXdiam} and \eqref{ChooseN}. The other bound follows by symmetry. The second part of the  observation follows by similar calculations, which we shall omit. 
\end{proof}

 \begin{dfn}\rm Let $a \in(0, a_0')$ and choose  $|b| >b_0, h, H$ as in Theorem \ref{dolgopyatoperatorsexist}. For each fixed $C_m = C_m(b)$,
  recall that $$\coc_i^{(m)} = \coc_N(v_i(\Pp^{n_1}(u))) \quad\text{for all } u \in C_m. $$
 Define the  functions
 \[ \chi^{(1)}[H, h](u) := \frac{\left| e^{(f^{(a)}_N + ib\tau_N)(v_1(u))}H(v_1(u))\coc_1^{(m)} + e^{(f^{(a)}_N + ib\tau_N)(v_2(u))}H(v_2(u)) \coc_2^{(m)} \right| }{  (1 - \mu) e^{f^{(a)}_N(v_1(u))}h(v_1(u)) + e^{f^{(a)}_N(v_2(u))}h(v_2(u)) } \]
and   \[ \chi^{(2)}[H, h](u) := \frac{\left| e^{(f^{(a)}_N + ib\tau_N)(v_1(u))}H(v_1(u)) \coc_1^{(m)} + e^{(f^{(a)}_N + ib\tau_N)(v_2(u))}H(v_2(u))\coc_2^{(m)}  \right| }{  e^{f^{(a)}_N(v_1(u))}h(v_1(u)) +  (1 - \mu) e^{f^{(a)}_N(v_2(u))}h(v_2(u)) }. \]
 \end{dfn}
We claim the following:

\begin{lem}\label{isdense}
 For every $C_m$, there exist $i \in \lbrace 1, 2 \rbrace$ and $ j  \in \lbrace 1 \ldots p \rbrace$ such that $D_{j} \subset C_m$ and $\chi^{(i)}[H, h](u) \leq 1 $ for all $u \in \hat Z_{j}$. 
\end{lem}
\begin{proof} Fix $m$ and choose $j', j''$ as in Lemma \ref{largeangle}.  Consider $\hat Z_{j'}$ and $\hat Z_{j''}$. If there exist $t\in \lbrace j', j'' \rbrace$ and $i \in \lbrace 1, 2\rbrace $ such that the first alternative of 
Lemma \ref{alternatives} (2) holds for $\hat Z_t$, then $\chi^{(i)}[H, h](u) \leq 1$ for all $u \in \hat Z_t$. So from now on in this proof we assume the converse, i.e.,
for each $i$, $ |H(v_i(u))| \geq \tfrac{1}{4} h(v_i(u)) $ for all  $u \in \hat Z_{j'}\cup \hat Z_{j''}$.

Consider now $u' \in \hat Z_{j'}$ and $u'' \in \hat Z_{j''}$.  Then the properties of $h$ and $ H$ imply 
\begin{eqnarray*} \frac{|H(v_i(u')) - H(v_i(u''))|}{\min \lbrace | H(v_i(u'))|, | H(v_i(u''))| \rbrace }  &\leq& \frac{E|b|h(v_i(u'))D(v_i(u'), v_i(u''))}{\min \lbrace | H(v_i(u'))|, | H(v_i(u''))| \rbrace}\\
&\leq& 4E|b|D(v_i (u'), v_i(u''))  \\ 
&<&\frac{\epsilon_1\delta_0\rho}{128} \mbox{ by \eqref{wj}},\end{eqnarray*}
where we have assumed $|H(v_i(u'))| \leq |H(v_i(u''))|$ without loss of generality.

In particular this is less than $\frac{1}{2}$. We write $c_2 = \frac{\delta_0 \rho}{16}$. The sine of the angle $\theta_i$ between $H(v_i(u'))$ and $H(v_i(u''))$ is therefore at most $\sin \theta_i\leq \frac{c_2\epsilon_1}{8}$, so 
\be \label{smallangle} \theta_i\leq \frac{c_2\epsilon_1}{4}. \ee
 We need to use this to show that at least one of the angles 
\[ \theta(e^{ib\tau_N(v_1(u')} H(v_1(u'))\coc_1^{(m)} ,  e^{ib\tau_N(v_2(u')} H(v_2(u'))\coc_2^{(m)} )\]
or 
\[\theta( e^{ib\tau_N(v_1(u'')}H(v_1(u''))\coc_1^{(m)} ,  e^{ib\tau_N(v_2(u'')}H(v_2(u''))\coc_2^{(m)} )\]
is greater than $c_2\epsilon_1/4$. Supposing that the first term is less than $c_2\epsilon_1/4$, we will
show that the second term is bigger than  $c_2\epsilon_1/4$.
 Write
 \[ \phi(w) := b \cdot \left( \tau_N(v_2(w)) - \tau_N(v_1(w)) \right)\]
 and note that 
 \[ c_2\e_1 \le  | \phi(u') - \phi(u'') | \leq \tfrac{1}{8}\]
 for all $u' \in \hat {Z}_{j'}$ and all $u'' \in \hat{Z}_{j''}$ by \eqref{oneeighth} and Lemma \ref{largeangle}.
 
We compute
\begin{eqnarray*} &&\theta( e^{ib\tau_N(v_1(u''))}H(v_1(u''))\coc_1^{(m)} ,  e^{ib\tau_N(v_2(u''))}H(v_2(u'')) \coc_2^{(m)} ) \\
&=& \theta( e^{-i\phi(u'')}H(v_1(u''))\coc_1^{(m)} ,  H(v_2(u''))\coc_2^{(m)} )\\
&\geq&  \theta( e^{-i\phi(u'')}H(v_1(u''))\coc_1^{(m)} ,  e^{-i\phi(u')}H(v_1(u''))\coc_1^{(m)} )\\
&&-   \theta( e^{-i\phi(u')}H(v_1(u''))\coc_1^{(m)} ,  H(v_2(u''))\coc_2^{(m)} )\\
 &\geq& | \phi(u')  -  \phi(u'')|  -  \theta( e^{-i\phi(u')}H(v_1(u''))\coc_1^{(m)} ,  H(v_2(u''))\coc_2^{(m)} ) \\
&\geq&
 c_2\epsilon_1 - c_2\e_1/ 2-  \theta( e^{-i\phi(u')}H(v_1(u'))\coc_1^{(m)} ,  H(v_2(u'))\coc_2^{(m)} )\\ &\geq&
 c_2\epsilon_1/4\end{eqnarray*}
  by \eqref{smallangle} and the assumption.
  Write
\[ v = e^{(f^{(a)}_N + ib\tau_N(v_1(u''))}H(v_1(u'')) \coc_1^{(m)} \text{ and} \]
\[ w = e^{(f^{(a)}_N + ib\tau_N)(v_2(u''))}H(v_2(u''))\coc_2^{(m)} \]
so that $|v + w|$ is the numerator of $\chi^{(i)}[H, h](u'') $. 
 Without loss of generality, we assume that $|v| \leq |w|$.
We now claim that $\chi^{(1)}[H, h](u'') \leq 1$ for all $u''\in \hat Z_{j''}$. 
This  now follows from rather simple trigonometry.  Since the angle $\tilde \theta$ between $v$ and $w$ is at least $c_2 \epsilon_1/4$, we have
\[ 1 + 2\cos\tilde \theta \leq 2 + \cos \tilde \theta \leq 3 - \frac{c_2^2\epsilon_1^2}{16} \leq 3(1 - \mu)^2 \leq (1 - \mu)^2 + 2 (1 - \mu). \]
Thus 
\[ |v| + 2|v| \cos\tilde \theta \leq (1 - \mu)^2 |v| + 2(1 - \mu)|v|.\]
Now
\[ |v| + 2 |w|\cos \tilde \theta \leq (1 - \mu)^2 |v| + 2(1 - \mu)|w|,\]
and so $ (1 - \mu)|v| + |w| \geq |v + w|$, and $\chi^{(1)}[H,h] \leq 1$ on $\hat Z_{j''}$ as expected.  
\end{proof}

\begin{proof}[Proof of Proposition \ref{l3.9}] Choose $h, H, |b|  > b_0$ as in the hypotheses of Theorem \ref{dolgopyatoperatorsexist} and choose $a_0 \in (0, a_0')$ to satisfy Lemma \ref{sl5.6}. We choose a subset $J \in \mathcal{J}( b)$ as follows. First include in $J$ all $(1, j) \in \Xi$ such that $\chi^{(1)}[H, h] \leq 1$ on $\hat Z_j$. Then for any $j \in \lbrace 1 \ldots p \rbrace$, include $(2, j)$ in $J$ if $(1, j)$ is not already in $J$ and $ \chi^{(2)}[H, h] \leq 1$ on $\hat Z_j$. By Lemma \ref{isdense},  this subset $J$ is dense (in the sense of Definition \ref{definedense}), so that  $J\in \mathcal J (b)$.
We will show that for all $u\in \hat U$
\[| \hat{\mathcal{M}}^N_{ab, q}H(u) |\leq \mathcal{N}_{J, a}h(u).\] Let $u \in \hat U$. Suppose first that $u \notin \hat{Z}_j$  for any $(i, j) \in J$; then $\beta_J(v) = 1$ whenever $\Pp^N(v) = u$, and the bound follows. Suppose instead that $u \in \hat Z_j\subset C_m $ with $(1, j) \in J$. Then $(2, j)\notin  J$ and so $\beta_J(v_1(u)) \geq 1 - \mu$ and $\beta_J(v_2(u)) = 1$. We therefore have $\chi^{(1)}[H, h] \leq 1$ on $\hat Z_j$, so 
\begin{align*}& | \hat{\mathcal{M}}^N_{ab, q}H(u)| \leq \sum_{\Pp^Nv = u, v\neq v_1(u), v_2(u) }e^{f_N^{(a)}(v)}|H(v)|\\
 &+\left| e^{(f^{(a)}_N + ib\tau_N(v_1(u))}H(v_1(u))\coc_1^{(m)} + e^{(f^{(a)}_N + ib\tau_N)(v_2(u))}H(v_2(u)) \coc_2^{(m)} \right| \\
 &\leq \sum_{\Pp^Nv = u, v\neq v_1(u), v_2(u) }e^{f_N^{(a)}(v)}|h(v)|\\
 &+(1 - \mu) e^{f^{(a)}_N(v_1(u))}h(v_1(u)) + e^{f^{(a)}_N(v_2(u))}h(v_2(u)) \\
 & \leq \mathcal{N}_{J, a}h(u). \end{align*}
 The case $u \in \hat Z_j$ with $(2, j) \in J$ is similar. This finishes the proof.
 \end{proof}
\noindent Together with Lemma \ref{sl5.6}, this completes the proof of Theorem \ref{dolgopyatoperatorsexist}.

\section{The expansion machinery} \label{bbounded}
 \subsection{Some reductions.}
In this section we assume that $\G$ is a convex cocompact subgroup in $\SL_2(\z)$ and that $q_0$ is as in  \eqref{strongapproximation}. Let $b_0>0$ be as in Theorem \ref{Dargtheorem}. The main aim of this Section is to prove the following theorem.
\begin{thm} \label{BGSbound} 
There exist $ \epsilon \in (0, 1), a_0 > 0, C>1, q_0' > 1$ such that for all $|a| < a_0$,  $|b| \leq b_0$, and all square free $q\ge 1$ with $(q,q_0q_0')=1$,
 we have
\[  ||\hat{\mathcal{M}}^m_{ab, q} H||_2 < C(1 - \epsilon)^mq^C ||H||_{\Lip(d)} \]
for all $m\in \mathbb{N}$ and all $H \in  \mathcal{W}(\hat U, \mathbb{C}^{\op{SL}_2(q)})$; see \eqref{defineW} for notation.  
\end{thm}
Since the $\Lip(d)$ norm and the $||\cdot||_{1, b}$ norm are equivalent for  all $|b| \leq b_0$, this theorem and Theorem \ref{Dargtheorem} imply Theorem \ref{spectralbound}. 

 The key ingredient of the proof of Theorem \ref{BGSbound}  is the expander technology, introduced in this context by Bourgain, Gamburd, and Sarnak \cite{BGS}, from which we draw heavily throughout this section. The idea of the expansion machinery is that random walks on the Cayley graphs of $\operatorname{SL}_2(q)$ have good spectral properties. We don't have a random walk in the usual sense, but the randomness inherent in the Gibbs measure provides the same effect.

We recall the sequence spaces $\Sigma^+, \Sigma$, the shift map $\sigma$ and the embedding $\zeta: \Sigma^+ \rightarrow \hat U$. 

\begin{Nota} For any function $H \in C(\hat U, \c^{\SL_2(q)})$, we will denote $\tilde H = H \circ \zeta: \Sigma^+ \rightarrow  \c^{\SL_2(q)}$.  Similarly $\tilde \tau$ will denote $\tau \circ \zeta$. \end{Nota}

We recall the constant $\theta \in (0, 1)$ chosen sufficiently close to one (see subsection \ref{ss2.2}) and the metric $d_\theta$ on $\Sigma$ (resp. on $\Sigma^+$). Write  
$$||\tH||_\infty := \sup_{\omega \in \hat  \Sigma^+} |\tilde H(\omega)| $$
 and
\[ \Lip_{d_\theta}(\tH) := \sup_{\omega \neq \omega' \in \hat \Sigma^+} \frac{ |\tH(\omega) - \tH(\omega')|}{d_\theta(\omega, \omega')} \]
which is the minimal Lipschitz constant of $\tH$. We also write $$||\tH||_{d_\theta} := ||\tH||_\infty + \Lip_{d_\theta}(\tH).$$ We fix the following constant for later convenience
\be \label{defineetatheta} \eta_\theta := \frac{\Lip_{d_\theta}(\tau) + \sup_{|a| < 1}\Lip_{d_\theta}(f^{(a)})  }{1 - \theta}. \ee

Rather than proving Theorem \ref{BGSbound} directly, we will instead start by describing some reductions to a simpler form. For $q'  | q$, we define $\hat E^q_{q'} \subset L^2_0(\SL_2(q))$ to be the space of functions invariant under the left action of $\G(q')$. We may then write
$$E^q_{q'} := \hat E^q_{q'} \cap \left( \oplus_{q' \neq q'' | q'} \hat E^q_{q''}\right) ^\perp.$$ 
We think of $E^q_{q'}$ as the space of {\it new} functions at the level $q'$. We can then define $\tilde E^q_{q'}$ as the subspace of functions $H$ in $ \mathcal{W}(\hat U, \mathbb{C}^{\op{SL}_2(q)})$ with $H(u, \cdot) \in E^q_{q'}$ for all $u$. We recall the orthogonal decomposition 
\[ L^2_0(\SL_2(q)) = \oplus_{1 \neq q'|q} E^q_{q'} \]
and the induced direct sum decomposition
\[  \mathcal{W}(\hat U, \mathbb{C}^{\op{SL}_2(q)}) = \oplus_{1\neq q' | q} \tilde E^q_{q'}. \]
Write 
\[e_{q, q'} : \mathcal{W}(\hat U, \mathbb{C}^{\op{SL}_2(q)}) \rightarrow \tilde E^q_{q'} \]
for the projection operator, and note that $e_{q, q'}$ is norm decreasing for both the $|| \cdot ||_{\Lip(d)}$ norm and the $||\cdot||_2$ norm. 
\begin{remark}
The projection operators commute with the congruence transfer operators: we have
\[ e_{q, q'}  \circ \hat{\mathcal{M}}_{ab, q}=\hat{\mathcal{M}}_{ab, q} \circ  e_{q, q'}   \]
for any $q'|q$. 
\end{remark}
The first reduction is that we only need to consider functions in $\tilde E^q_q$. 
\begin{thm} \label{BGSbound2} 
There exist $\epsilon \in (0, 1), a_0 > 0, C>1, q_1>1$ such that for all $|a| < a_0$, $|b| <  b_0$ and $q\ge q_1$ square free with $(q,q_0)=1$, 
 we have
\begin{equation}||  \hat{\mathcal{M}}^{m}_{ab, q} H||_2 < C(1 - \epsilon)^mq^C ||H||_{\Lip(d)} \end{equation}
for all $m\in \mathbb{N}$ and all $H \in \tilde E^q_q$.  
\end{thm}

\begin{proof}[Proof that Theorem \ref{BGSbound2} implies Theorem \ref{BGSbound}] Set $q_0'$ to be the product of all primes less than or equal to $q_1$. We will first explain how to deduce Theorem \ref{BGSbound} from Theorem \ref{BGSbound2}. Fix $\epsilon, a_0, b_0, C, q_1, q_0$ as in Theorem \ref{BGSbound2}. Fix also $q$ square free such that $(q, q_0q_0') = 1$.  For $q'|q$, we consider the projection maps
\[ \mbox{proj}_{q, q'} :  E^q_{q'} \rightarrow E_{q'}^{q'}\]
by choosing $(\mbox{proj}_{q, q'} F)(\gamma) = F( \tilde \gamma)$, where $\tilde \gamma$ is any pre-image of $\gamma$ under the natural projection map $\SL_2(q) \rightarrow \SL_2(q')$. By abuse of notation we will also write $\mbox{proj}_{q, q'}$ for the induced maps $\tilde{E}^q_{q'} \rightarrow \tilde{E}^{q'}_{q'}$. Write $\spadesuit_{q, q'}:=\frac{\# \SL_2(q')}{\# \SL_2(q)} $. We note that 
\[ || (\mbox{proj}_{q, q'} H)  ||_{\Lip(d)} = \sqrt{\spadesuit_{q, q'}} || H ||_{\Lip(d)},\]
that
\[  || (\mbox{proj}_{q, q'} H)  ||_2 = \sqrt{\spadesuit_{q, q'}} || H ||_2,\]
and that 
\[ \hat{\mathcal{M}}_{ab, q'} \circ \mbox{proj}_{q, q'}   =  \mbox{proj}_{q, q'} \circ   \hat{\mathcal{M}}_{ab, q} .\]
Now consider  $ H \in  \mathcal{W}(\hat U, \mathbb{C}^{\op{SL}_2(q)})$. We calculate, for $|a| < \min(a_0, a_0')$ and $|b| \leq b_0$, 
\begin{eqnarray*} ||\hat{\mathcal{M}}^m_{ab, q} H ||^2_2 &=& \sum_{1 \neq q' | q}   ||e_{q, q'} \hat{\mathcal{M}}^m_{ab, q} H ||^2_2 \\
&=& \sum_{1 \neq q' | q}   ||\hat{\mathcal{M}}^m_{ab, q}( e_{q, q'}  H) ||^2_2\\
&=&  \sum_{1 \neq q' | q}  {\spadesuit_{q, q'}  } || \mbox{proj}_{q, q'}(\hat{\mathcal{M}}^m_{ab, q}( e_{q, q'}  H)) ||^2_2\\ 
&=&  \sum_{1 \neq q' | q}   {\spadesuit_{q, q'}}    ||\hat{\mathcal{M}}^m_{ab, q'}( \mbox{proj}_{q, q'}( e_{q, q'}  H)) ||^2_2.\end{eqnarray*}
Applying Theorem \ref{BGSbound2}, we obtain 
\begin{eqnarray*} ||\hat{\mathcal{M}}^m_{ab, q} H ||^2_2  &\leq&   C^2(1 - \epsilon)^{2m} (q')^{2C}  \sum_{1 \neq q' | q}    \spadesuit_{q, q'}   || \mbox{proj}_{q, q'}( e_{q, q'}  H) ||^2_{\Lip(d)}.\\
&\leq&    C^2(1 - \epsilon)^{2m} \sum_{1 \neq q' | q, q' \geq q_1}    (q')^{2C}  ||e_{q, q'}  H ||^2_{\Lip(d)}\\
&\leq& (C'')^2 (1 - \epsilon'')^{2m}q^{2C'' + 1}|| H ||^2_{\Lip(d)}\end{eqnarray*}
as expected.
\end{proof}

The most convenient formulation to prove will be the following:
\begin{thm}\label{af1}
There exist $\kappa > 0, a_0 > 0, q_1 > 0$ such that 
\[ ||\hat{\mathcal{M}}^{ln_q}_{ab, q} \tH||_2 \leq q^{- l\kappa} ||\tH||_{d_\theta} \]
for all $|a| < a_0, |b| \leq b_0 ,l \in \mathbb{N}$, all $q > q_1$ square free and coprime to $q_0$, and all $H \in \tilde E^q_q$; here $n_q$ denotes the integer part of $\log q$. 
\end{thm}

\begin{proof}[Proof that Theorem \ref{af1} implies Theorem \ref{BGSbound2}] Choose $a_0>0$ small enough that $|\log \lambda_a| \leq \epsilon \leq \min( \kappa/ 2, 1)$ for all $|a| < a_0$. 
Set  $$C: = \max \left\lbrace \log\left(  \sup_{|a| \leq a_0, b\in\br} ||\hat{\mathcal{M}}_{ab, q}||_2\right), 0 \right\rbrace .$$ Then for all
 $0 \leq r < n_q$,  we have  $|| \hat{\mathcal{M}}_{ab, q}||^r_2 \leq q^C$. 

For any $m\in \N$, we write  $m = ln_q + r$, with $0 \leq r < n_q$. Thus Theorem \ref{af1} yields 
\begin{eqnarray*}   ||\hat{\mathcal{M}}^{m}_{ab, q} \tH||_2 &\leq& || \hat{\mathcal{M}}_{ab, q}||^r_2 \cdot ||\hat{\mathcal{M}}^{ln_q}_{ab, q} \tH||_2\\
&\leq& q^C  q^{- l\kappa} ||\tH||_{d_\theta}\\
&\leq& q^C e^{-  ln_q\epsilon} ||\tH||_{d_\theta}\\
&\leq& q^{C+1} e^{-\epsilon m} ||\tH||_{d_\theta},\end{eqnarray*}
as desired. 
\end{proof}

\subsection{The $\ell^2$-flattening lemma} The rest of this section is devoted to a proof of Theorem \ref{af1}.  The key ingredient is a version of the $\ell^2$-flattening lemma \ref{bgsflatteninglemma} of Bourgain-Gamburd-Sarnak \cite[Lemma 7.2]{BGS}. For the rest of this section we will assume that
\[q \mbox{ is square free and coprime to $q_0$ (as in \eqref{strongapproximation})}.\] 

\begin{dfn}\rm For a complex valued measure $\mu$ on $\SL_2(q)$ and $q' | q$, we define $|||\pi_{q'}(\mu)|||_\infty$ to be the maximum weight of $|\mu|$ over all cosets of subgroups of $\operatorname{SL}_2(q')$ that have proper projection in each divisor of $q'$. 
\end{dfn}
\begin{Nota}
For a function $\phi$ and a measure $\mu$ on $\SL_2(q)$, we denote the convolution by 
$$\mu*\phi(g)=\sum_{\gamma \in \SL_2(q)} \mu(\gamma)\phi(g\gamma ^{-1}) .$$
\end{Nota}

\begin{lem} [\cite{BG}, \cite{BGS}] \label{bgsflatteninglemma}\label{flat}
Given $\kappa > 0$ there exist $\kappa' > 0$ and $ C>0$ such that if $\mu$ satisfies $||\mu||_1 \leq B$ and 
\[ |||\pi_{q'}(\mu)|||_\infty < q^{-\kappa}B \mbox{ for  all } q' | q, q' > q^{1/10} \quad\text{for some $B>0$},\]
then for each $q$ and $\phi\in  E^q_q$,
\[ ||\mu*\phi||_2 \leq Cq^{-\kappa'}B||\phi||_2 .\]
\end{lem}

\subsection{Measure estimates on cylinders} Before we can apply the expansion machinery we must first establish certain a priori measure estimates on cylinders; that will be the topic of this subsection.  We define:
\begin{Nota}
\begin{itemize}
\item For $x = (x_1, x_2, \ldots ) \in \Sigma^+$ and a sequence $i_1, \ldots ,  i_n$ of symbols,
 we denote the concatenation by $$(i_n ,\ldots ,i_1, x) = (i_n, \ldots ,i_1, x_1, x_2, \ldots );$$
\item For a function $f$ on $\Sigma^+$ and  $x\in \Sigma^+$, we set 
$$f_n(x) := f(x) + f(\sigma(x)) + \ldots + f(\sigma^{n-1}(x));$$
\item For $x=(x_i)\in \Sigma^+$, put $\coc(x) = \coc(\zeta(x))$, and  $$\coc_n(x) := \coc(\zeta(x))\coc(\zeta(\sigma x)) \ldots \coc(\zeta(\sigma^{n-1}x))\in \G.$$
\end{itemize} \end{Nota}

\begin{lem} \label{cocdep}
For sequences $x, y \in \hat \Sigma^+$ with $x_i = y_i$ for $i = 0\ldots k$ for some $k\ge 1$, we have $\coc_{k}(x) = \coc_{k}(y)$.  
\end{lem}
\begin{proof}
This is a straightforward consequence of Lemma \ref{cocycleexists}. 
\end{proof}

We may therefore write $\coc(x)=\coc(x_0, x_1)$, and more generally,
for $n\ge 2$,
$\coc_n(x)$ is the product $\coc(x_0,x_1)\cdots \coc(x_{n-1}, x_n)$.
\begin{Nota} \rm In the rest of the section, the notation $\sum_{i_1, \cdots, i_\ell}$ means the sum taken over all sequences $(i_1, \cdots, i_\ell)$ such that
any concatenation following the sum sign is admissible.
\end{Nota}

\begin{lem} \label{a_0}There exist $0<a_0<1$ and $c>1$ such that for all $|a|<a_0$, $x\in \hat \Sigma^+$ and for all $n\in \N$,
$$ \sum_{i_n, \cdots, i_1} e^{f^{(a)}_n(i_n, \cdots, i_1, x)}\le c.$$
\end{lem}
\begin{proof}
This follows easily from \eqref{ggibbs}. 
\end{proof}

We recall that for $x, y\in \bH^2$ and $r>0$, the shadow $O_r(x,y)$
is defined to be the set of 
 all points $\xi \in \partial(\bH^2)$ such that the geodesic ray from $x$ to $\xi$ intersects the ball $B_{r}(y)$ non-trivially.
 We need the following: recall the Patterson-Sullivan density $\{\mPS{x}: x\in \bH^2\}$ for $\Gamma$.
\begin{lem}[Sullivan's shadow lemma  \cite{Sullivan1979}] \label{ssl} Let $x\in \bH^2$. There exists $r_0=r_0(x)>1$ such that for all $r>r_0$, there exists $c>1$ such that
for all $\gamma\in \G$,
$$c^{-1} e^{-\delta d(x, \gamma x)} 
 \le \mu_x^{\PS} (O_r(x, \gamma x)) \le c e^{-\delta d(x, \gamma x)} .$$

\end{lem}

\begin{lem}\label{sha} There exists $c'>1$ such that
for any $x\in  \hat \Sigma^+$, $\gamma\in \Gamma$, any $m \in \mathbb{N}$, and any fixed $i_{m + 1}$,  we have
$$ \sum_{i_1 ,\ldots ,i_m, \coc_{m+1}(i_{m+ 1}, i_m, \ldots , i_1, x) = \gamma} e^{f^{(0)}_m(i_m, \ldots ,i_1, x)} \le  c'\cdot  e^{-\delta m\inf (\tau)}.$$
\end{lem}
\begin{proof} Recalling the definition \eqref{normalized} of $f^{(0)}$ in terms of $\tau$, we calculate
 \begin{eqnarray*}&&\sum_{i_1, \ldots , i_m, \coc_{m + 1}(i_{m+ 1}, i_m, \ldots , i_1, x) = \gamma} e^{f^{(0)}_m(i_m, \ldots ,i_1, x)}  \\
 &\leq &\frac{\sup(h_0)}{\inf(h_0)}  \sum_{i_1 ,\ldots ,i_m, \coc_{m + 1}(i_{m+ 1}, i_m, \ldots , i_1, x) = \gamma} e^{-\delta  \tau_m(i_m, \ldots ,i_1, x)} \\
 &\leq&e^{\delta \sup(\tau)}\frac{\sup(h_0)}{\inf(h_0)}   \sum_{i_1, \ldots ,i_m, \coc_{m + 1}(i_{m+ 1}, i_m, \ldots , i_1 ,x) = \gamma} e^{-\delta  \tau_{m+1}(i_{m+1}, \ldots , i_1, x)} \\
 &\leq&c_1'\sum_{i_1, \ldots ,i_m, \coc_{m  + 1}(i_{m+ 1}, i_m, \ldots ,i_1, x) = \gamma}\nu(\mathsf{C}[i_{m+1}, i_m ,\ldots , i_1]) \mbox{ (see Fact \ref{propertiesofGibbsmeasures}}) \end{eqnarray*}
 where $c_1' =c_1 e^{\delta \sup(\tau)}\frac{\sup(h_0)}{\inf(h_0)}  $.
 
  Recall that $\mathcal D$ denotes the intersection of the Dirichlet domain for $(\G, o)$ with the convex core of $\G$, and
   the lifts $\tilde U_i$ of $U_i$ chosen to intersect $\mathcal{D}$. Recall also the projection map $\pi$ from $G$ to $\G \ba G$. It is a consequence of the definition of $\coc$ \eqref{defcoc} that 
  
   \be   \bigcup_{i_m, \ldots ,i_1: c_{ m + 1}(i_{m+1} ,\ldots ,i_1, x) = \gamma}\hspace{-15mm} \mathsf{C}[i_{m+1}, i_m ,\ldots , i_1] \subset \pi \left( \lbrace \tilde u \in \tilde U_{i_{m+1}}: d( \tilde ua_{ \tau_{m + 1}(\tilde u)}, \gamma o ) < R_1 \rbrace \right) \ee
  where $R_1$ denotes thrice the size of the Markov section plus twice the diameter of $\mathcal{D}$ plus the constant $r_0(o)$ defined in Lemma \ref{ssl}. 
  
  \textbf{Case 1:}  If $d(\tilde u, \gamma \tilde u) < (m + 1)\inf (\tau)  - R_1$, then 
 \[\lbrace \tilde u \in \tilde  U_{i_{m+1}}: ua_{\tau_{m + 1}(\tilde u)} \in B_{R_1}(\gamma \tilde u) \rbrace = \emptyset,\]
 and the claim follows.
 
 \textbf{Case 2:} We now assume that $d(\tilde u, \gamma \tilde u) \geq (m + 1)\inf (\tau)  - R_1$.  Then $d(o, \g o) \geq (m + 1) \inf (\tau) - 2R_1$. A straightforward argument in hyperbolic geometry yields
  \begin{eqnarray*} \lbrace \tilde u \in \tilde U_{i_{m+1}}: d( \tilde ua_{ \tau_{m + 1}(\tilde u)}, \gamma o ) < R_1 \rbrace  &\subset& \lbrace  \tilde u \in \tilde U_{i_{m+1}}: \mbox{vis}(\tilde u) \in O_{R_1}(\tilde u, \gamma o)\rbrace  \\
  &\subset&  \lbrace  \tilde u \in \tilde U_{i_{m+1}}: \mbox{vis}(\tilde u) \in O_{2R_1}(o, \gamma o)\rbrace.   \end{eqnarray*}
  Applying Corollary \ref{c2.4} and Lemma \ref{ssl}, we obtain 
      \begin{align*}& \sum_{i_1 ,\ldots , i_m, \coc_{m + 1}(i_{m+ 1}, i_m, \ldots , i_1, x) = \gamma} e^{-\delta  \tau_m(i_m, \ldots , i_1, x)} \\ &\le c_1'\sum_{i_1 , \ldots , i_m, \coc_{m + 1}(i_{m+ 1}, i_m, \ldots , i_1, x) = \gamma}\nu(\mathsf{C}[i_{m+1}, i_m ,\ldots , i_1]) \\
      &\le c_1' \nu \left(  \lbrace \pi \tilde u:   \tilde u \in \tilde U_{i_{m+1}}  \mbox{ and vis}(\tilde u) \in O_{2R_1}(o, \gamma o)\rbrace \right)\\
            &\le c_2' \mPS{o} \left(   O_{2R_1}(o, \gamma o)\right) \\
             &\le c_3' e^{-\delta \inf (\tau)m}
   \end{align*}
  as required. 
\end{proof}

\subsection{Decay estimates for convolutions.} We will now use the cylinder estimates and the expansion machinery to provide technical estimates on the $L^2$ norm of certain convolutions. This is the last preparatory step before we begin the proof of Theorem \ref{af1} in earnest.

We observe that the set $$\mathcal{S} :=\{\pm \coc(x), \pm \coc(x)^{-1} \in \G:x\in \Sigma^+\}$$
is a finite symmetric subset of $\G$.
 
 \begin{lem}
 The set $\mathcal{S}$ generates $\G$. 
 \end{lem}
\begin{proof} By our assumption, the projection $p(\Gamma)$ is a torsion-free
convex cocompact subgroups of $\PSL_2(\br)$, and hence it is a classical Schottky group by \cite{Bu}. 
Therefore the Dirichlet domain $D$ for $(p(\Gamma ), o)$ is the common exterior of a finitely many disks $D_i$, $i=1, \cdots, 2\ell$, which meets
$\partial(\bH^2)$ perpendicularly and
whose closures are pairwise disjoint. 
It is now clear from the definition of the cocycle that $\{\coc (x)\}$ contains all $\gamma\in p(\G)$ such that $\overline D \cap \gamma (\overline D)$
is non-empty, and hence contains a generating set for $p(\Gamma)$. 

\end{proof}

\begin{Nota}\rm  For $m \in \N$, we write $\mathcal{B}_m(e) \subset \G$ for the ball of radius $m$ around the identity $e$ in the word metric defined by $\mathcal{S}$. 

\end{Nota}

\begin{Nota}\rm  We write $n_q$ for the integer part of $\log q$. There exists $d_0>3$  such that for any $m_q \le \frac{1}{d_0}\log q,$ the ball $\mathcal{B}_{m_q}(e)$ injects to $\SL_2(q')$ whenever   $q'|q, q'>q^{1/10}$. For each $q$ we fix a choice
\[ \frac{n_q}{2d_0} < m_q < \frac{n_q}{d_0}\] 
and denote $r_q = n_q - m_q$, so that 
\[ \frac{(d_0 - 1)n_q}{d_0} < r_q < \frac{(d_0 - 1/2)n_q}{d_0}. \]
\end{Nota}
\begin{Nota}\rm For each element $i$ of the alphabet defining  $\Sigma$, we choose an element $\omega(i) \in \hat \Sigma^+$ such that the concatenation $(i, \omega(i))$ is admissible. \end{Nota}

For $\g \in \SL_2(q)$, we write $\delta_{\g}$ for the dirac measure at $\g$. Given real numbers $a, b$, an element $x\in \hat \Sigma^+$, and an admissible sequence $(i_{n_q}, \cdots, i_{m_q+1})$, we define a complex valued measure $ \mu^{a, b}_{x, (i_{n_q}, \cdots, i_{m_q+1})}$ on $\SL_2(q)$  by
\begin{equation} \mu^{a, b}_{x, (i_{n_q}, \cdots, i_{m_q+1})} := \sum_{i_1, \ldots ,i_{m_q}}e^{(f^{(a)}_{n_q} + ib\tau_{n_q})(i_{n_q} ,\ldots , i_1, x)}\delta_{\coc_{m_q + 1}(i_{m_q + 1}, \ldots , i_1, x)}\label{hereismu} . \end{equation}

That is, for $\gamma \in \SL_2(q)$, the value $\mu^{a, b}_{x, (i_{n_q}, \cdots, i_{m_q+1})} (\gamma)$ is given by the sum 
$$ \sum e^{(f^{(a)}_{n_q} + ib\tau_{n_q})(i_{n_q} ,\ldots , i_1, x)}$$ over all indices
$(i_1, \ldots ,i_{m_q})$ satisfying
$ \coc  (i_{m_q + 1}, i_{m_q}) \cdots \coc(i_2, i_1) \coc( i_1, x_0)=\gamma$.


Our first goal in this subsection is to prove the following proposition, which is essential to prove bounds on
 the supremum norm  $\|\hat{\mathcal{M}}_{ab, q} \tilde H\|_{\infty}$. 
\begin{prop} \label{arb1}\label{flatteningoutput} Let 
$$\mu:= \mu^{a, b}_{x, (i_{n_q}, \cdots, i_{m_q+1})}\quad$$
 and 
$$B = B^{a}_{i_{n_q}, \ldots ,i_{m_q + 1}}:=c e^{f^{(a)}_{r_q}(i_{n_q}, \ldots ,i_{m_q+1}, \omega)} e^{\eta_\theta},$$
with $c > 1$ the constant from Lemma \ref{a_0}, $\omega = \omega(i_{m_q + 1})$, and $\eta_\theta$ as in \eqref{defineetatheta}. There exist constants $\kappa>0, a_0 > 0, q_1 > 1, C>1$ such that for any $x \in \hat \Sigma^+$, for any   square-free $q > q_1$, $|a| < a_0, |b| \leq b_0$, and  for all $\phi \in E^q_q$,
\begin{eqnarray*} ||\mu*\phi||_2 &\leq& BCq^{-\kappa} ||\phi||_2 .\end{eqnarray*}
 The constants $C, \kappa$ may be chosen independent of $q, x, a, b$, and $i_{n_q} ,\ldots , i_{m_q + 1}$.  
\end{prop}
\begin{proof} The idea is to apply the $\ell^2$ flattening lemma \ref{bgsflatteninglemma}.  For ease of notation, we will fix $n = n_q$ and
$ r=r_q$ throughout this proof. We will assume that $a_0$ is small enough so that we may apply Lemma \ref{a_0}.

\noindent{\bf {Claim 1:}} We have the following bound
\begin{equation}   ||\mu||_1\le B. \label{claim1}\end{equation}

We first observe that 
$$f^{(a)}_n(i_n, \ldots, i_1, x)= f^{(a)}_r(i_n, \ldots , i_1, x)  +  f^{(a)}_{m}(i_{m} ,\ldots ,i_{1}, x) $$
and
\begin{align*} 
 &| f^{(a)}_r(i_n ,\ldots ,i_1, x)  -   f^{(a)}_r(i_n ,\ldots ,i_{m+ 1}, \omega(i_{m +1})) |
\\ & \le  \sum_{j = 0}^{r-1} | f^{(a)} (i_{n - j}  ,\ldots , i_{m + 1}, \omega(i_{m +1}))  - f^{(a)} (i_{n - j}  ,\ldots , i_1, x) |\\
&\leq  \sum_{j = 0}^{r-1} \theta^{r - 1 -j}  \Lip_{d_\theta} (f^{(a)}) \le  \eta_\theta . \end{align*}

Using the triangle inequality, we deduce
\begin{align*}&  f^{(a)}_r(i_n ,\ldots ,i_1, x) \\ &\le  f^{(a)}_r(i_n ,\ldots, i_{m+ 1}, \omega(i_{m +1})) +| f^{(a)}_r(i_n ,\ldots ,i_1, x)  -   f^{(a)}_r(i_n, \ldots , i_{m+ 1}, \omega(i_{m +1})) |
\\ &\le  f^{(a)}_r(i_n ,\ldots ,i_{m+ 1}, \omega(i_{m +1})) + \eta_\theta . \end{align*}
We therefore have
 \begin{eqnarray*} ||\mu||_1&\leq&  \sum_{i_1  ,\ldots ,i_{m}}e^{f^{(a)}_n(i_n, \ldots ,i_1, x)}\\
 &\leq&  \sum_{i_1, \ldots , i_{m}}e^{f^{(a)}_r(i_n ,\ldots ,i_{m + 1}, \omega(i_{m +1}))}e^{  f^{(a)}_{m}( i_{m},\ldots, i_1, x) }e^{\eta_\theta }\\
 &\leq& B.
 \end{eqnarray*}
 
 \noindent{\bf {Claim 2:}} For some $\kappa_1 > 0$, 
\be \label{lbgsest}   ||\mu||_\infty \leq q^{-\kappa_1}B .\ee
Using the bound on $m< n /d_0$, it suffices to bound 
 \begin{align*}& \left| \sum_{i_1, \ldots ,i_{m}, \coc_{m + 1 }(i_{m+1}, \ldots, i_1, x) = \gamma} e^{(f^{(a)}_n + ib\tau_n)(i_n, \ldots ,i_1, x)} \right| \\
 &\leq B \sum_{i_1, \ldots ,i_{m}, \coc_{m+1}(i_{m+1}, \ldots ,i_1, x) = \gamma} e^{f^{(a)} _{m}(i_{m}, \ldots ,i_1, x)} \\
 &\leq B e^{m( |a| \sup \tau  + |\log \lambda_a |)}  \sum_{i_1 ,\ldots ,i_{m}, \coc_{m+1}(i_{m+1}, \ldots ,i_1, x) = \gamma} e^{f^{(0)}_{m}(i_{m}, \ldots, i_1, x)} \\
  &\leq B c'e^{m( |a| \sup \tau  + |\log \lambda_a |) }  e^{-\delta m \inf (\tau) } \mbox{ (see Lemma \ref{sha})} \\
  &\leq B q^{-   \kappa_1 }  \end{align*}
  as long as we choose $ a_0$  so small  that
\[ \max_{|a| < a_0}  e^{m( ||\tau||_\infty|a| + |\log \lambda_a |)} \leq e^{\frac{1}{3} \delta  \inf(\tau)   },\]
and $\kappa_1 >0 $ is chosen such that (recalling $m \asymp \log q / d_0$)
\[ q^{\kappa_1} \leq  e^{\frac{1}{3} \delta m \inf(\tau)}\] 
 and $q > q_1 > (c')^{1/3\kappa_1}$. 
 
 \noindent{\bf{Claim 3:}} We have
 \begin{eqnarray}  |||\pi_{q'}(\mu)|||_\infty <B (q')^{- \kappa} \label{claim3} \end{eqnarray}
for all   $q' | q$ with $q' > q^{1/10}$ for some $\kappa > 0$.

  Choose such a $q'$, and let $\Gamma_0 < \Gamma$ be a subgroup
 such that the projection $\pi_p(\Gamma_0)$ is a proper subgroup of $ \operatorname{SL}_2(p)$ for each divisor $p | q'$. As in \cite{BGS} (see also Lemma 5.5 of \cite{MOW} for more details), we know that  $\#(a\G_0 \cap B_m(e))$ grows sub-exponentially in $m$, and in particular we have
\[ \#(a\Gamma_0 \cap \mathcal{B}_m(e))  = O(q^{\kappa_1 / 2}),\]
with $\kappa_1$ as in Claim (2). Claim (3) now follows from Claim (2).

By Claims (1) and  (3), we have now verified the conditions of the flattening lemma (Lemma \ref{bgsflatteninglemma}). We therefore apply it to obtain
\begin{eqnarray*} ||\mu*\phi||_2 &\leq& BCq^{-\kappa'} ||\phi||_2 \end{eqnarray*}
for $q$ large. 
\end{proof}
The following bound, which will be useful in a number of places, follows from direct calculation.
\begin{lem} \label{expound} There is $\tilde c > 0$ such that, for any $x, y \in \Sigma^+$ with $d_\theta(x, y) < 1$, any admissible sequence $(i_{n_q} ,\ldots ,i_1, x)$, and any $|a| < 1, |b| < b_0$, we have 
$$ \left| 1 - e^{(f^{(a)}_n + ib\tau_n)(i_n, \ldots, i_1, y) - (f^{(a)}_n + ib\tau_n)(i_n, \ldots, i_1, x)}\right| \leq \tilde c\cdot d_\theta(x, y) .$$
\end{lem}

\begin{remark}
It is in this type of bound that large values of $|b|$ cause problems. This characterizes the difference between Dolgopyat's approach and that of Bourgain-Gamburd-Sarnak. 
\end{remark}

Before moving on we will establish another proposition; this one will furnish Lipschitz bounds on $\hat{\mathcal{M}}_{ab, q} \tilde H$. For real numbers $a, b$, for $x,y  \in \hat \Sigma^+$ with $d_\theta(x, y) \leq \theta $, and a sequence $i_{n_q}, \ldots , i_{m_q + 1}$, we define
\begin{align}& \mu'= \mu'^{(a, b)}_{x, y, i_{m_q + 1}  , \ldots ,i_{n_q} }\\
&=   \sum_{i_1, \ldots , i_{m_q}}  \left( e^{(f^{(a)}_{n_q} + ib\tau_{n_q})(i_{n_q}, \ldots ,i_1, x)} - e^{(f^{(a)}_{n_q} + ib\tau_{n_q})(i_{n_q}, \ldots, i_1, y)} \right) \delta_{\coc_{m_q+1} (i_{m_q +1} ,\ldots , i_1, x)}. \label{defmu'}  \notag \end{align}

\begin{prop} \label{flatteningoutput2}\label{arb2} As usual, we write $n, m, r$ for $n_q, m_q, r_q$.  Let 
\be \label{definemuprime}  \mu' = \mu'^{(a, b)}_{x, y, i_{m + 1}, \ldots ,i_{n} }\ee
 and 
\begin{eqnarray*} B'&=&B'_{a, i_{m + 1} ,\ldots, i_n}\\
&:=&c\tilde c e^{\eta_\theta}e^{f^{(a)}_r(i_n, \ldots ,i_{m + 1}, \omega(i_{m +1}) ) } \end{eqnarray*}
with $c > 1$ the constant from Lemma \ref{a_0}. There exist constants $\kappa>0, a_0 > 0, q_1 > 1, C>1$ such that for any $x, y \in \hat \Sigma^+$ with $d_\theta(x, y) < 1$ and any  square-free $q > q_1$, $|a| < a_0, |b| \leq b_0$ 
\begin{eqnarray*} ||\mu'*\phi||_2 &\leq& B'C'q^{-\kappa} ||\phi||_2d_\theta(x, y) \end{eqnarray*}
for all $\phi \in E^q_q$.  The constants $C, \kappa$ may be chosen independent of $q, x, a, b$, and $i_{n}, \ldots , i_{m + 1}$.  
\end{prop}
\begin{proof}
\noindent \textbf{Claim 1:} A calculation similar to that for \eqref{claim1} yields
\begin{eqnarray*}  ||\mu'||_1 &\leq& B'd_{\theta}(x, y) .\end{eqnarray*}
\noindent \textbf{Claim 2:} There exists $\kappa>0$ with $ ||\mu'||_\infty < B' q^{-\kappa}$. For any $\g \in \SL_2(q)$, we estimate
\begin{align*}&  |\mu'|(\gamma) = \left| \sum_{i_1 ,\ldots , i_{m}: \coc_{m+1}(i_{m+1}, \ldots ,i_1, x) = \gamma}  \hspace{-10mm} e^{(f^{(a)}_{n} + ib\tau_{n})_n(i_n, \ldots ,i_1, x)} - e^{(f^{(a)}_{n} + ib\tau_{n})_n(i_n, \ldots ,i_1, y) }\right|\\
&\leq \hspace{-15mm} \sum_{i_1 ,\ldots ,i_{m}: \coc_{m+1}(i_{m+1}, \ldots, i_1, x) = \gamma} \hspace{-15mm}  e^{f^{(a)}_{n} (i_n, \ldots ,i_1, x) }   \left|1 - e^{(f^{(a)}_{n} +  ib\tau_n)(i_n, \ldots ,i_1, y) -(f^{(a)}_{n}  + ib\tau_n)(i_n, \ldots , i_1, x)} \right| \\
&\leq \tilde c d_\theta(x, y)\sum_{i_1, \ldots ,i_{m}: \coc_{m+1} (i_{m+1}, \ldots ,i_1, x) = \gamma} \hspace{-15mm}  e^{f^{(a)}_{n} (i_n, \ldots ,i_1, x)}  \mbox{ by Lemma \ref{expound}}  \end{align*}
The same argument as used in claim 2 of Proposition \ref{flatteningoutput} now yields
 \[ ||\mu'||_\infty < B'q^{-\kappa}\]
as expected.


\noindent \textbf{Claim 3:} An argument similar to the one leading to claim 3 in the proof of Proposition \ref{flatteningoutput} gives that 
\[ |||\pi_{q'}(\mu')|||_\infty \leq B'd_\theta(x, y)(q')^{-\kappa}  \mbox{ for } q' | q, q' > q^{1/10}.\] 
The proposition now follows from the $\ell^2$ flattening lemma as in the proof of Proposition \ref{flatteningoutput}

\end{proof}


 \subsection{Supremum bounds and Lipschitz bounds.} The purpose of all the estimates in the last two subsections is to provide bounds on the congruence transfer operators. We'll need to bound both the supremum norms and the Lipschitz constants. Start with the supremum norm. We observe that $(\hat{\mathcal{M}}^n_{ab, q}\tH)(x, g)$ has a good approximation by an appropriate sum of the convolutions, and use this fact, together with the convolution estimates in propositions \ref{flatteningoutput} and \ref{flatteningoutput2} to estimate the supremum and Lipschitz norms of $(\hat{\mathcal{M}}^n_{ab, q}\tH)$. 
 
For any $q$, for $\tilde H \in \tilde E_q^q$, and a sequence $i_{n_q}, \ldots , i_{m_q + 1}$, define the function $\phi$ on $\SL_2(q)$ by
\begin{align}\phi(g)& = \phi_{ \tilde H, (i_{n_q},\ldots ,i_{m_q+1})}(g)
\\& : = \tilde H(i_{n_q}, \ldots , i_{m_q+1},  \omega( i_{m_q + 1})), g \coc^{-1}_{r_q -1}(i_{n_q}, \ldots , i_{m_q + 1}, \omega( i_{m_q + 1}))).\label{definephi} \end{align}
 Note that 
$$|\phi| \leq ||\tilde H ||_\infty\quad\text{and }\quad \phi \in E^q_q .$$

\begin{lem} \label{notausefulname}  There exist  $\tilde C > 1$ and $ a_0 > 0$ such that the following holds for any $q, |a| < a_0, |b| \leq b_0, x \in \hat \Sigma^+$, and any $\tilde H \in \tilde E^q_q$:
\begin{eqnarray*}\left| (\hat{\mathcal{M}}^{n_q}_{ab, q}\tH)(x, \cdot ) - \sum_{i_{m_q+1}, \ldots , i_{n_q}}  \mu^{a, b}_{x, (i_{n_q}, \cdots, i_{m_q+1})} * \phi_{\tH, (i_{n_q},\cdots, i_{m_q+1})}(\cdot) \right|  \\ \le \tilde C \Lip_{d_\theta}(\tH) \theta^r.\end{eqnarray*}
\end{lem}

\begin{proof} Fix $q, x,$ and $\tilde H$ and  write $n = n_q, r = r_q, m = m_q$. Choose $|a| < a_0$, the constant from Lemma \ref{a_0}. For a sequence $i_n, \ldots , i_{m + 1}$, set
\begin{equation} \phi(g) = \phi_{ \tH, (i_n,\ldots ,  i_{m+1})}(g) . \end{equation}

Choose $|a| < a_0$, the constant from Lemma \ref{a_0}. We observe, as a consequence of the definitions and of Lemma \ref{cocdep},  that
\begin{eqnarray*}&& \sum_{i_{m+1} , \ldots , i_{n}}\mu^{a, b}_{x, (i_n, \cdots, i_{m+1})} * \phi_{ (i_n,\cdots, i_{m+1})}(g) \\&=&  \sum_{i_1 ,\ldots , i_n} e^{(f^{(a)}_n+ib\tau_n)(i_n ,\ldots , i_1, x)}\tH((i_n, \ldots , i_{m + 1}, \omega(i_{m+1})), g \coc^{-1}_n(i_n, \ldots ,i_1, x)). \end{eqnarray*}
We may therefore compute
\begin{eqnarray*}  && \left| (\hat{\mathcal{M}}^n_{ab, q}\tH)(x, \cdot) - \sum_{i_{m+1}, \ldots , i_{n}}  \mu_{x, (i_n, \cdots, i_{m+1})} * \phi_{\tH, (i_n,\cdots, i_{m+1})}(\cdot) \right|  \\
&\leq&  \sum_{i_1 ,\ldots, i_n}e^{f^{(a)}_n(i_n ,\ldots ,i_1, x)}\\
&&\left| \left( \tH((i_n ,\ldots, i_1, x), \cdot)   - \tH((i_n ,\ldots ,i_{m + 1},  \omega(i_{m+1})), \cdot) \right)   \right|\\
&\leq& \sum_{i_1, \ldots ,i_n}e^{f^{(a)}_n(i_n, \ldots ,i_1, x)} \Lip_{d_\theta}(\tH) \theta^{r-1}\\
&\leq&c   \Lip_{d_\theta}(\tH)  \theta^{r} \end{eqnarray*}
where $ c>1$ is the constant from Lemma \ref{a_0}. 
\end{proof}
The next lemma provides bounds on the supremum norm for $\hat{\mathcal{M}}^{n_q}_{ab, q}\tH$ using the description in terms of convolutions we just proved together with the convolution estimate, Proposition \ref{flatteningoutput}.
\begin{lem} \label{supbound} There exist constants $a_0 > 0$, $q_1 > 0$, $\kappa' > 0$ such that the following holds for any $|a| < a_0, |b| \leq b_0$, $q > q_1$ and $\tilde H \in \tilde E^q_q$:
$$||\hat{\mathcal{M}}^{n_q}_{ab, q}\tH ||_{\infty}\le  \frac{1}{2}  q^{-\kappa'} 
\left(||\tH||_\infty + \Lip_{d_\theta}(\tH) \theta^{n_q/2}\right).$$
\end{lem}
\begin{proof} We choose $a_0 >0$ small and $q_1$ large as on Lemma \ref{a_0} and Proposition \ref{flatteningoutput}. Consider $q > q_1, |a| < a_0, |b| \leq b_0,$ and a function $\tilde H \in \tilde E^q_q$. We write $n$ for $n_q$ and $r$ for $r_q$. We recall the function $\phi$ on $\SL_2(q)$  as in \eqref{definephi}.  Summing over all admissible sequences $i_{m+1}, \ldots ,i_n$  and applying Lemma \ref{notausefulname} and Proposition \ref{flatteningoutput} we obtain, for $x \in \hat \Sigma^+$,
\begin{eqnarray*}&& |\hat{\mathcal{M}}^n_{ab, q}\tH(x)| \\
&\leq& \left| \sum_{i_{m+1}, \ldots , i_{n}}  \mu^{a, b}_{x, (i_{n}, \cdots, i_{m+1})} * \phi_{ \tH, (i_{n},\cdots, i_{m+1})}(\cdot)\right| + \tilde C \Lip_{d_\theta}(\tH) \theta^r \\
&\leq& q^{-\kappa} C\sum_{i_{m+1}, \ldots ,i_{n}}  B^a_{i_{n}, \ldots, i_{m + 1}} | \phi_{ \tH, (i_{n},\cdots, i_{m+1})}|  +\tilde  C \Lip_{d_\theta}(\tH) \theta^r \\
&\leq& q^{-\kappa} cC e^{\eta_\theta} ||\tilde H||_\infty \sum_{i_{m+1}, \ldots , i_{n}}  e^{f^{(a)}_{r}(i_n , \ldots ,  i_{m + 1},  \omega(i_{m+1}))}  +\tilde C\Lip_{d_\theta}(\tH) \theta^r \\
&\leq& c^2Cq^{-\kappa}e^{\eta_\theta} ||\tilde H||_\infty + \tilde C \Lip_{d_\theta}(\tH) \theta^r \end{eqnarray*}
by Lemma \ref{a_0}. We may therefore choose $\kappa'>0$ and $q_1>1$ so that
$q^{-\kappa '} > 2c^2 Ce^{ \eta_\theta}q^{-\kappa}$ and $2\theta^{r_q - n_q/2}\tilde C < q^{-\kappa'}$ for all $q>q_1$
and hence obtain  
\begin{eqnarray*}  ||\hat{\mathcal{M}}^n_{ab, q}\tH ||_\infty  \leq  \frac{1}{2} q^{-\kappa'} 
\left( ||\tH||_\infty +  \Lip_{d_\theta}(\tH) \theta^{n/2} \right)\end{eqnarray*}
so long as $q>q_1$. \end{proof}

We'd like to iterate this argument, but before we can do that we need to estimate $\Lip_{d_\theta}(\hat{\mathcal{M}}^n_{ab, q}\tH)$. The proof of the next lemma is similar to the proof of the last one, though slightly longer. 
\begin{lem} \label{Lipbound} There exist $C> 0, q_1>0, \kappa' > 0, a_0 >0$ such that for all $|a| < a_0, |b| \leq  b_0,q > q_1$, and $\tilde H \in \tilde E^q_q$,  we have
\begin{equation} \Lip_{d_\theta}(\hat{\mathcal{M}}^{n_q}_{ab, q}\tH) \leq \frac{1}{2} q^{-\kappa'}\left( ||\tH||_\infty + \Lip_{d_\theta}(\tH) \theta^{n_q/2} \right)  .\end{equation}

\end{lem}
\begin{proof}  Again, we choose $a_0 >0$ small and $q_1$ large as on Lemma \ref{a_0} and propositions \ref{flatteningoutput}, \ref{flatteningoutput2}. Consider $q > q_1, |a| < a_0, |b| \leq b_0,$ and a function $\tilde H \in \tilde E^q_q$. We write $n$ for $n_q$ and $r$ for $r_q$. For $x, y \in \hat  \Sigma^+$ with $x_i = y_i $ for all $i \leq l$ (that is, with $d_\theta(x, y) \leq \theta^l < 1)$ we have
\begin{eqnarray*}&&  \left|\hat{\mathcal{M}}_{ab, q}^n\tH(x, g) - \hat{\mathcal{M}}_{ab, q}^n\tH(y, g)\right|\\&\leq& \sum_{i_1, \ldots , i_n} e^{f^{(a)}_n(i_n, \ldots , i_1, x)}\\
&&|\tH((i_n, \ldots ,i_1, x), g\coc_n^{-1}(i_n, \ldots , i_1, x)) - \tH((i_n, \ldots , i_1, y), g\coc_n^{-1}(i_n, \ldots  , i_1, y))|\\
&&+ \left| \left(\sum_{i_1, \ldots , i_n} e^{(f^{(a)}_n + ib\tau_n)(i_n, \ldots ,i_1, x) } - e^{(f^{(a)}_n + ib\tau_n)(i_n, \ldots ,  i_1, y) }\right) \right. \\
&&\left. \tH((i_n, \ldots,  i_1, y), g\coc_n^{-1}(i_n, \ldots , i_1, y)) \right|\\
&:=& W + V. \end{eqnarray*}
The first term $W$ is bounded as 
\begin{equation} \label{ubound} W \leq c \Lip_{d_\theta}(\tH) \theta^{n}d_\theta(x, y) \end{equation}
by Lemma \ref{a_0}. We estimate the other term as
\begin{eqnarray*} V &\leq&  \left| \sum_{i_1,\ldots , i_n}  \left( e^{(f^{(a)}_n + ib\tau_n)(i_n, \ldots , i_1, x)} - e^{(f^{(a)}_n + ib\tau_n)(i_n, \ldots , i_1, y) } \right) \right. \\
&&\left. \tH((i_n, \ldots , i_{m+1}, \omega(i_{m+1}) ), g\coc_n^{-1}(i_n, \ldots , i_{1}, x)) \right|\\
&&+ \theta^{r-1} \Lip_{d_\theta}(\tH)  \sum_{i_1, \ldots , i_n} \left| e^{(f^{(a)}_n + ib\tau_n)(i_n, \ldots , i_1, x)} - e^{(f^{(a)}_n + ib\tau_n)(i_n, \ldots , i_1, y)} \right| \\
&:=& L + K.\end{eqnarray*}
Next address $K$:
\begin{eqnarray*} K &=& \theta^{r-1} \Lip_{d_\theta} (\tilde H)\sum_{i_1, \ldots , i_n} \left| e^{(f^{(a)}_n + ib\tau_n)(i_n, \ldots , i_1, x)} - e^{(f^{(a)}_n + ib\tau_n)(i_n, \ldots , i_1, y)} \right|  \\
&\leq& \theta^{r-1} \Lip_{d_\theta}  (\tilde H) \sum_{i_1, \ldots ,i_n} e^{f^{(a)}_n(i_n, \ldots ,i_1, x)} \left|1 - e^{(f^{(a)}_n + ib\tau_n)(i_n, \ldots , i_1, y) - (f^{(a)}_n + ib\tau_n)(i_n, \ldots , i_1, x))} 
\right|\\
&\leq& \tilde c \theta^{ r-1}d_\theta(x, y)  \Lip_{d_\theta} (\tilde H) \sum_{i_1, \ldots , i_n} e^{f^{(a)}_n(i_n, \ldots ,i_1, x)}\end{eqnarray*}
by Lemma \ref{expound}. A final application of Lemma \ref{a_0} then gives
\begin{equation} K \leq c\tilde c\theta^{r-1}\Lip_{d_\theta} (\tilde H)d_\theta(x, y) \label{kbound}. \end{equation}

The bound on $L$ uses the $\ell^2$ flattening lemma once again. We observe that 
\[ L =\left|  \sum_{i_n , \ldots , i_{m+1}}   \mu'^{a, b} _{x, y, i_1, \ldots , i_{m} }*\phi_{\tH (i_n,\ldots , i_{m+1})} \right|\]
for $\mu', \phi$ as in \eqref{definemuprime},  \eqref{definephi} respectively. Proposition \ref{flatteningoutput2} then gives 

\begin{eqnarray*}|\mu'*\phi| \leq C'q^{-\kappa'}B'_{a, i_{m + 1} , \ldots , i_{n}} ||\tH||_\infty d_\theta(x, y),  \end{eqnarray*}
and summation over $i_n, \ldots , i_{m + 1}$ yields
\begin{eqnarray} \label{lbound}  L &\leq& C''  q^{-\kappa'} d_\theta(x, y)||\tH||_\infty \end{eqnarray}
for an appropriately chosen constant $C''$ (more precisely, $C'' = c^2 \tilde c C'e^{\eta_\theta} $ will do).  Putting together the equations \eqref{ubound}, \eqref{kbound}, \eqref{lbound}, we see that 
\begin{equation} \Lip_{d_\theta}(\hat{\mathcal{M}}^{n}_{ab, q}\tH) \leq \tilde C q^{-\kappa'}\left( ||\tH||_\infty +  \Lip_{d_\theta}(\tH) \theta^{n/2}  \right). \end{equation}
for an appropriate constant $\tilde C$. Now choose $\kappa'' = \kappa'/2$ and $ q_1$ large enough that $C'' q^{-\kappa''} < \frac{1}{2}$ for all $q > q_1$. Then
\begin{equation} \Lip_{d_\theta}(\hat{\mathcal{M}}^{n}_{ab, q}\tH) \leq \frac{1}{2} q^{-\kappa''}\left( ||\tH||_\infty +  \Lip_{d_\theta}(\tH)  \theta^{n/2} \right). \end{equation}
\end{proof}

\noindent{\bf Proof of Theorem \ref{af1} } Combining Lemmas \ref{supbound} and \ref{Lipbound}, we obtain that for some $\kappa'>0$,
\[ ||\hat{\mathcal{M}}^n_{ab, q} \tH||_\infty + \theta^{n/2} \Lip_{d_\theta}(\hat{\mathcal{M}}^n_{ab, q} \tH) \leq  q^{-\kappa'} (||\tH||_\infty +  \theta^{n/2}\Lip_{d_\theta}(\tH) )\] where $n=n_q$.
Iterating, we obtain that
 for any  $l \in \mathbb{N}$, 
\begin{align*}& ||\hat{\mathcal{M}}^{ln}_{ab, q} \tH||_\infty + \theta^{n/2} \Lip_{d_\theta}(\hat{\mathcal{M}}^{ln}_{ab. q} \tH) \\ &
\leq  
q^{-l\kappa'} (||\tH||_\infty +    \theta^{n/2} \Lip_{d_\theta}(\tH)). \end{align*}
It follows that 
$$\| \hat{\mathcal{M}}^{ln}_{ab, q} \tH\|_2 \le  ||\hat{\mathcal{M}}^{ln}_{ab, q} \tH||_\infty \le q^{-l\kappa'} ||\tH||_{d_\theta} $$
as desired.

\section{Uniform mixing of the BMS measure and the Haar measure} We assume that $\G <\SL_2(\z)$ is convex cocompact.  For each $q\in \N$, we denote by $m_{q}^{\operatorname{BMS}}$ the measure on $\Gamma(q) \backslash G$ induced by $\tilde m^{\BMS}$ and normalized so that
its total mass is $\# \SL_2(q)$.
 \subsection{Uniform exponential mixing}\label{defineprojections} 
 Our aim in this subsection is to prove Theorem \ref{BMS2} using Theorem \ref{spectralbound} on spectral bounds for the transfer operators. 
 Although this argument is similar to that contained \cite{Do} and \cite{Av},
  we shall include it in order to understand the dependence of the implied constants on the level $q$. 
    First we establish some more notation.
We fix $q$ such that $\G(q)\ba \G=\SL_2(q)$.
We recall the equivalence relation $(u, t) \sim (\sigma u, t - \tau(u))$ on $\Sigma \times \mathbb{R}$ and the suspension space
\[ \Sigma^\tau := \Sigma \times \mathbb{R}/\sim.  \]

\begin{definition} Similarly, we write \label{defineuqt}
\be  \hat U^{q, \tau} :=  \hat U \times \SL_2(q) \times \mathbb{R}^+/ (u, \g , t + \tau(u)) \sim (\hat\sigma(u), \g c(u), t).  \label{defineuqt} \ee
\end{definition}

For a function $\phi : \hat U^{q, \tau} \rightarrow \mathbb{C}$,
we say $\phi\in \mathcal{B}_0$ if $\|\phi\|_{\mathcal B_0}<\infty$ where
\begin{multline*}  ||\phi ||_{\mathcal{B}_0} := ||\phi||_\infty +\\
  \sup\{ \frac{|\phi(u, \g , s) - \phi(u', \g, s')|}{d(u, u') + |s - s'|} :
  u \neq u', \g \in \SL_2(q),  s \in [0, \tau(u)), s' \in [0, \tau(u'))\}.\end{multline*}
  
 We also say $\phi\in \mathcal{B}_1$ if $\|\phi\|_{\mathcal B_1}<\infty$ where
 \begin{multline*} ||\phi||_{\mathcal{B}_1} := ||\phi||_\infty +\sup\{ \mbox{Var}_{0, \tau(u)}( t\mapsto  \phi(u, \g, t)): u\in \hat U, \g\in \SL_2(q)\} .\end{multline*}


For  a bounded measurable function $\phi : \hat U^{q, \tau}  \rightarrow \mathbb{C}$, we define the function $\hat \phi_\xi$ on $\hat U \times \SL_2(q)$ by
\[ \hat \phi_\xi(u, \g) := \int_0^{\tau(u)} \phi(u, \g, t) e^{-\xi t} dt;\]
we will sometimes regard this as a vector valued function  on $\hat U$. 
The following lemma can be easily checked.
\begin{Lem} \label{bb0} If $\psi \in \mathcal{B}_0$ with $ \sum_{\g \in \G} \psi(u,  \gamma, s) = 0 $
for all $(u, s) \in \hat U^\tau$, then $\hat \psi_\xi \in  \mathcal{W}(\hat U, \c^{\SL_2(q)})$ when considered as a vector valued function.  \end{Lem}



For functions $\phi \in \mathcal{B}_1$ and $\psi \in\mathcal{B}_0$, we define the correlation function:
\be \label{phit} \tilde \rho_{\phi, \psi}(t) := \sum_{\g \in\SL_2(q)} \int_{\hat U}\int_0^{\tau(u)}  \phi(u,\g,  s + t) \psi(u,\g, s) ds d\nu(u) .\ee

In order to establish an exponential decay for $\tilde \rho_{\phi, \psi}(t)$ for a suitable class of functions $\phi, \psi$,  we consider its Laplace transform
and relate it with the transfer operators.
 We decompose $\tilde \rho_{\phi, \psi}(t)$ as 
\begin{eqnarray*}   \tilde \rho_{\phi, \psi}(t) &=&  \sum_{\g \in\SL_2(q)} \int_{\hat U}\int_{\max (0, \tau(u) - t)} ^{\tau(u)}  \phi(u, \g, s+ t) \psi(u, \g, s) ds d\nu(u)\\&& +   \sum_{\g \in\SL_2(q)} \int_{\hat U}\int_0^{\max(0, \tau(u) - t)}  \phi(u, \g, s+ t) \psi(u, \g, s ) ds d\nu(u)\\
&:=&\rho_{\phi, \psi}(t) + \bar \rho_{\phi, \psi}(t)  .\end{eqnarray*}
The reason for this decomposition is that
the Laplace transform of $\rho_{\phi, \psi}(t)$ can be expressed neatly in terms of transfer operators (see Lemma
\ref{Ltrfromtransoperators} below). More importantly, the Laplace transform of $  \rho_{\phi, \psi}$ has better decay properties than the Laplace transform of $ \tilde \rho_{\phi, \psi} $; this is needed when we apply the inverse Laplace transform at the end of the argument. Moreover,
since $ \tilde \rho_{\phi, \psi}(t)  = \rho_{\phi, \psi}(t)$ for all $t\ge \sup \tau$, 
the exponential decay of $\tilde{\rho}_{\phi, \psi}(t)$ follows from that of  $\rho_{\phi, \psi}(t)$.

So, consider the Laplace transform $\hat\rho$ of $\rho$: for $\xi\in \c$,
\[ \hat \rho_{\phi, \psi} (\xi) = \int_0^\infty e^{-\xi t } \rho_{\phi, \psi}(t) dt.\]

For the rest of the section, we shall use the notation
 \[ \xi = a - ib .\]
The first task is to write $\hat \rho(\xi)$ in terms of the transfer operators:
  \begin{lem} \label{Ltrfromtransoperators}
For $\phi \in \mathcal{B}_1$ and $\psi \in\mathcal{B}_0$ and $\Re(\xi) > 0$, we have
\begin{eqnarray*} \hat \rho_{\phi, \psi}(\xi) 
&=&  \sum_{k= 1}^\infty  \lambda_a^{k}   \int_{\hat U} \hat \phi_\xi(u) \cdot \hat{\mathcal{M}}_{ab, q}^k \hat\psi_{-\xi}(u ) d\nu(u)\end{eqnarray*}
where $\lambda_a$ is the lead eigenvalue of $\mathcal{L}_{-(\delta + a)\tau}$ as in Section 2. The right hand side should be understood as an inner product between two vectors in $\mathbb{C}^{\SL_2(q)}$. 
\end{lem}
\begin{proof} We calculate
\begin{align*} 
& \hat \rho_{\phi, \psi}(\xi) = \sum_{\g\in \SL_2(q)} \int_{\hat U}\int_{s = 0}^{\tau(u)}\int_{\tau(u) - s}^\infty e^{-\xi t} \phi(u, \g, s + t) \psi(u, \g, a) dtdsd\nu(u) \\
&= \sum_{\g\in \SL_2(q)}   \int_{\hat U}\int_{0}^{\tau(u)}\int_{\tau(u)}^\infty e^{-\xi (t - s)} \phi(u, \gamma, t) \psi(u, \g, s) dtdsd\nu(u) \\
&= \sum_{\g\in \SL_2(q)}  \sum_{k = 1}^\infty  \int_{\hat U}\int_{ 0}^{\tau(u)}\int_{\tau_k(u)}^{\tau_{k+1}(u)} e^{-\xi (t-s)} \phi(u, \g, t) \psi(u, \g, s) dtdsd\nu(u)
=\\& \sum_{\g\in \SL_2(q)}   \sum_{k = 1}^\infty  \int_{\hat U}\int_{ 0}^{\tau(u)}\int_{0}^{\tau(\hat \sigma ^k u)} e^{-\xi (t + \tau_k(u) -s)} \phi(\Pp ^k (u), \gamma c_k(u), t) \psi(u, \g, s) dtdsd\nu(u)\\
&= \sum_{\g\in \SL_2(q)}   \sum_{k = 1}^\infty  \int_{\hat U}e^{-\xi \tau_k(u)} \hat \phi_\xi (\Pp ^k (u), \g c_k(u) ) \hat \psi_{-\xi}(u, \g)d\nu(u)\\
&=  \sum_{k = 1}^\infty  \lambda_a^{k}  \int_{\hat U} \hat \phi_\xi(u)\cdot  \hat{\mathcal{M}}^k_{ab, q} \hat \psi_{-\xi}(u) d\nu(u)
\end{align*}
using the fact that $\hat{\mathcal{L}}_{00}^*(\nu)=\nu$. 
\end{proof}
\begin{lem}\label{l5.3}  If $\phi \in \mathcal{B}_1$, then $ ||\hat{\phi}_{\xi}||_2 \le
||\hat{\phi}_{\xi}||_\infty  \leq\frac{2 \sqrt{\#\SL_2(q)}e^{|a|\sup(\tau)} ||\phi||_{\mathcal{B}_1}}{\max(1, |b|)}$.
\end{lem}
\begin{proof} This follows from integration by parts in the flow direction. \end{proof}
\begin{lem} \label{l5.4}
If $\psi \in  \mathcal{B}_0$, then $||\hat \psi_{\xi}||_{1, b} \leq \frac{\sqrt{\#\SL_2(q)}e^{|a| \sup(\tau)}( 3\sup(\tau) + \Lip_d(\tau) ) ||\psi||_{\mathcal{B}_0}}{\max(1, |b|)}$. \end{lem}
\begin{proof} The trivial bound $\mbox{var}_{[0, \tau(u))} \psi(u, \gamma, \cdot) \leq  ||\psi||_{\mathcal{B}_0} \sup(\tau)$ provides 
\be \label{tr1} ||\hat \psi_{\xi}||_\infty \leq  \frac{2\sqrt{\#\SL_2(q)} e^{|a| \sup(\tau)}\sup(\tau)||\psi||_{\mathcal{B}_0}}{\max(1, |b|)}. \ee 
On the other hand consider any $u, u'\in \hat U$, $\g \in \SL_2(q)$, and suppose, without loss of generality, that $\tau(u') \geq \tau(u)$. Then
\begin{align*}& |\hat \psi_{\xi}(u, \g) -\hat \psi_{\xi}(u', \g)  | \\
 \leq & \int_0^{\tau(u)} |\psi(u, \g, t) - \psi(u', \g, t)|e^{|a|t} dt +  \int_{\tau(u)}^{\tau(u')} |\psi(u', \g, t)| e^{|a|t}dt   \\
&\leq d(u, u')  e^{|a| \sup(\tau) } \left( \sup(\tau)  ||\psi||_{\mathcal{B}_0}+ \Lip_d(\tau) ||\psi||_\infty  \right). \end{align*}
Together with \eqref{tr1}, this proves the claim.
\end{proof}
We will now use the spectral bounds (Theorem \ref{spectralbound}) to prove a rate of decay for the correlation functions.
\begin{prop} \label{ifbound} Let $a_0, q_0, q_0'$ be as in Theorem \ref{spectralbound}. There exist $C > 0, \eta > 0$ such that for all square free $q$ with
$(q, q_0q'_0)=1$, we have 
\[ |\tilde \rho_{\phi, \psi}(t) | \leq Cq^C ||\phi||_{\mathcal{B}_1} ||\psi||_{\mathcal{B}_0} e^{-\eta t}\]
 for all   $\phi \in \mathcal{B}_1$ and  $\psi \in \mathcal{B}_0$  satisfying $ \sum_{\g \in \SL_2(q)} \psi( u , \g, s) = 0.$
 \end{prop}
\begin{proof} We will establish that the Laplace transform $\hat \rho_{\phi, \psi}$ extends to an appropriate half plane and then apply the inversion formula. Lemma \ref{Ltrfromtransoperators} gives
\[ \hat \rho_{\phi, \psi}(\xi) = \sum_{k= 0}^\infty \lambda_a^{k} \int_{\hat U} \hat \phi_\xi(u) \cdot \hat{\mathcal{M}}_{ab, q}^k \hat\psi_{-\xi}(u) d\nu(u), \]
for $\Re(\xi)>0$. We claim an analytic continuation of $\hat \rho_{\phi, \psi}(\xi)$ to $\Re(\xi)> -a_0$ for some $a_0>0$.
Each term of the above infinite sum is analytic, so it suffices to check that
the sum is absolutely convergent.  For $|a| \le \min(1, a_0)$, Theorem \ref{spectralbound}, together with Lemma \ref{bb0}, gives that for some $\e>0$,
\begin{eqnarray*}  \lambda_a^{k}
\int_{\hat U} \hat \phi_\xi(u)\cdot  \hat{\mathcal{M}}_{ab, q}^k \hat\psi_{-\xi}(u) d\nu &\leq& \lambda_a^{k} || \hat{\mathcal{M}}_{ab, q}^k \hat\psi_{-\xi} ||_{2} ||\phi_\xi ||_{2} \\
&\leq& \lambda_a^{k} Cq^C e^{-\epsilon k}||\hat \psi_{-\xi}||_{1, b} ||\phi_\xi  ||_{2}\\
&\leq&\lambda_a^{k}  \frac{C'q^{C'} e^{-\epsilon k}}{\max(1, |b|)^2} ||\psi||_{\mathcal{B}_0}  ||\phi||_{\mathcal{B}_1}, \end{eqnarray*}
where $C'$ is given by Lemmas \ref{l5.3} and \ref{l5.4}; this is clearly summable so long as we choose $a_0$ small enough that 
\[ \max_{|a| \leq a_0}  \lambda_a  \leq e^{\e/2}. \] 
This computation also gives that for some absolute constant $C_1>0$,
$$|\hat \rho_{\phi, \psi}(\xi)| \le \frac{C_1 q^{C'}}{1+|b|^2}||\psi||_{\mathcal{B}_0}  ||\phi||_{\mathcal{B}_1} $$
for all $\xi$ with $|\Re(\xi)|<a_0$.
 Now $\rho_{\phi, \psi}(t)$ is Lipschitz, so we may apply the inverse Laplace transform formula \cite[Chapter II, Theorem 7.3]{Wi} and obtain for all $t>0$,
\be \label{finall}\rho_{\phi, \psi}(t) = e^{-\tfrac{a_0}2 t} \lim_{T\to \infty}\int_{-\tfrac{a_0}2-iT}^{-\tfrac{a_0}2+iT} \hat \rho_{\phi, \psi}(-\tfrac{a_0}2-ib ) e^{-ibt} db.\ee

Since $\int_{-\tfrac{a_0}2-iT}^{-\tfrac{a_0}2+iT}| \hat \rho_{\phi, \psi}(-\tfrac{a_0}2-ib ) | db \ll  q^{C'}\int_{0}^T \frac{1}{1+b^2} db<\infty$, the limit in the right hand side of \eqref{finall}
is $O(q^{C'})$ with the implied constant independent of $t$,
yielding the result for a uniform constant $C>0$ with $\rho_{\phi, \psi}$ in place of $\tilde \rho_{\phi, \psi}$. Since those two functions agree on $ t > \sup (\tau)$, and since $\bar \rho_{\phi, \psi} $ is bounded as $q^C ||\psi||_{\mathcal{B}_0} ||\phi||_{\mathcal{B}_1}$, the result follows. \end{proof}


We can convert a function $\phi$ on $\G(q) \ba G$ to give a function $\phi_t$ on $\hat U^{q, \tau}$ as follows: for $t > 0$, $u \in \hat U_i, 0\leq s\leq \tau(u)$ and $\gamma \in \G(q) \ba \G$, we set $\tilde u$ to be the lift of $u$ to $\tilde U$, and 
\be \label{definephithere} \phi_t( u, \g, s) := \int_{\tilde S_i} \phi(\g [\tilde u, \tilde y]a_{t + s}) d\nu_u(\tilde y) \ee
where $\nu_u$ is the probability measure on $\tilde S_i$ conditioned from the measure $\nu$ at $u$. 
For a general $s>0$, 
we define  \be \label{pfff} \phi_t (u, \g, s) := \phi_t(\hat \sigma^k (u), \gamma \coc_k(u),  s - \tau_k( u)) \ee
where $k\in \N$ is such that  $0\le s - \tau_k( u)\le \tau(\hat \sigma^k(u))$.
By the equivalence relation \eqref{defineuqt}, this defines $\phi_t$ on all of $\hat{U}\times \SL_2(q)\times \br_{\ge 0}$.
\begin{lem} There exists $C > 0$ such that, for any $\tilde y \in \tilde S_i$, we have
\[ |\phi(\gamma [\tilde u, \tilde y]a_{2t + s}) - \phi_t(u,\g,  s+t)| \leq Ce^{- t} ||\phi||_{C^1}. \]
\end{lem}
\begin{proof} Let $u, \tilde u, \tilde y \in \tilde S_i$ be as above. Choose $k\in \N$ such that $0 \leq t + s - \tau_k(u) \leq \tau(\sigma^ku)$, and write $ u' = \hat \sigma^k u \in U_j$. Set $\tilde u'$ to be the lift of $u'$ to $\tilde U_j$.  If $y' \in S_j$ with lift $\tilde y'\in \tilde S_j$,  then the definition of the cocycle $\coc$ tells us that 
 both $\gamma[\tilde u, \tilde y]a_{\tau_k(u)}$ and 
 $ \g \coc_k(u) [\tilde u', \tilde y'] $ lie in the stable leaf of $\gamma \coc_k(u) \tilde R_j \subset G$. It follows that  for some $c_1>0$,
 \[ d( \gamma[\tilde u, \tilde y]a_{2t + s}, \g \coc_k(u) [\tilde u', \tilde y']a_{2t + s - \tau_k(u)} ) \leq c_1  e^{-(2t + s - \tau_k(u))} \leq  c_1 e^{- t },\]
 and so that
 \[ |\phi(\gamma[\tilde u, \tilde y]a_{2t + s}) - \phi( \g \coc_k(u) [\tilde u', \tilde y']a_{2t + s - \tau_k(u)}) | \leq c_1 ||\phi||_{C^1}  e^{-t }.\] 
 Integrating this inequality over $\tilde y' \in S_j$ and using \eqref{pfff}
 which gives $$\int_{\tilde S_j} \phi( \g \coc_k(u) [\tilde u', \tilde y']a_{2t + s - \tau_k(u)})d\nu_u= \phi_t(u,\gamma,s+t)$$
  this gives the required result. 
\end{proof}
We therefore have the following lemma (cf. \cite[Lemma 8.2]{Av})


\begin{lem} \label{approxbyU} There are constants  $\eta > 0, C>0$ independent of $\psi, \phi, q$ such that
 \[ \left| \int_{\G(q) \ba G} (\phi\circ a_{2t}) \cdot \psi \; dm_q^{\BMS} -\frac{  \tilde \rho_{ \phi_t, \psi_0}(t)}{\nu(\tau)}  \right| < C\cdot \#\SL_2(q)\cdot  ||\phi||_{C^1}||\psi||_{\infty}e^{-\eta t}\]
for all $\phi, \psi \in C^1(\G(q) \backslash G)$. 
\end{lem}

\begin{proof}[Proof of Theorem \ref{t5.1} ]  We assume that $m^{\BMS}(\G \ba G) = 1$ without loss of generality, so that the total mass of $m_q^{\BMS} $ is equal to $ \# \SL_2(q)$. Fix $q$ with $(q, q_0 q_0') = 1$ and compactly supported functions $\psi, \phi \in C^1(\Gamma(q) \backslash G)$.

We write \[ \psi = \psi' + \psi'' \]
where $\psi'$ is (left) $\G$ invariant, and $\psi''$ satisfies $\sum_{\g \in \G(q) \ba \G} \psi''(\g x) = 0$ for all $x \in \G(q)\ba G$. 
Exponential mixing of $\psi'$ (with constant independent of $q$) follows from the bounds established in section \ref{bunbounded} together with the complex
RPF theorem, as was carried out in the work of Dolgopyat and  Stoyanov \cite{St}. 
So we can and shall assume that $\psi = \psi''$, so that
$ \sum_{\g \in \G(q) \ba \G} \psi(\gamma x) = 0 \mbox{ for all $x\in \G(q)\ba G$.} \label{orthcond} $

We consider the functions $\phi_t, \psi_0$ as defined in \eqref{definephithere}; note that $
\psi_0$ satisfies $\sum_{\g \in \G} \psi_0(u,  \gamma, s) = 0$ and that
$ ||\psi_0||_{\mathcal{B}_0} \ll ||\psi||_{C^1}.$
 We also need to bound $||\phi_t||_{\mathcal{B}_1}$. It's clear that $\sup|\phi_t| \leq ||\phi||_{C^1}$. On the other hand, we know that, for fixed $u$, and $s$ such that $(u, s)$ is not of the form $(u', 0), \phi_t(u, s)$ is differentiable in the flow direction with derivative bounded by $||\phi||_{C^1}$. On the other hand there are at most $\frac{\sup \tau}{\inf \tau} + 1$ values of $s$ such that $(u, s) \sim (u', 0)$. Each of these may be a discontinuity, but each jump is at most $2||\phi||_{C^1}$. We can therefore bound the variation as 
\[ \mbox{var}_{[0, \tau(u))} ( s\rightarrow \phi_t(u, \g, s)) \leq  \left(\tau(u) + 2\left(  \frac{\sup\tau}{\inf\tau} + 1\right)\right) ||\phi||_{C^1}.\]
In other words
\[ ||\phi_t||_{\mathcal{B}_1} \ll ||\phi||_{C^1}.\]
Now calculate; for any $t>0$
\begin{align*} & \left | \int_{\G(q) \backslash G} \phi(ga_{2t}) \psi(g) dm^{\BMS} \right|
\\ &\leq\left| \frac{  \tilde \rho_{\phi_t,\psi_0}(t)}{\nu(\tau)} \right|   
+ C(\#\SL_2(q))    ||\phi||_{C^1}||\psi||_{C^0}e^{-\eta t}    \\ 
&\leq  C' q^{C'} \left( ||\psi_0||_{\mathcal{B}_0} ||\phi_t||_{\mathcal{B}_1}  + ||\phi||_{C^1}||\psi||_{C^0}\right) e^{-\eta' t}   \\
&\leq  C'' q^{C''}  ||\phi||_{C^1}||\psi||_{C^0}e^{-\eta' t}   \end{align*} for some $C', C'', \eta' >0$,
by Lemma \ref{approxbyU} and Proposition \ref{ifbound}. 
\end{proof}


\subsection{Exponential decay of the matrix coefficients}
 

Let $\G$ be a geometrically finite subgroup of $\PSL_2(\br)$.
We begin by recalling the definitions of measures $m^{\BR}$, $m^{\BR_*}$ and $m^{\Haar}$.
Similar to the definition of the BMS measure 
$$d \tilde m^{\BMS}(u)= e^{\delta \beta_{u^+}(o, u)} e^{\delta \beta_{u^-}(o, u) }
d\mu_{o}^{\PS} (u^+) d\mu_{o}^{\PS} (u^-) ds
$$ given in section 2,
the measures
$ \tilde m^{\BR} =\tilde m^{\BR}_\G$, $\tilde m^{\BR_*}=\tilde m^{\BR_*}_\G$  and $\tilde m^{\Haar}$ on $\PSL_2(\br)$ are defined 
as follows:
\begin{align*}
 d \tilde m^{\BR}(u)&= e^{\beta_{u^+}(o, u)}\;
 e^{\delta \beta_{u^-}(o, u) }\; dm_o(u^+)  d\mu_{o}^{\PS} (u^-) ds ; \\
 d \tilde m^{\BR_*}(u)&= e^{\delta \beta_{u^+}(o, u)}\;
 e^{ \beta_{u^-}(o, u) }\; dm_o(u^-) d\mu_{o}^{\PS} (u^+) ds ;\\
 d \tilde m^{\Haar}(u)&= e^{\beta_{u^+}(o, u)}\;
 e^{ \beta_{u^-}(o, u) }\; dm_o(u^+) dm_o(u^-) ds
 \end{align*}
 where $m_o$ is the unique probability measure on $\partial (\bH^2)$ which is invariant under the stabilizer of $o$.
 
These measures are all left $\G$-invariant and
induce measures on $\G\ba G$, which we will denote by $ m^{\BR}, m^{\BR_*}, m^{\Haar}$ respectively.

 Let $$N=\{n_s:=\begin{pmatrix} 1 & 0 \\ s & 1\end{pmatrix}: s\in \br\}\quad \text{and}\quad H=
  \{h_s:=\begin{pmatrix} 1 & s \\ 0 & 1\end{pmatrix}: s\in \br\} .$$


For $g\in G$, denote by $g^\pm$ the forward and backward end points of the geodesic determined by $g$
and set
$$\alpha(g,{\Lambda(\G)}):=\inf\{|s|: (g n_s)^+\in \Lambda(\G)\} +\inf \{|s|: (g h_s)^-\in \Lambda(\G)\} + 1.$$

It follows from the continuity of the visual map that
 for any compact subset $\mathcal Q\subset G$, 
$$\alpha (\mathcal Q,\Lambda(\G)):=\sup \alpha (g, \Lambda(\G))< \infty .$$

If $\G'$ is a normal subgroup of $\G$ of finite index,
then $\Lambda(\G)=\Lambda(\G')$, and hence
 $\alpha (\mathcal Q,\Lambda(\G))=\alpha(\mathcal Q, \Lambda(\G'))$. Therefore 
the following theorem implies that Theorem \ref{main2} can be deduced from Theorem \ref{BMS2};
note that even though we need the following theorem only for $\G$ convex cocompact in this paper, 
we record it for a general geometrically finite group $\G$ of $G$ for future reference.
Let $\pi: G\to \Gamma \ba G$ be the canonical projection.
\begin{thm}\label{de} Let $\mathcal Q\subset G$ be a compact subset. 
Suppose that there exist constants $c_\G>0$ and $\eta_\G>0$ such that
for any $\Psi,  \Phi\in C^1(\G\ba G)$ supported on $\pi(\mathcal Q)$,
\begin{multline} \label{bms3}  \int_{\G \ba G} \Psi (ga_t) \Phi (g) dm^{\BMS} =
 \tfrac{m^{\BMS}(\Psi) \cdot m^{\BMS}(\Phi)}{m^{\BMS}(\G\ba G)}  +O( c_\G \cdot ||\Psi||_{C^1} ||\Phi||_{C^1}
 \cdot e^{-\eta_\G t})\end{multline}
where the implied constant  depends only on $\mathcal Q$.
Then for any $\Psi,  \Phi\in C^1(\G\ba G)$ supported on $\pi(\mathcal Q)$, as $t\to +\infty$,
\begin{multline} \label{bms3}  e^{(1-\delta)t} \int_{\G \ba G} \Psi (ga_t) \Phi (g) dm^{\Haar}\\ =
 \tfrac{ m^{\BR}(\Psi) \cdot m^{\BR_*}(\Phi)}{m^{\BMS}(\G\ba G)} +O( c_\G \cdot ||\Psi||_{C^1} ||\Phi||_{C^1}
 \cdot e^{-\eta_\G' t}) \end{multline}
where $\eta_\G'= \tfrac{\eta_\G }{8 + 2\eta_\G}$ and  the implied constant  depends only on $\mathcal Q$ and $\alpha (\mathcal Q,\Lambda(\G))$.
\end{thm} 
 
 The rest of this section is devoted to the proof of this theorem.
The proof involves effectivizing the original argument of Roblin \cite{R_T}, extended in \cite{Sch}, \cite{OS}, \cite{MO},
while making the dependence of the implied constant
 on the relevant functions precise.

 For $\e>0$ and a subset $S$ of $G$, $S_\e$ denotes the set
$\{s\in S: d(s, e)\le \e\}$.

Let $$P:=H A.$$
 Then the sets $B_\e:= P_\e N_\e$, $\e>0$ form a basis of neighborhoods of $e$ in $G$.

For $g\in \PSL_2(\br)$, we define measures on $gN$:
\begin{align*} d\tilde \mu^{\Leb}_{gN}(gn)&=e^{\beta_{(gn)^+}(o, gn)} dm_o(gn^+) ; \\
 d\tilde \mu^{\PS}_{gN}(gn)& =e^{\delta \beta_{(gn)^+}(o, gn)} d\mPS{o}( gn^+). \end{align*}

If $x=[g]\in \G\ba G$, for a compact subset $N_0$ of $N$  such that $gN_0$ injects to $\G\ba G$, and
for a function $\psi$ on $xN_{0}$, 
we write $d\mu^{\Leb}_{xN}(\psi)$ and $d\mu^{\PS}_{xN}(\psi)$ for
 the push-forward of the above measures to $xN_0$
via the isomorphism $gN_0$ with $xN_0$.
The measure $ d\tilde \mu^{\Leb}_{gN}(gn)$ is simply the Haar measure on $N$, and hence we 
write $dn=d\tilde \mu^{\Leb}_{gN}(gn)$.


The quasi-product structure of $\tilde m^{\BMS}$ is a key ingredient in the arguments below:
for $\Psi\in C_c(G)$ supported on $gB_{\e}$ for all $\e>0$ small,
$$\tilde m^{\BMS}(\Psi)=\int_{gP_{\e}}\int_{gpN_{\e}} \Psi(gpn) d\tilde \mu_{gpN}^{\PS}(gpn) d\nu_{gP}(gp)$$
where $d\nu_{gP}(gp)=e^{\delta \beta_{(gp)^-}(o, gp)} d\mPS{o}(gp^-) ds$ for $s=\beta_{gp^-}(o, gp)$.

In the rest of this section, we  fix a compact subset $\mathcal Q$ of $G$, and assume that the
hypotheses of Theorem \ref{de} are satisfied for functions supported in $\pi(\mathcal Q)$.
Let $2\e_0>0$ be
the injectivity radius of $\pi(\mathcal Q)$. 
Fix $x=[g]\in \pi(\mathcal Q)$  and functions
$\Psi, \Phi \in C^1(\G\ba G)$ which are supported in $xB_{\e_0/2}$.

\begin{prop} \label{imp1} Fix $y\in xP_{\e_0}$ and put $\phi:=\Phi|_{yN_{\e_0}}\in C^1(yN_{\e_0})$. 
Then for $t>1$,
$$\int_{yN_{\e_0}} \Psi(yna_t) \phi(yn) d\mu_{yN}^{\PS}(yn)
=\frac{\mu_{yN}^{\PS}(\phi)}{|m^{\BMS}|} m^{\BMS}(\Psi)+O(c_\G \|\Psi\|_{C^1}\|\phi\|_{C^1} e^{-\eta_1 t})$$
where $\eta_1=\eta_\G/ (4+\eta_\G)$ and the implied constant depends only on $\mathcal Q$ and $\alpha(\mathcal Q,\Lambda(\G))$.
\end{prop}
\begin{proof} 
Set $R_0:=\alpha(y,\Lambda(\G)) +2$.
For a sufficiently small $\e\in (0, 1)$,
if we set $t_0:= \log (R_0 \e^{-1})$,
 $y_0=ya_{t_0}$,  then $\nu(y_0P_\e)>0$.  Hence we may choose a smooth positive function $\rho_\e$ supported on $y_0P_\e$
 such that $\nu(\rho_\e )=1$  and that $\|\rho_\e\|_{C^1}\ll \e^{-3}$.
  Define a $C^1$-function $ \Phi^\dag$ supported on $y_0P_\e N_{\e_0}$ as follows:
 $$ \Phi^\dag(y_0 p n):=e^{-\delta t_0 -\delta \beta_{n_p^+}(n_p, p 
n)} {\phi(y_0n_p a_{-t_0}) \rho_{\e} (y_0p)}$$ where  $n_p\in N$ is the unique element such that
 $p^{-1} n_p \in n P$.
We have $m^{\BMS}(\Phi^\dag)=\mu_{yN}^{\PS}(\phi) $.
Now by the hypothesis of Theorem \ref{de}, we have
\begin{align*}
&\int_{yN_{\e_0}} \Psi(yna_t) \phi(yn) d\mu_{yN}^{\PS}(yn) =(1+O(\e)) \la a_{t-t_0} \Psi, \Phi^\dag\ra_{m^{\BMS}}\\
&= (1+O(\e))\left( \frac{\mu_{yN}^{\PS}(\phi) }{|m^{\BMS}|} m^{\BMS}(\Psi)+O(c_\G \e^{-3}e^{-\eta (t-t_0)})\right)\\
&= \frac{\mu_{yN}^{\PS}(\phi) }{|m^{\BMS}|} m^{\BMS}(\Psi)+O(\e + c_\G R_0^\eta \e^{-\eta-3} e^{-\eta t})
\end{align*}
where the implied constant depends only on the $C^1$-norms of $\Psi$ and $\phi$ and $\mathcal Q$.
By taking $\e=e^{-\eta t /(4+\eta)}$ and by setting $\eta_1:={\eta  /(4+\eta)}$,
we obtain
$$\int_{yN_{\e_0}} \Psi(yna_t) \phi(yn) d\mu_{yN}^{\PS}(yn) =
\frac{\mu_{yN}^{\PS}(\phi) }{|m^{\BMS}|} m^{\BMS}(\Psi) +O(c_\G R_0^\eta e^{-\eta_1 t}).$$
Since $R_0$ is bounded above by  $\alpha(\mathcal Q,\Lambda(\G))$,
 this proves the claim.
\end{proof}

\begin{prop} \label{imp2} Keeping the same notation as in Proposition \ref{imp1},
we have
$$e^{(1-\delta)t}\int_{yN_{\e_0}} \Psi(yna_t) \phi(yn) dn
=\frac{\mu_{yN}^{\PS}(\phi)}{|m^{\BMS}|} m^{\BR}(\Psi)+O(c_\G \|\Psi\|_{C^1}\|\phi\|_{C^1} e^{-\eta_1 t/2})$$
where the implied constant depends only on $\mathcal Q$ and $\alpha(\mathcal Q,\Lambda(\G))$.
\end{prop}
\begin{proof}
We deduce this proposition from Proposition \ref{imp1} by comparing the two integrals on the left hand sides
via transversal intersections.


Define $\phi_{\e}^{\pm}\in C^1(yN)$
by
\be\label{dp}\phi_{\e}^{+}(yn)=\sup_{n'\in N_\e} \phi(ynn') \quad \text{and}\quad
\phi_{\e}^{-}(yn)=\inf_{n'\in N_\e} \phi(ynn') .\ee

Fix $R_1:=\alpha (\mathcal Q, \Lambda(\G))+1$. For each $p\in P_{\e_0}$, let
 $N_p:=\{n\in N: (pn)^+=n_s^+ \text{ for some $|s|<R_1$}\}$; then
$\mu^{\PS}_{xpN}(xpN_p)  >0$, and
the map $xp \mapsto \mu^{\PS}_{xpN}(xpN_p)$ is a positive smooth function
on $xP_{\e_0}$. Set $B_{\e_0}':=\cup_{p\in P_{\e_0}} pN_p$; we may assume that
the map $g\to xg$ is injective on $B_{\e_0}'$ by replacing $\e_0$ by a smaller number if necessary.

Define the finite set
$$P_x(t):=\{p\in P_{\e_0}: xpn\in \supp(\phi) a_t  \text{ for some $n\in N_p$} \}.$$

 Define functions $\psi$ and $\Psi'$ supported on $xP_{\e_0}$ and $xB_{\e_0}'$ respectively:
 $$\psi (xp):=\int_{xpN_{\e_0}} {\Psi}(xpn)  dn\text{ and }
\Psi'(xpn):=\tfrac{\psi(xp) }{\mu^{\PS}_{xpN}(xpN_p)} \text{ for $pn\in B'_{\e_0}$}. $$
We then have $ m^{\BMS}(\Psi')=\nu_{xP}(\psi)=m^{\BR}(\Psi)$, and 
we can find \\$C^1$-approximations $\Psi'_{\e,-} \le \Psi'\le \Psi'_{\e, +}$ such that
$m^{\BMS}(\Psi'_{\e,\pm}) = m^{\BMS}( \Psi' )+O(\e)$, and   $\|\Psi'_{\e, \pm}\|_{C^1} =O( \e^{-1} \|\Psi \|_{C^1})$.
 The  following computation holds
for all small $0<\e \ll \e_0$:
\begin{align*}
&e^{ (1-\delta) t} \int_{yN} \Psi(yna_t)\phi(yn) dn \\&
=
(1+O(\e))e^{-\delta t} \sum_{p\in P_x(t)} \psi (xp) {\phi_{ce^{-t}\e_0}^\pm}(xpa_{-t}) \\ &=
(1+O(\e)) 
 \int_{yN} \Psi' (yna_t) \phi_{c' (\e_0+R_1)e^{-t} }^\pm (yn)d\mu_{yN}^{\PS}(yn)
\\&= (1+O(\e)) 
 \int_{yN} \Psi'_{\e, \pm} (yna_t) \phi_{c' (\e_0+R_1)e^{-t} }^\pm (yn)d\mu_{yN}^{\PS}(yn)
\\ &=(1+O(\e)+ O((\e_0+ R_1)e^{-t} )) \left( \tfrac{m^{\BR}(\Psi)\mu^{\PS}_{yN} (\phi)}{|m^{\BMS}|} 
+ O(c_\G\e^{-1}  \|\Psi\|_{C^1} \|\phi\|_{C^1} e^{-\eta_1 t}) \right)
 \end{align*}
by Proposition \ref{imp1} (we refer \cite{OS} and \cite{MO} for details in this step).

Therefore taking $\e=e^{-{\eta_1 t} /2}$,
$$e^{ (1-\delta) t} \int_{yN} \Psi(yna_t)\phi(yn) dn =
 \tfrac{m^{\BR}(\Psi)\mu^{\PS}_{yN} (\phi)}{|m^{\BMS}|} + O(c_\G \|\Psi\|_{C^1} \|\phi\|_{C^1} e^{-{\eta_1 t} /2})
$$
where the implied constant depends only on $\e_0$ and $R_1$, and hence only on $\mathcal Q$
and $\alpha(\mathcal Q,\Lambda(\G))$.

\end{proof}

In order to finish the proof of Theorem \ref{de},
we first observe that
by the partition of unity argument, it suffices to prove the claim for $\Phi$ and $\Psi$
supported on $xB_{\e_0/2}$ for $x\in \mathcal Q$. 
We note that $dm^{\Haar}(pn)=dp dn$ where $dp$ is a left Haar measure on $P$, and hence
 $$\int_{\G\ba G} \Psi(xa_t) \Phi(x) dm^{\Haar}(x)=\int_{xp\in zP_{\e_0}} \int_{xpN_{\e_0}}
 \Psi(xpna_t) \Phi(xp n) dn dp.$$

Hence applying Propositions \ref{imp1} and \ref{imp2} for each $y=xp\in xP_{\e_0}$,
we deduce that
 \begin{align*}&e^{ (1-\delta) t}  \int_{\G\ba G} \Psi(xa_t) \Phi(x) dm^{\Haar}(x)\\
 &=\int_{xp\in xP_{\e_0}} \left( \tfrac{m^{\BR}(\Psi)\mu^{\PS}_{xpN} (\Phi|_{xpN_{\e_0}})}{|m^{\BMS}|} 
+ O(c_\G \|\Psi\|_{C^1} \|\Phi|_{xpN_{\e_0}} \|_{C^1} e^{-\eta_1 t/2}) \right) dp
\\&= \tfrac{m^{\BR}(\Psi)m^{\BR_*}(\Phi)}{|m^{\BMS}|} + O(c_\G \|\Psi\|_{C^1} \|\Phi \|_{C^1} e^{-\eta_1 t/2}) 
\end{align*}
 where the implied constant depends only on $\mathcal Q$ and $\alpha (\mathcal Q,\Lambda(\G))$ .
 This finishes the proof.

\section{Zero-free region of the  Selberg zeta functions} \label{resolvent_c}

Let $\G<\SL_2(\z)$ be as in Theorem \ref{main2}. 
In \cite{MMO, MO}, it was shown that Theorem \ref{main2} implies the following:
\begin{Thm}\label{Clo}
There exist $C'>0$ and $\e_0>0$ such that
for all square free $q\ge 1$ with $(q, q_0)=1$,
\begin{enumerate}\item 
\begin{align*} \label{ppq} \mathcal P_q(T) &:=\#\{C: \text{primitive closed geodesic in $\G(q)\ba \PSL_2(\br)$ with } \ell(C)< T\}\\ &=
\operatorname{li} {(e^{\delta T})} +O(q^{C'} e^{(\delta -\e_0)T})\end{align*}
where $\operatorname{li}(x)=\int_2^x \frac{dx}{\log x}$ and $\ell(C)$ is the length of $C$;
\item for any $z, w\in \bH^2$, \begin{align*} N_q(T; z, w) &:= \#\lbrace \g \in \G(q) : d(z, \gamma w) \leq T \rbrace\\
& = C_q(z, w)e ^{\delta T} + O(q^{C'} e^{(\delta -\e_0)T}) \end{align*}
for some constant $C_q(z, w)>0$.
\end{enumerate}
\end{Thm} 


\begin{proof}[Proof of Theorem \ref{resolvent}]
We use the well-known relation between the Poincar\'e series and
the leading term for the resolvent of the Laplacian $R_q(s) = (\Delta_q - s(1-s))^{-1}$. More precisely there is a decomposition, valid on $\Re(s) > \delta$, of the resolvent as
\begin{equation} R_q(s)  = f(s) P_q (s) + K_q(s) \label{decompresolve}\end{equation}
where $P_q(s)$ is the integral operator with kernel
\[ P_q(s, z, w) := \sum_{\g \in \G(q)} e^{-d(z, \g w)s} =  s \int_0^\infty e^{-st} N_q(t; z, w)dt ,\]
where $K_q$(s) is holomorphic on $\Re(s) > \delta - 1$, and $f$ is a ratio of Gamma functions holomorphic on $\Re(s)  > 0$ (see \cite[Proposition 2.2]{GN} and its proof). Applying the estimates on $N_q(t; z, w)$ from Theorem \ref{Clo} (2), we see that the right hand side of \eqref{decompresolve} has analytic extension to the half plane $\Re(s) > \delta - \epsilon_0$ (with $\epsilon_0$ as in Theorem \ref{Clo})  
except for a simple pole at $s=\delta$.
\end{proof}

  \footnotesize

  Hee ~Oh, \textsc{Mathematics department, Yale university, New Haven, CT 06511 and Korea Institute for Advanced Study, Seoul, Korea }\par\nopagebreak
  \textit{E-mail address:} \texttt{hee.oh@yale.edu}

  \medskip

  Dale ~Winter, \textsc{Department of mathematics, Brown university, Providence, RI 02906 }\par\nopagebreak
  \textit{E-mail address:} \texttt{dale\_winter@brown.edu}

\end{document}